\documentclass[11pt]{article}
\usepackage[margin=2.5cm]{geometry}
\usepackage[utf8]{inputenc}
\usepackage[T1]{fontenc}
\usepackage{lmodern}
\usepackage{microtype}
\usepackage{amsmath,amssymb,amsthm,mathtools,bm}
\usepackage{enumitem}
\usepackage{booktabs}
\usepackage{array}
\usepackage{tabularx}
\usepackage{xspace}
\usepackage{multicol}
\usepackage{capt-of}
\usepackage{xcolor}

\usepackage[authoryear,round]{natbib}

\usepackage[colorlinks=true,linkcolor=blue,citecolor=blue,urlcolor=blue]{hyperref}
\hypersetup{pdftitle={SDDP has polynomial iteration complexity in the horizon},pdfauthor={Vincent Lecl\`ere}}
\usepackage[capitalize,nameinlink]{cleveref}

\makeatletter\def\input@path{{paper/}{./}}\makeatother

\theoremstyle{plain}
\newtheorem{theorem}{Theorem}[section]
\newtheorem{proposition}[theorem]{Proposition}
\newtheorem{lemma}[theorem]{Lemma}
\newtheorem{corollary}[theorem]{Corollary}

\theoremstyle{definition}
\newtheorem{definition}[theorem]{Definition}
\newtheorem{example}[theorem]{Example}
\newtheorem{assumption}{Assumption}

\theoremstyle{remark}
\newtheorem{remark}[theorem]{Remark}
\crefname{assumption}{Assumption}{Assumptions}
\crefname{conjecture}{Conjecture}{Conjectures}
\crefname{example}{Example}{Examples}

\newcommand{\R}{\mathbb{R}}
\newcommand{\EE}{\mathbb{E}}
\newcommand{\PP}{\mathbb{P}}
\newcommand{\cB}{\mathcal{B}}
\newcommand{\Bd}{\dot{\mathcal{B}}}
\newcommand{\cX}{\mathcal{X}}
\newcommand{\bxi}{\bm{\xi}}
\newcommand{\bx}{\bm{x}} \newcommand{\be}{\bm{e}} \newcommand{\bgamma}{\bm{\gamma}} \newcommand{\bGamma}{\bm{\Gamma}}
\newcommand{\cF}{\mathcal{F}} \newcommand{\cA}{\mathcal{A}}

\newcommand{\cP}{\mathcal{P}}
\newcommand{\cW}{\mathcal{W}}
\newcommand{\cBh}{\widehat{\mathcal B}}
\newcommand{\BR}{\cB^{\mathrm R}}
\newcommand{\Vh}{\widehat V}
\newcommand{\Ph}{\widehat{\PP}}
\newcommand{\lv}[2]{\underline V^{#2}_{#1}}
\newcommand{\uV}[1]{\underline V^{#1}}
\newcommand{\Ub}[1]{\bar U_{#1}}
\newcommand{\avg}{\operatorname*{avg}}
\newcommand{\diam}{\operatorname{diam}}
\newcommand{\vol}{\operatorname{vol}}
\newcommand{\aff}{\operatorname{aff}}
\newcommand{\dist}{\operatorname{dist}}
\newcommand{\Lip}{\operatorname{Lip}}
\newcommand{\cb}{\bar\ell}
\newcommand{\one}{\mathbf 1}
\newcommand{\Wa}{\mathsf W_1}
\newcommand{\TV}{\mathrm{TV}}
\newcommand{\Lh}{\hat L}
\newcommand{\Mb}{\bar M}
\newcommand{\Oe}{X_\eta}
\newcommand{\erel}{\varepsilon_{\rm rel}}

\newcommand{\eg}{e.g.\@\xspace}

\title{SDDP has polynomial iteration complexity in the horizon}
\author{Vincent Lecl\`ere\thanks{CERMICS, \'Ecole nationale des ponts et chauss\'ees, Marne-la-Vall\'ee, France. \texttt{vincent.leclere@enpc.fr}}}
\date{}

\begin{document}
\maketitle

\begin{abstract}
Stochastic dual dynamic programming (SDDP) and its variants are run in practice with forward passes
that sample the noise at random. Recent iteration-complexity results favour another choice: when the
trial points are selected deterministically, for instance along the scenario with the largest
gap, the number of iterations is polynomial in the horizon $T$, whereas every bound we know of for
randomly sampled forward passes is exponential in $T$. We show that this is not intrinsic to random sampling, for a broad class of
trajectory-following dynamic programming algorithms, which includes SDDP, stochastic dual dynamic integer
programming and stochastic Lipschitz dynamic programming. With random sampling, and with Lipschitz constants, diameters and
per-stage loss ranges bounded independently of the stage and of the horizon, the expected number of iterations whose policy is not $T\varepsilon$-optimal is $O(T\varepsilon^{-(d+1)})$, where $d$ is the state dimension: linear in $T$, as for deterministic
selection rules. The proof needs no
particular scenario to be sampled: the gap between the expected cost of the current
policy and the current lower bound equals the expected sum of one-step Bellman residuals along the sampled path,
and Lipschitz cuts that are tight at the visited states turn the positive residuals into permanent increases of
a bounded potential. The results extend to noise without finite
support, to backward passes on a sample average approximation, and to stationary and periodic infinite-horizon SDDP with geometric pass lengths.


\end{abstract}

\section{Introduction}\label{sec:intro}

Trajectory-following dynamic programming (TFDP) \citep{forcier2023trajectory} solves multistage stochastic
programs with stagewise-independent noise by alternating two passes. A forward pass simulates a trajectory of
the policy induced by the current (lower) approximations of the cost-to-go functions; a backward pass adds, at
each visited state, a cut, affine or not, that tightens the approximation there. Stochastic dual dynamic
programming (SDDP) \citep{pereira1991multi}, the most widely used instance, has affine Benders cuts and is run
in practice with forward passes that sample the noise at random. Almost sure convergence of these methods is
classical \citep{philpott2008convergence,shapiro2011analysis,girardeau2015convergence,zou2019sddip}. See also the survey of
\citet{lan2023numerical} and the review of \citet{fullner2025sddp}.

\paragraph{Known bounds.} For deterministic selection rules of the trial points, such as the ``problem child'' rule that sends
the forward pass to the scenario with the largest current gap, the number of iterations is
polynomial in the horizon $T$ for a fixed state dimension $d$
\citep{lan2022complexity,zhang2022stochastic,forcier2023trajectory,julan2023infinite}. For randomly
sampled forward passes, every bound we know of is exponential in~$T$
\citep{lan2022correction,zhang2022stochastic,forcier2023trajectory,julan2023infinite}, including the
dimension-free bounds of \citet{barros2026bidirectional} under strong convexity.
The reason is common to all these proofs. They count \emph{saturated} points, where the true gap
$V_t-\lv tk$ is certified small, an argument introduced by \citet{lan2022complexity} and extended to nonconvex
problems by \citet{zhang2022stochastic} through Lipschitz regularisation. But saturation at stage $t$ follows
from saturation at stage $t+1$ only at the point reached by the same sampled path. Each proof therefore bounds from
below the probability that the sampled path follows a given chain of scenarios, for instance by $\prod_t1/N_t$
under uniform sampling on $N_t$ atoms, which is exponentially small in~$T$. \citeauthor{zhang2022stochastic}
ask, in their Remark~3, whether a bound polynomial in $T$ exists in general for the random-sampling version.

\begin{table}[t]
\centering\small
\setlength{\tabcolsep}{4pt}
\begin{tabular}{>{\raggedright\arraybackslash}p{3.9cm}>{\raggedright\arraybackslash}p{2.4cm}>{\raggedright\arraybackslash}p{6.1cm}>{\raggedright\arraybackslash}p{2.6cm}}
\toprule
\multicolumn{2}{@{}l}{setting} & expected number of bad iterations & reference\\
\midrule
\multicolumn{2}{@{}l}{general oracle, volume argument} & $(d+1)\cb\,(T+1)(4LD)^d/\varepsilon^{d+1}$ & \cref{cor:uniform}\\
\multicolumn{2}{@{}l}{general oracle, covering (no convexity)} & $Z\cb\,(T+1)(1+8LD/\varepsilon)^d/\varepsilon$ & \cref{cor:cover}\\
\multicolumn{2}{@{}l}{Zhang--Sun, Algorithm~3} & $\cb\,(T+1)(1+8LD/\varepsilon)^d/\varepsilon$ & \cref{thm:ZS}\\
\multicolumn{2}{@{}l}{SDDiP, binary states} & $2^{n-1}\cb\,(T+1)/\varepsilon$ & \cref{prop:SDDiP}\\
\multicolumn{2}{@{}>{\raggedright\arraybackslash}p{6.3cm}}{discounted, any $T$; infinite horizon, geometric length (accuracy $\varepsilon$)} & $2(d+1)\cb(4\Lh D)^d/((1-\beta)^{d+2}\varepsilon^{d+1})$ & \cref{cor:disc,thm:geom}\\
\bottomrule
\end{tabular}
\caption{An iteration $k$ is \emph{bad} if $G^k>T\varepsilon$ ($G^k>\varepsilon$ in the discounted row). For Zhang
and Sun's algorithm, the row bounds the index of the first good iteration. $Z$, $n$, $\beta$ and $\Lh$ are defined
with the cited results.}\label{tab:summary}
\end{table}

\paragraph{Main result.} We prove bounds polynomial in $T$ for TFDP-type algorithms with randomly
sampled forward passes. Let $\pi^k$ be the policy of iteration $k$, greedy for the approximations of
iteration $k-1$, and $G^k:=V^{\pi^k}_0(x_0)-\lv0{k-1}(x_0)$ the gap between its expected cost and the lower bound.
Assume that, at every stage, the set of states has dimension at most $d$ and diameter at
most $D$, that the cuts are $L$-Lipschitz on it, and that the initial lower bounds are within $\cb$ per
remaining stage of the cost-to-go functions (\eg when the stage losses lie in an interval of width
$\cb$, \cref{sec:rates}).
Then the expected number of iterations with $G^k>T\varepsilon$, an error of $\varepsilon$ per stage, is
$O(T\varepsilon^{-(d+1)})$, with a constant that depends only on $d$, $D$, $L$ and $\cb$ (\cref{cor:uniform}). In the discounted case this number does not depend on
the horizon (\cref{cor:disc}). All constants are explicit (\cref{tab:summary}).

For the same accuracy, and with each state set in a ball of diameter $D$ (a convention under which our bound also
holds), the deterministic problem-child rule needs at most $T(1+2LD/\varepsilon)^d$ iterations
\citep{zhang2022stochastic}, so random sampling costs a factor of about $(d+1)2^d\cb/\varepsilon$, independent of
$T$ (\cref{sec:compare-sddp}). Linearity in $T$ needs Lipschitz constants bounded independently of the stage and
of the horizon, as in the inventory model of \cref{ex:absorb}; under extended relatively complete recourse, the Lipschitz
constants of Benders cuts are only known to grow like $T-t$ (\cref{prop:benders}).
As for all random-sampling results, $G^k$ is not observed: the bounds count bad iterations
and do not provide a stopping test (\cref{rem:compare}).

\paragraph{Proof outline.} The proof does not use saturation. It follows the analysis of real-time dynamic
programming (RTDP) \citep{barto1995rtdp} in reinforcement learning: the SDDP forward pass is exactly the
sampling scheme of RTDP, which follows greedy trajectories, drawn from the true transition law, of a policy
built on an optimistic value estimate (here a lower bound on costs). Let $\lv tk$ be the approximation of the
cost-to-go function $V_t$ at time $t$ after $k$ iterations, $\cB_t$ the Bellman operator between times $t$ and
$t+1$, $(\bx^k_t)_{0\le t\le T}$, with $\bx^k_0=x_0$, the states visited by the forward pass of iteration $k$, and
$\cA_{k-1}$ the history of the first $k-1$ iterations. Besides the gap, two quantities are tracked: the one-step
Bellman residual $r^k_t:=\cB_t(\lv{t+1}{k-1})-\lv t{k-1}$, which is local in time, and the volume
$\Phi^k:=\sum_{t<T}\int(\lv tk-\lv t0)\,d\mu_t$ gained by the approximations, for a reference probability
measure $\mu_t$ on the states at time $t$ (the normalised Lebesgue measure when the state set is
convex). With exact cuts and an exact forward pass, three facts suffice.
\begin{enumerate}[label=(\roman*),nosep,leftmargin=*]
\item \emph{Gap identity} (\cref{lem:gap}), as in \citet[Lemma~2]{efroni2019tight} for RTDP. The gap equals the
expected sum of the residuals along the path that the algorithm samples, so no particular path needs to be
sampled:
\begin{equation}\label{eq:intro:gap_identity}
G^k=\EE\Big[\sum_{t<T}r^k_t(\bx^k_t)\Bigm|\cA_{k-1}\Big].
\end{equation}
\item \emph{Lift} (\cref{lem:lift}). The backward pass lifts $\lv tk$ by its residual at the trial point;
the approximations being Lipschitz, this pointwise lift is a lift on a ball, hence of the integral:
\begin{equation}\label{eq:intro:volume_lift}
\Phi^k-\Phi^{k-1}\ \ge\ \Psi_T\Big(\sum_{t<T}r^k_t(\bx^k_t)\Big)
\end{equation}
for an explicit convex function $\Psi_T$.
\item \emph{Bounded potential}. The approximations are nondecreasing and stay below the value functions, so
$0\le\Phi^k\le\sum_{t<T}M_t$ with $M_t\le\sup(V_t-\lv t0)$.
\end{enumerate}
Taking expectations in~(ii), using~(i) and Jensen's inequality, and summing over $k$ with~(iii) gives
\begin{equation}\label{eq:intro:potential_bound}
\EE\sum_{k\ge1}\Psi_T(G^k)\ \le\ \sum_{t<T}M_t
\end{equation}
(\cref{thm:master}), from which all the rates follow. In metric state spaces, \citet[App.~F]{domingues2021kernel}
and \citet[Lemma~7.3]{sinclair2020adaptive} already accumulate Lipschitz value updates of this kind over a
covering and over adaptive partitions. We transfer this residual accounting to a general cut oracle:
approximate tightness supplies the local lift, while a fixed upper bound on the cuts bounds the potential. To our knowledge,
no earlier paper applies the RTDP analysis to SDDP or to cutting-plane methods.

\paragraph{Organisation.} \Cref{sec:setting} sets out the problem, the algorithms and the cut oracle, whose cuts
need only be Lipschitz, approximately tight at the trial point and below a fixed \emph{ceiling}, on state sets that
may be nonconvex, mixed-integer or finite. \Cref{sec:main} proves the master theorem and derives explicit rates.
\Cref{sec:instances} applies it to Benders cuts, to the generalised conjugacy cuts of Zhang and Sun, and to
stochastic dual dynamic integer programming (SDDiP) on binary states. \Cref{sec:noise} treats noise without finite
support: valid inexact cuts, backward passes on a sample average approximation (SAA), and sampled cuts.
\Cref{sec:infinite} extends the analysis to stationary and periodic infinite-horizon problems and gives, for SDDP
under extended relatively complete recourse, bounds polynomial in the effective horizon $1/(1-\beta)$. \Cref{sec:discussion} compares the results with the SDDP
and reinforcement-learning literature, and \cref{app:notation} collects the notation.

\section{Setting}\label{sec:setting}

This section fixes the problem, the algorithms and the accuracy measure, the gap $G^k$, and states the two
assumptions of the analysis, on the cut oracle (\cref{sec:oracle}) and on the state sets (\cref{sec:VG}).

\subsection{Problem}\label{sec:problem}

We consider a multistage stochastic program with stagewise-independent noise. Stage $t\in[T]:=\{t\in\mathbb N:1\le t\le T\}$ starts at time $t-1$ in a state $x_{t-1}\in X_{t-1}$, observes
the noise $\bxi_t$, with law $\PP_t$ on a measurable space $\Xi_t$ and independent across stages, chooses the
next state $x_t\in\cX_t(x_{t-1},\bxi_t)\subset X_t\subset\R^{n_t}$, and pays the loss
$\ell_t(x_{t-1},x_t,\bxi_t)$; the initial state is fixed, $X_0=\{x_0\}$. We write $\bxi_{[t]}:=(\bxi_s)_{s\in[t]}$
for the noise history up to stage $t$. The loss $\ell_t$ is
Borel and real-valued (infinite losses are encoded in the constraint set $\cX_t$). For $t<T$, a bounded
Borel function $W$ on $X_{t+1}$ and a probability law $\PP$ on $\Xi_{t+1}$, let
\begin{equation}\label{eq:setting:bellman_operator}
\Bd_t(W;x,\xi):=\inf_{y\in \cX_{t+1}(x,\xi)}\ell_{t+1}(x,y,\xi)+W(y),\quad
\cB^\PP_t(W)(x):=\int \Bd_t(W;x,\xi)\,\PP(d\xi),\quad x\in X_t .
\end{equation}
We write $\cB_t:=\cB^{\PP_{t+1}}_t$ for the Bellman operator of the true law, and set
$V_T:=0$ and $V_t:=\cB_t(V_{t+1})$ for $t<T$ (backward in $t$). The operator $\cB_t$ is monotone ($W\le W'$ on
$X_{t+1}$ implies $\cB_t(W)\le\cB_t(W')$ on $X_t$) and commutes with constants.
The following mild conditions make
the stage problems and the cost of a policy well defined.

\begin{assumption}[Standing assumptions]\label{ass:S}
For every $t\in[T]$:
\begin{enumerate}[label=(S\arabic*),ref=(S\arabic*),nosep]
\item\label{S1} $X_{t-1}$ is compact and, for all $x\in X_{t-1}$ and $\xi\in\Xi_t$,
$\emptyset\ne\cX_t(x,\xi)\subset X_t$ (forward invariance); $X_T$ need not be compact ($V_T=0$ on it).
\item\label{S2} (Integrability.)
$\sup_{x\in X_{t-1}}\int\big|\inf_{y\in \cX_t(x,\xi)}\ell_t(x,y,\xi)\big|\,\PP_t(d\xi)<\infty$.
\item\label{S3} (Measurability.) Measurable selections of the stage problems exist, the approximations
of \cref{sec:algo} are Borel functions of finitely many random parameters (for instance cut coefficients), and all
oracles are Borel functions of their inputs.
\end{enumerate}
\end{assumption}

Condition~\ref{S2} implies that $\cB_t(0)$ is bounded on $X_t$ for $t<T$, hence
$|\cB_t(W)|\le\sup_{X_t}|\cB_t(0)|+\sup|W|$.
Condition~\ref{S3} is automatic for finitely supported noise and
deterministic oracles with a tie-breaking rule; \cref{app:measurability} gives the precise statement and
a sufficient condition (a normal integrand and approximations continuous in the state).

\subsection{Randomised TFDP-type algorithms}\label{sec:algo}

The algorithm keeps approximations $\lv tk:X_t\to\R$ of the value functions, $0\le t\le T$, with
$\lv Tk:=0$, where $k$ counts the completed iterations: quantities carrying the superscript $k$ are produced
during iteration $k$, which uses the approximations of index $k-1$. The approximations need not lie below $V_t$ (\cref{sec:oracle}).

A \emph{forward operator} $\cF^k_{t-1}$ maps a bounded Borel function $W$ on $X_t$ to a selection
$\cF^k_{t-1}(W)(x,\xi)\in\cX_t(x,\xi)$ that is greedy for $W$ up to a \emph{forward error}
$\bgamma^{F,k}_t\ge0$: for all $x$ and $\xi$ with $\Bd_{t-1}(W;x,\xi)>-\infty$,
\begin{equation}\label{eq:setting:forward_error}
\ell_t\big(x,\cF^k_{t-1}(W)(x,\xi),\xi\big)+W\big(\cF^k_{t-1}(W)(x,\xi)\big)\ \le\ \Bd_{t-1}(W;x,\xi)+\bgamma^{F,k}_t .
\end{equation}
We write $\gamma^F=0$ for exact greedy selections ($\bgamma^{F,k}_t=0$ for all $k$ and $t$). The policy
$\pi^k$ of iteration $k$ is greedy for the approximations of the previous iteration: at stage $t$, in state
$x$ and under noise $\xi$, it moves to $\cF^k_{t-1}(\lv t{k-1})(x,\xi)$.

Let $\cA_k$ be the $\sigma$-field generated by everything drawn during the first $k$ iterations, including
any internal randomness of the oracles, which must be independent of the noise of the current forward pass.
The forward operators and the forward errors of iteration $k$ are $\cA_{k-1}$-measurable
(\cref{app:measurability}), so $\pi^k$ is known before iteration $k$ starts. Iteration $k\ge1$ runs two passes.
\begin{itemize}[nosep]
\item \emph{Forward pass.} Draw $(\bxi^k_t)_{t\in[T]}$ independently, with laws $(\PP_t)_{t\in[T]}$ and
independently of $\cA_{k-1}$, and simulate $\pi^k$ from $\bx^k_0:=x_0$, with 
$\bx^k_t:=\cF^k_{t-1}(\lv t{k-1})(\bx^k_{t-1},\bxi^k_t)$ for $t\in[T]$.
\item \emph{Backward pass.} For $t$ from $T-1$ down to $0$, an oracle returns a function $f^k_t:X_t\to\R$, the
\emph{cut}, at the trial point $\bx^k_t$, possibly using the already updated $\lv{t+1}{k}$. Then
$\lv t{k}:=\max(\lv t{k-1},f^k_t)$, possibly with further cuts, see~\ref{O4}.
\end{itemize}
SDDP is the case of finitely supported noise, exact greedy passes and Benders cuts. The general case is TFDP with
random node selection from the law of $\bxi_t$, whose upper approximations we do not use.

\paragraph{Cost, gap and residuals.} Since the forward pass of iteration $k$ is a trajectory of $\pi^k$
drawn independently of $\cA_{k-1}$, the expected cost of $\pi^k$ and the \emph{gap} between this cost and the lower bound are
\begin{equation}\label{eq:setting:expected_policy_cost}
V^{\pi^k}_0(x_0):=\EE\Big[\sum_{t\in[T]}\ell_t(\bx^k_{t-1},\bx^k_t,\bxi^k_t)\ \Big|\ \cA_{k-1}\Big],
\qquad G^k:=V^{\pi^k}_0(x_0)-\lv0{k-1}(x_0) .
\end{equation}
The \emph{residuals} $r^k_t$ measure how far the approximations are from satisfying the Bellman equation; we also
use their sum along the sampled path:
\begin{equation}\label{eq:setting:residuals}
r^k_t:=\cB_t(\lv{t+1}{k-1})-\lv t{k-1}\quad\text{on }X_t\ (t<T),\qquad
R^k:=\sum_{t<T}r^k_t(\bx^k_t).
\end{equation}
Residuals are measured with the true operator $\cB_t$, whatever the oracle computes. They may be negative,
since the sub-Bellman inequality $\lv t{k-1}\le\cB_t(\lv{t+1}{k-1})$ is not assumed; they are nonnegative under
the conditions of \cref{lem:valid}(a).

\paragraph{What the gap measures.}\label{sec:accuracy} All results are stated in terms of the gap $G^k$ defined in~\eqref{eq:setting:expected_policy_cost}. It
cannot be observed, but it controls both the policy and the lower bound. Let
$\zeta_0:=(U_0(x_0)-V_0(x_0))_+$, where $U_0$ is the ceiling of \cref{sec:oracle}; $\zeta_0=0$ in the
standard case $U=V$. By
\cref{lem:gap} below, for every $k\ge1$,
\begin{equation}\label{eq:setting:gap_meaning}
G^k\ \ge\ V_0(x_0)-\lv0{k-1}(x_0)\qquad\text{and}\qquad G^k\ \ge\ V^{\pi^k}_0(x_0)-V_0(x_0)-\zeta_0 .
\end{equation}
So, when $\zeta_0=0$, $G^k\le\theta$ means that the whole policy $\pi^k$ (not only its first
decision) is $\theta$-optimal and that the lower bound $\lv0{k-1}(x_0)$ is $\theta$-tight.

\begin{remark}[Comparison of accuracy notions]\label{rem:compare}
With $\zeta_0=0$, $G^k\le\theta$ implies the guarantees of \citet[eqs.~(5.18)--(5.19)]{lan2022complexity}
(a $\theta$-optimal first-stage solution, since $V^{\pi^k}_0(x_0)\ge\EE[\ell_1(x_0,\bx^k_1,\bxi^k_1)+V_1(\bx^k_1)]$),
of \citet[Def.~1]{zhang2022stochastic} (a $\theta$-optimal root-node solution) and of
\citet[Thms.~4.2 and~4.11]{forcier2023trajectory} (a $\theta$-solution and a $\theta$-tight lower bound).
These works bound the first index at which their guarantee holds. That index is at most the first
index with $G^k\le\theta$, hence at most $1+\#\{k\ge1:G^k>\theta\}$, which our bounds control. Such bounds are of
uniform-PAC type \citep{dann2017unifying}. As in all earlier
results on random sampling, identifying a good iteration requires Monte Carlo estimates of $V^{\pi^k}_0(x_0)$, unless one returns a
randomly selected iterate (\cref{cor:anytime}(d)).
\end{remark}

\subsection{The cut oracle}\label{sec:oracle}

In the standard TFDP case, each cut is \emph{valid}, that is, below $V_t$, and \emph{tight} at the trial point.
\Cref{ass:O} weakens both: tightness is required only up to an error $\gamma$, and with respect to the approximation of the
previous iteration~\ref{O2}; validity is replaced by a \emph{ceiling}~\ref{O3}, a function $U_t$ above all cuts. A ceiling other
than $V_t$ is needed when the forward pass is greedy for one problem while the cuts lie below the value functions of another: a
regularised and the original problem in \cref{sec:ZS}, the true law and an SAA in \cref{sec:saa}.

\begin{assumption}[Lipschitz cut oracle]\label{ass:O}
For each $t<T$ there are a metric $\dist_t$ on $X_t$, possibly extended (with values in
$[0,+\infty]$), a constant $L_t\ge0$ and a \emph{ceiling} $U_t:X_t\to(-\infty,+\infty]$
such that, for all $k\ge1$:
\begin{enumerate}[label=(O\arabic*),ref=(O\arabic*),nosep]
\item\label{O1} (Lipschitz) $\lv t0$ and every cut of index $t$ are Borel and
$L_t$-Lipschitz for $\dist_t$, and $\lv t0$
is deterministic and bounded.
\item\label{O2} (tightness at the trial point, up to $\gamma$) At $x=\bx^k_t$,
$f^k_t(x)\ge\cB_t(\lv{t+1}{k-1})(x)-\bgamma^k_t$, where the error
$\bgamma^k_t\ge0$ is integrable and $\cA_k$-measurable.
\item\label{O3} (ceiling) $\lv t0\le U_t$ and every cut of index $t$ is $\le U_t$ on
$X_t$. Moreover the \emph{potential bound} $M_t:=\sup_{X_t}(\Ub t-\lv t0)$ is finite, where
$\Ub t(x):=\inf_{z\in X_t}\{U_t(z)+L_t\dist_t(x,z)\}$, with $L_t\cdot\infty:=\infty$.
\item\label{O4} (monotone update) $\lv t{k}=\max\big(\{\lv t{k-1},f^k_t\}\cup\{f'_i\}_{i\in[m]}\big)$ for some $m\ge0$,
where the additional cuts $f'_i$, if any, satisfy~\ref{O1} and~\ref{O3}.
\end{enumerate}
\end{assumption}

The function $\Ub t$, the \emph{Lipschitz envelope} of the ceiling, is the largest $L_t$-Lipschitz function below $U_t$; if $\dist_t$ is finite and $U_t(z_0)<\infty$,
then $M_t\le U_t(z_0)+L_t\diam_{\dist_t}(X_t)-\inf\lv t0$, and if $U_t=V_t$ is $L_t$-Lipschitz, then
$M_t=\sup(V_t-\lv t0)$. The usual validity and tightness conditions imply \cref{ass:O}:

\begin{lemma}[Sufficient conditions]\label{lem:valid}
Assume~\ref{S1}--\ref{S2}, and let each $\lv t0$ be bounded.
\begin{enumerate}[label=(\alph*),nosep]
\item (Strong validity implies~\ref{O3}.) Assume $\lv t0\le V_t$ for all $t$, and that every cut
created at index $t<T$ during iteration $k$ satisfies $f\le\cB_t(\lv{t+1}{k})$ on
$X_t$; this is the validity notion of TFDP \citep[Assumption~4(ii)]{forcier2023trajectory}
and of stochastic Lipschitz dynamic programming (SLDP) \citep[eq.~(20)]{ahmed2022sldp}. Then every cut is below $V_t$ on $X_t$, so~\ref{O3}
holds with $U_t=V_t$. If moreover $\lv t0\le\cB_t(\lv{t+1}0)$, then $\lv tk\le\cB_t(\lv{t+1}k)$ on
$X_t$ for all $k\ge0$, so that $r^k_t\ge0$ for all $k\ge1$.
\item (Tightness for the updated approximation implies~\ref{O2}.) If
$f^k_t(x)\ge\cB_t(\lv{t+1}{k})(x)-\bgamma^k_t$ at $x=\bx^k_t$ and
$\lv{t+1}{k}\ge\lv{t+1}{k-1}$, then~\ref{O2} holds.
\end{enumerate}
\end{lemma}

The proof is in \cref{app:setting}. The following are \emph{not} assumed: convexity of $V_t$,
of $\ell_t$ or of the graph of $\cX_t(\cdot,\xi)$; affine cuts; Lipschitz continuity of $V_t$.
Convexity only serves, through the oracle, to guarantee that good cuts exist.

\subsection{Conditions on the state sets}\label{sec:VG}

The backward pass raises an approximation at a single trial point; the main theorem turns this into
progress on the whole state set in two ways, each with its condition on $X_t$: through the volume under the
approximation, or through its values on a finite net. The volume version (\cref{lem:lift}) uses the
following \emph{volume-growth} condition: there are a Borel
probability measure $\mu_t$ on $X_t$, constants $\kappa_t\in(0,1]$ and $D_t>0$, and an
integer $d_t\ge0$ such that, with the convention $0^0=1$,
\begin{equation}\label{eq:setting:volume_growth}\tag{VG$_t$}
\mu_t\big(\{y\in X_t:\dist_t(x,y)\le s\}\big)\ \ge\ \kappa_t\min\{1,(s/D_t)^{d_t}\}
\qquad\forall x\in X_t,\ s\ge0 .
\end{equation}
We always take $d_0=0$ and $\kappa_0=1$ (recall $X_0=\{x_0\}$), and we fix an exponent
$d\ge\max_{t<T}d_t$.

\begin{lemma}[Cases of~\eqref{eq:setting:volume_growth}]\label{lem:VG}
\begin{enumerate}[label=(\alph*),nosep]
\item \emph{Convex sets.} If $X_t$ is convex and compact, $\dist_t$ is given by a norm,
$d_t:=\dim\aff X_t$, $D_t\ge\diam X_t$ and $\mu_t$ is the normalised Lebesgue measure of
$\aff X_t$ on $X_t$, then~\eqref{eq:setting:volume_growth} holds with $\kappa_t=1$.
\item \emph{Finite unions of convex sets (mixed-integer states).} If
$X_t=\bigcup_{j\in I_t}K^j$ with $I_t$ finite and each $K^j$ convex and compact,
$\mu_t:=|I_t|^{-1}\sum_j\mu_{K^j}$ with $\mu_{K^j}$ as in~(a), $d_t\ge\max_j\dim\aff K^j$ and
$D_t\ge\max_j\diam K^j$, then~\eqref{eq:setting:volume_growth} holds with $\kappa_t=1/|I_t|$ for $\dist_t$ given by a
norm. If the $K^j$ are pairwise disjoint, it also holds for the extended metric equal to
$\|x-y\|$ when $x,y$ lie in the same $K^j$ and $+\infty$ otherwise; for that metric,~\ref{O1}
only asks for Lipschitz continuity on each piece. For instance,
$X_t=\bigcup_{z\in Z_t}\{z\}\times K_z$ with $Z_t$ finite gives $\kappa_t=1/|Z_t|$.
\item \emph{Finite sets.} If $X_t$ is finite and $\mu_t$ is uniform,~\eqref{eq:setting:volume_growth} holds
with $\kappa_t=1/|X_t|$ and $d_t=0$, for any $\dist_t$.
\end{enumerate}
\end{lemma}

The proof is in \cref{app:setting}. Condition~\eqref{eq:setting:volume_growth} can fail for compact sets with
outward cusps: at the origin of $\{(u,v):0\le u\le1,\ |v|\le u^2\}$, the Lebesgue measure of a
ball of radius $s$ is of order $s^3$, not $s^2$. The covering version of the main theorem needs
no such condition, only finite nets. A set $C\subset X_t$ is a \emph{$\sigma$-net} for $\dist_t$ if every
$x\in X_t$ has a point $c\in C$ with $\dist_t(x,c)\le\sigma$. By volume packing (\cref{lem:nets}), a compact
subset of a $d_t$-dimensional affine space with diameter at most $D_t$ has $\sigma$-nets of size
$(1+2D_t/\sigma)^{d_t}$, and a union of $|Z_t|$ such fibres has nets of size $|Z_t|(1+2D_t/\sigma)^{d_t}$.

\paragraph{The standard case.}\label{sec:standard} A first reading may assume the standard case: SDDP with an
exact forward pass ($\gamma^F=0$), cuts that are valid ($U=V$) and exact at the trial point ($\gamma=0$),
convex compact state sets and $\dist_t$ given by a norm. Then $\kappa_t=1$ and $d_t=\dim\aff X_t$
(\cref{lem:VG}(a)). The other cases serve later results: extended metrics and finite sets for mixed-integer and
binary states, and finite nets for nonconvex state sets.

\section{Main results}\label{sec:main}

This section proves the master theorem (\cref{thm:master}) along the three steps outlined in the introduction,
then turns it into anytime guarantees and explicit rates.

\subsection{Gap identity and lift}\label{sec:lemmas}

The errors of an iteration add up to
\begin{equation}\label{eq:main:error_sums}
\bGamma_k:=\sum_{t<T}\bgamma^k_t,\qquad \bGamma^F_k:=\sum_{t\in[T]}\bgamma^{F,k}_t,\qquad
\bar{\bGamma}_k:=\bGamma^F_k+\EE[\bGamma_k\mid\cA_{k-1}],
\end{equation}
and, for deterministic errors $\bgamma^k_t=\gamma_t$ and $\bgamma^{F,k}_t=\gamma^F_t$, we write
$\Gamma:=\sum_{t<T}\gamma_t+\sum_{t\in[T]}\gamma^F_t$.

\begin{lemma}[Bounds]\label{lem:basic}
Under \cref{ass:S,ass:O}, we have, for all $k\ge1$ and $t<T$:
\begin{enumerate}[label=(\roman*),nosep]
\item $\lv tk$ is Borel and $L_t$-Lipschitz for $\dist_t$, and
$\lv t0\le\lv t{k-1}\le\lv tk\le\Ub t$ on $X_t$, where $\Ub t$ is the Lipschitz envelope of the ceiling~\ref{O3};
\item the residual $r^k_t$ of~\eqref{eq:setting:residuals} is bounded on $X_t$, uniformly in $k$;
\item $V^{\pi^k}_0(x_0)\ge V_0(x_0)$.
\end{enumerate}
\end{lemma}

\begin{proof}
(i) A finite maximum of Borel $L_t$-Lipschitz functions is Borel and $L_t$-Lipschitz. If $f$
is $L_t$-Lipschitz and $f\le U_t$, then $f(x)\le f(z)+L_t\dist_t(x,z)\le U_t(z)+L_t\dist_t(x,z)$
for all $z$, so $f\le\Ub t$. This applies to $\lv t0$ and to every cut, and~\ref{O4} gives the
monotonicity. 
(ii) $\lv t0\le\Ub t\le M_t+\sup\lv t0$ by the definition of the potential bound $M_t$ in~\ref{O3}, so $\Ub t$ is bounded; then use
$|\cB_t(W)|\le\sup|\cB_t(0)|+\sup|W|$ with $W=\lv{t+1}{k-1}$. (iii) The realised costs of $\pi^k$ are
integrable given $\cA_{k-1}$ (proof of \cref{lem:gap}), and a policy defined by measurable selections with
integrable costs satisfies $\EE[\sum_{s>t}\ell_s\mid\bx_t]\ge V_t(\bx_t)$, by backward induction and
stagewise independence.
\end{proof}

\begin{lemma}[Gap identity]\label{lem:gap}
Under \cref{ass:S}, let $\lv t{k-1}$, $t\le T$, be any bounded approximations as in the measurability condition~\ref{S3}, with
$\lv T{k-1}=0$. Then the gap $G^k$
of~\eqref{eq:setting:expected_policy_cost} and the sum $R^k$ of the residuals along the sampled path
\eqref{eq:setting:residuals} satisfy
\begin{equation}\label{eq:main:gap_identity}
G^k=\EE\big[R^k\mid\cA_{k-1}\big]+\EE\Big[\sum_{t\in[T]}\be^k_t\ \Big|\ \cA_{k-1}\Big],
\qquad \be^k_t\in[0,\bgamma^{F,k}_t],
\end{equation}
where $\be^k_t:=\ell_t(\bx_{t-1},\bx_t,\bxi_t)+\lv t{k-1}(\bx_t)-\Bd_{t-1}(\lv t{k-1};\bx_{t-1},\bxi_t)$ is the
error of the forward step, with $\Bd_{t-1}$ the value of the stage problem~\eqref{eq:setting:bellman_operator}, and
$(\bx_t,\bxi_t)=(\bx^k_t,\bxi^k_t)$. In particular $G^k\le\EE[R^k\mid\cA_{k-1}]+\bGamma^F_k$,
with equality for the exact greedy pass. Under \cref{ass:O}, the gap inequality~\eqref{eq:setting:gap_meaning} holds.
\end{lemma}

\begin{proof}
Let $\mathcal H_{t-1}:=\cA_{k-1}\vee\sigma(\bxi^k_{[t-1]})$. Then $\bx_{t-1}$ is
$\mathcal H_{t-1}$-measurable and lies in $X_{t-1}$ by~\ref{S1}, and $\bxi_t$ is independent of
$\mathcal H_{t-1}$ with law $\PP_t$. The substitution (``freezing'') property of independent variables
\citep{kallenberg2002foundations} states that if $g$ is jointly measurable, $Y$ is $\mathcal H$-measurable, $\bxi$ is
independent of $\mathcal H$ and $\EE|g(Y,\bxi)|<\infty$, then $\EE[g(Y,\bxi)\mid\mathcal H]=\bar g(Y)$ with
$\bar g(y):=\EE g(y,\bxi)$. It gives
$\EE[\Bd_{t-1}(\lv t{k-1};\bx_{t-1},\bxi_t)\mid\mathcal H_{t-1}]=\cB_{t-1}(\lv t{k-1})(\bx_{t-1})=\lv{t-1}{k-1}(\bx_{t-1})+r^k_{t-1}(\bx_{t-1})$,
so
\begin{equation}\label{eq:main:one_step_expectation}
\EE\big[\ell_t(\bx_{t-1},\bx_t,\bxi_t)+\lv t{k-1}(\bx_t)\,\big|\,\mathcal H_{t-1}\big]
=\lv{t-1}{k-1}(\bx_{t-1})+r^k_{t-1}(\bx_{t-1})+\EE[\be^k_t\mid\mathcal H_{t-1}] .
\end{equation}
Take $\EE[\cdot\mid\cA_{k-1}]$ and sum over $t\in[T]$. The terms in $\lv{}{k-1}$ telescope to $-\lv0{k-1}(x_0)$,
because $\lv T{k-1}=0$ and $x_0$ is deterministic. The bounds on $\be^k_t$ are the forward-error condition~\eqref{eq:setting:forward_error}
and the definition of $\Bd_{t-1}$. All terms are integrable given $\cA_{k-1}$: the realised cost satisfies
$\inf_y\ell_t(\bx_{t-1},y,\bxi_t)\le\ell_t(\bx_{t-1},\bx_t,\bxi_t)\le\inf_y\ell_t(\bx_{t-1},y,\bxi_t)+2\sup|\lv t{k-1}|+\bgamma^{F,k}_t$,
and $\EE\big[|\inf_y\ell_t(\bx_{t-1},y,\bxi_t)|\,\big|\,\mathcal H_{t-1}\big]$ is at most $\sup_{x\in X_{t-1}}\int|\inf_y\ell_t(x,y,\xi)|\,\PP_t(d\xi)$,
which is finite by the integrability condition~\ref{S2}, by the freezing property. For the gap inequality~\eqref{eq:setting:gap_meaning}: $G^k\ge V^{\pi^k}_0(x_0)-\Ub0(x_0)$ and $\Ub0(x_0)=U_0(x_0)$
because $X_0=\{x_0\}$, so $G^k\ge V^{\pi^k}_0(x_0)-V_0(x_0)-\zeta_0$ with $\zeta_0$ as in \cref{sec:accuracy}; and $V^{\pi^k}_0(x_0)\ge V_0(x_0)$ by
\cref{lem:basic}(iii).
\end{proof}

For finite state spaces and an exact greedy pass, this lemma is \citet[Lemma~2]{efroni2019tight}
for RTDP; the proof is their telescoping argument, written for general noise, an inexact forward pass and
arbitrary approximations. The identity involves the forward law only: it holds whatever the backward pass
computes (sampled cuts, cuts for another law, inexact cuts). What the backward pass must do is lift the
approximation at the trial point by an amount comparable to the residual there.

\begin{lemma}[A Lipschitz lift lifts a volume]\label{lem:lift}
Assume the volume-growth condition~\eqref{eq:setting:volume_growth}.
Let $V,W:X_t\to\R$ be Borel and $L_t$-Lipschitz for $\dist_t$, with
$W\ge V$, and let $x\in X_t$ and $s\ge0$ satisfy $W(x)\ge V(x)+s$. Then
\begin{equation}\label{eq:main:lift_volume}
\int(W-V)\,d\mu_t\ \ge\ \kappa_t\,\varphi_t(s),\quad \text{where}\quad 
\varphi_t(s):=\frac s2\min\Big\{1,\Big(\frac{s}{4L_tD_t}\Big)^{d_t}\Big\},
\end{equation}
with $s/(4L_tD_t):=\infty$ if $L_t=0$. If $d_t=0$, the bound improves to $\kappa_ts$.
\end{lemma}

\begin{proof}
If $\dist_t(x,y)\le u$, then $W(y)-V(y)\ge W(x)-V(x)-2L_tu\ge s-2L_tu$, and $W-V\ge0$
everywhere. So $\int(W-V)\,d\mu_t\ge(s-2L_tu)_+\,\kappa_t\min\{1,(u/D_t)^{d_t}\}$ for every
$u\ge0$. Take $u=D_t$ if $s\ge4L_tD_t$ and $u=s/(4L_t)$ otherwise; if $d_t=0$, take $u=0$.
\end{proof}

At iteration $k$, tightness~\ref{O2} and the monotone update~\ref{O4} give, at $x=\bx^k_t$,
$\lv tk(x)\ge\lv t{k-1}(x)+(r^k_t(x)-\bgamma^k_t)_+$, so \cref{lem:lift} applies at every index $t<T$
of every backward pass with $V=\lv t{k-1}$, $W=\lv tk$ and $s=(r^k_t(x)-\bgamma^k_t)_+$.

\paragraph{The aggregated lift function.} Let $\rho_t:=4L_tD_t$ and let $\hat\psi_t$ be the convex
``Huber-type'' minorant of $\varphi_t$ with exponent $d\ge\max_{t<T}d_t$:
\begin{equation}\label{eq:main:huber_potential}
\hat\psi_t(s):=\begin{cases}\dfrac{s^{d+1}}{2(d+1)\rho_t^{\,d}}, & 0\le s\le\rho_t,\\[2mm]
\dfrac{\rho_t}{2(d+1)}+\dfrac{s-\rho_t}{2}, & s\ge\rho_t.\end{cases}
\end{equation}
When $\rho_t=0$ (that is, $L_t=0$, for instance at the root), only the second branch applies: $\hat\psi_t(s)=s/2$.
For every $\rho_t\ge0$, $\hat\psi_t\le\varphi_t$ because $d_t\le d$. Set $\psi_t:=\kappa_t\hat\psi_t$ and let
$\Psi_T:=\square_{t<T}\,\psi_t$ be their infimal convolution,
$\Psi_T(g)=\min\{\sum_{t<T}\psi_t(s_t):s\ge0,\ \sum_{t<T}s_t=g\}$. It is convex, finite and nondecreasing
on $[0,\infty)$, with $\Psi_T(0)=0$, and we extend it to $\R$ by $\Psi_T(u):=\Psi_T(u_+)$, which
keeps it convex and nondecreasing. This \emph{aggregated lift function} $\Psi_T(g)$ is the smallest total
lift that these minorants guarantee when the residuals along a path, net of the cut errors, add up to $g$. For small $g$ and stage-uniform constants,
$\Psi_T(g)\approx\kappa g^{d+1}/(2(d+1)T^d\rho^d)$, which is proportional to $T$ at the relative accuracy
$g=T\varepsilon$. The convex minorant $\hat\psi_t$ replaces $\varphi_t$ because the conditional Jensen step in
the proof of \cref{thm:master} needs $\Psi_T$ convex.

\subsection{The master theorem}\label{sec:master}

In expectation, an iteration with a large gap raises the potential by a definite amount, and the potential
is bounded, so there can be only few such iterations. The two versions of the theorem match the two
conditions of \cref{sec:VG}; its high-probability statements rest on the following bound.

\begin{lemma}[{Decreasing bounded processes; \citealp[Thm.~3]{efroni2019tight}}]\label{lem:dbp}
Let $(Y_k)_{k\ge0}$ be adapted to $(\cA_k)$ and nonincreasing, with $0\le Y_k\le Y_0=C$ and $C$
deterministic, and let $\mathrm{Reg}_K:=\sum_{k\in[K]}(Y_{k-1}-\EE[Y_k\mid\cA_{k-1}])$. For $\delta\in(0,1]$, with
probability at least $1-\delta$, $\mathrm{Reg}_K\le9C\ln(3/\delta)$ for all $K$. In particular, if
$\EE[Y_{k-1}-Y_k\mid\cA_{k-1}]\ge f_k\ge0$ with $f_k$ $\cA_{k-1}$-measurable, then, since
$\mathrm{Reg}_K\ge\sum_{k\in[K]}f_k$, $\sum_{k\ge1}f_k\le9C\ln(3/\delta)$ with probability at least $1-\delta$.
\end{lemma}

In the theorem, $\Psi_T$ is the aggregated lift function of \cref{sec:lemmas}, $\bar{\bGamma}_k$ and $\Gamma$ are
the error sums~\eqref{eq:main:error_sums}, and $\Ub t$ and $M_t$ are the Lipschitz envelope of the ceiling and
the potential bound of~\ref{O3}; if $\Ub t$ is not Borel, its integral is an upper integral (\cref{app:measurability}).

\begin{theorem}[Master theorem]\label{thm:master}
Let \cref{ass:S,ass:O} hold.
\begin{enumerate}[label=(\alph*),nosep]
\item \emph{Volume version.} If~\eqref{eq:setting:volume_growth} holds for all $t<T$, then
\begin{equation}\label{eq:main:master_volume}
\EE\sum_{k\ge1}\Psi_T\big((G^k-\bar{\bGamma}_k)_+\big)\ \le\ \sum_{t<T}\int(\Ub t-\lv t0)\,d\mu_t
\ \le\ \sum_{t<T}M_t .
\end{equation}
In particular, with deterministic errors, $\EE\,\#\{k\ge1:G^k>\theta+\Gamma\}\le\sum_{t<T}M_t/\Psi_T(\theta)$
for every $\theta>0$. Moreover, for every $\delta\in(0,1]$, with probability at least $1-\delta$,
simultaneously for all $\theta>0$,
$\#\{k\ge1:G^k>\theta+\bar{\bGamma}_k\}\le9\ln(3/\delta)\sum_{t<T}M_t/\Psi_T(\theta)$.
\item \emph{Covering version.} Assume deterministic errors $\gamma_t$, $\gamma^F_t$. Fix $\theta>0$, numbers
$\sigma_t\ge0$ with $\sum_{t<T}L_t\sigma_t\le\theta/4$, and finite $\sigma_t$-nets $C_t\subset X_t$. Then
\begin{equation}\label{eq:main:master_covering}
\EE\sum_{k\ge1}\big(G^k-\Gamma-\theta/2\big)_+\ \le\ \sum_{t<T}\sum_{c\in C_t}(\Ub t-\lv t0)(c)\ \le\ \sum_{t<T}|C_t|\,M_t .
\end{equation}
In particular, $\EE\,\#\{k\ge1:G^k>\theta+\Gamma\}\le\frac2\theta\sum_{t<T}|C_t|M_t$. Moreover, for every
$\delta\in(0,1]$, with probability at least $1-\delta$,
$\#\{k\ge1:G^k>\theta+\Gamma\}\le\frac{18\ln(3/\delta)}{\theta}\sum_{t<T}|C_t|M_t$.

\item \emph{Finite state sets.} Assume deterministic errors, every $X_t$ finite and $\dist_t=+\infty$ off the
diagonal. Then
$\EE\sum_{k\ge1}(G^k-\Gamma)_+\le C_{\rm fin}:=\sum_{t<T}\sum_{x\in X_t}(U_t-\lv t0)(x)$ and, for every
$\delta\in(0,1]$, with probability at least $1-\delta$, $\sum_{k\ge1}(G^k-\Gamma)_+\le9\ln(3/\delta)\,C_{\rm fin}$.
\end{enumerate}
\end{theorem}

Part~(a) handles all thresholds at once, allows random errors and has the better constants; \cref{sec:saa}
uses it, and \cref{thm:master-ih} extends all three parts to the infinite horizon. Part~(b) needs no volume growth, hence no convexity of the
state sets; its proof (\cref{app:cover}) is the covering argument of \citet[App.~F]{domingues2021kernel},
with~\ref{O2} in place of the Bellman update. It gives all the orders of this paper, and it is the version
used for TFDP's generic cuts and Zhang and Sun's algorithm (\cref{sec:instances}). Part~(c) is used for SDDiP.
In part~(c), every function on $X_t$ is Borel, bounded and
$L_t$-Lipschitz, so~\ref{O1} only asks $\lv t0$ to be deterministic, and $\Ub t=U_t$.

\begin{proof}[Proof of~(a)]
Let $\Phi^k:=\sum_{t<T}\int(\lv tk-\lv t0)\,d\mu_t$ and $C_U:=\sum_{t<T}\int(\Ub t-\lv t0)\,d\mu_t$. By
\cref{lem:basic}(i), $0=\Phi^0\le\Phi^{k-1}\le\Phi^k\le C_U\le\sum_{t<T}M_t$.

\emph{Pathwise lift.} Fix $k\ge1$ and $t<T$, let $x=\bx^k_t$, and let
$s_t:=(r^k_t(x)-\bgamma^k_t)_+$. As noted after \cref{lem:lift},
$\int(\lv tk-\lv t{k-1})\,d\mu_t\ge\kappa_t\varphi_t(s_t)\ge\psi_t(s_t)$. Sum over $t<T$. By the
definition of $\square$, by $\sum_{t<T}(a_t)_+\ge(\sum_{t<T}a_t)_+$ and because $\Psi_T$ is nondecreasing,
\begin{equation}\label{eq:main:pathwise_lift}
\Phi^k-\Phi^{k-1}\ \ge\ \sum_{t<T}\psi_t(s_t)\ \ge\ \Psi_T\Big(\sum_{t<T}s_t\Big)\ \ge\ \Psi_T\big((R^k-\bGamma_k)_+\big).
\end{equation}

\emph{Conditional expectation.} The map $(u,v)\mapsto\Psi_T((u-v)_+)$ is convex and finite on
$\R^2$, $R^k$ is bounded by \cref{lem:basic}(ii), and $\bGamma_k$ is integrable. The
conditional Jensen inequality, \cref{lem:gap} and monotonicity give
\begin{equation}\label{eq:main:conditional_lift}
\EE[\Phi^k-\Phi^{k-1}\mid\cA_{k-1}]\ \ge\ \Psi_T\big((\EE[R^k\mid\cA_{k-1}]-\EE[\bGamma_k\mid\cA_{k-1}])_+\big)
\ \ge\ \Psi_T\big((G^k-\bar{\bGamma}_k)_+\big)=:f_k .
\end{equation}
Summing over $k\in[K]$, the left-hand sides add up to $\EE\Phi^K\le C_U$; let $K\to\infty$ (monotone convergence).

\emph{High probability.} \Cref{lem:dbp} applies to $Y_k:=C_U-\Phi^k$, which is adapted by~\ref{S3}, with
$Y_0=C_U$ deterministic since $\lv{}0$, the envelopes $\Ub t$ and the measures $\mu_t$ are, and to $f_k$, which
is $\cA_{k-1}$-measurable. It gives $\sum_{k\ge1}f_k\le9C_U\ln(3/\delta)$ with probability at least $1-\delta$, and
$\sum_kf_k\ge\Psi_T(\theta)\,\#\{k\ge1:G^k>\theta+\bar{\bGamma}_k\}$ for every $\theta$. The last claim follows
from~\eqref{eq:main:master_volume}, since $\Psi_T$ is nondecreasing.
\end{proof}

We now read the master theorem as statements about individual iterations.

\begin{corollary}[Anytime guarantee]\label{cor:anytime}
In the setting of \cref{thm:master}(a):
\begin{enumerate}[label=(\alph*),nosep]
\item If $\bar{\bGamma}_k\le g_k$ for a deterministic sequence $(g_k)_{k\ge1}$, then for all
$\theta_1,\theta_2>0$,
$\EE\,\#\{k\ge1:G^k>\theta_1+\theta_2\}\le\#\{k\ge1:g_k>\theta_2\}+\sum_{t<T}M_t/\Psi_T(\theta_1)$.
\item With deterministic errors, let $K_\theta:=9\ln(3/\delta)\sum_{t<T}M_t/\Psi_T(\theta)$. With
probability at least $1-\delta$, for every $\theta>0$, $V_0(x_0)-\lv0k(x_0)\le\theta+\Gamma$ at every iteration
$k\ge K_\theta$, and one of the policies
$(\pi^k)_{1\le k\le\lfloor K_\theta\rfloor+1}$ is $(\theta+\Gamma+\zeta_0)$-optimal, with $\zeta_0$ as in \cref{sec:accuracy}.
\item (Almost sure convergence.) With deterministic errors, $\limsup_kG^k\le\Gamma$ almost surely. In particular, with
exact cuts, an exact forward pass and $U=V$, $G^k\to0$ almost surely: the lower bound $\lv0{k-1}(x_0)$ and the expected
cost $V^{\pi^k}_0(x_0)$ of the policy both converge to $V_0(x_0)$.
\item (Uniformly selected iterate; no estimate of $G^k$ is needed.) For $K\ge1$ and $\theta>0$, if $\mathbf k$ is uniform on $[K]$ and independent
of the run, then
$\PP(G^{\mathbf k}>\theta+\bar{\bGamma}_{\mathbf k})\le\sum_{t<T}M_t/(K\Psi_T(\theta))$; on the complementary event,
$\pi^{\mathbf k}$ is $(\theta+\bar{\bGamma}_{\mathbf k}+\zeta_0)$-optimal and
$V_0(x_0)-\lv0{\mathbf k-1}(x_0)\le\theta+\bar{\bGamma}_{\mathbf k}$.
\end{enumerate}
\end{corollary}

\begin{proof}
(a) If $g_k\le\theta_2$ and $G^k>\theta_1+\theta_2$, the $k$-th term of~\eqref{eq:main:master_volume} is at
least $\Psi_T(\theta_1)$. (b) On the event of \cref{thm:master}, at most $K_\theta$ indices $k$ have
$G^k>\theta+\Gamma$, so some $k\le\lfloor K_\theta\rfloor+1$ has $G^k\le\theta+\Gamma$; use the gap inequality~\eqref{eq:setting:gap_meaning} and
the monotonicity of $\lv0j(x_0)$ in $j$. (c) For $\theta>0$, $\Psi_T(\theta)>0$, so
$\#\{k\ge1:G^k>\theta+\Gamma\}$ is integrable by \cref{thm:master}(a), hence finite almost surely; intersect
these events over rational $\theta$ and use the gap inequality~\eqref{eq:setting:gap_meaning} with $\zeta_0=0$.
(d) Given the run, the event $\{G^{\mathbf k}>\theta+\bar{\bGamma}_{\mathbf k}\}$ has conditional probability $\#\{k\in[K]:G^k>\theta+\bar{\bGamma}_k\}/K$, and each
such $k$ contributes at least $\Psi_T(\theta)$ to~\eqref{eq:main:master_volume}; use~\eqref{eq:setting:gap_meaning}
on the complement.
\end{proof}

We call $\bar{\bGamma}_k$ (or $\Gamma$, for deterministic errors) the \emph{error floor}; it cannot be lowered in
general (\cref{rem:floor}).

\subsection{Explicit rates}\label{sec:rates}

We now make $\Psi_T$ and $M_t$ explicit, which gives the rates of \cref{tab:summary}. For a finite horizon the rates are
stated for the relative accuracy $T\varepsilon$, an average error of $\varepsilon$ per stage, the natural
scale since the losses add up over $T$ stages; replacing $\varepsilon$ with $\varepsilon/T$ gives the
absolute accuracy $\varepsilon$ (\cref{rem:absolute}).

The rates use one scale $\cb$: the \emph{initial error is at most $\cb$ per stage} if
\begin{equation}\label{eq:main:initial_error}
\sup_{X_t}\big(V_t-\lv t0\big)\ \le\ (T-t)\,\cb\qquad\text{for every }t<T .
\end{equation}
This holds when every stage loss $\ell_t$ takes values in an interval $[a_t,a_t+\cb]$ and $\lv t0$ is the
constant $\sum_{s>t}a_s$. With
$U_t=V_t$, \eqref{eq:main:initial_error} gives $M_t\le(T-t)\cb$, because $\Ub t\le U_t$.

\begin{corollary}[Stage-uniform constants]\label{cor:uniform}
In the setting of \cref{thm:master}(a), assume $L_t\le L$, $D_t\le D$ and $\kappa_t\ge\kappa$, an initial
error at most $\cb$ per stage, $U_t=V_t$ and deterministic errors, and let $\varepsilon\le4LD$. Then
$\Psi_T(g)\ge\kappa T\psi(g/T)$, where $\psi$ is $\hat\psi_t$ with $\rho=4LD$, and
\begin{equation}\label{eq:main:rate_relative}
\EE\,\#\{k\ge1:G^k>T\varepsilon+\Gamma\}\le\frac{(d+1)\,\cb\,(T+1)\,(4LD)^d}{\kappa\,\varepsilon^{d+1}}.
\end{equation}
\end{corollary}

\begin{proof}
First, $\hat\psi_t\ge\psi$, because $\hat\psi_t$ is nonincreasing in $\rho_t$ on both branches, and
$\square_t(\kappa_t\hat\psi_t)\ge\kappa\,\square_t\hat\psi_t$. Jensen gives
$\square_t\psi(g)\ge T\psi(g/T)$. By~\eqref{eq:main:initial_error}, $M_t\le(T-t)\cb$, so
$\sum_{t<T}M_t\le\cb T(T+1)/2$. Finally $T\psi(\varepsilon)=T\varepsilon^{d+1}/(2(d+1)(4LD)^d)$; the
condition $\varepsilon\le4LD$, the regime of interest, selects the first branch.
\end{proof}

\begin{remark}[Lipschitz constants that grow with the horizon]\label{rem:lipT}
Assume exact valid cuts, $\kappa_t=1$, an initial error at most $\cb$ per stage, $D_t=D$ (no restriction
when $D_t\le D$, since increasing $D_t$ only weakens~\eqref{eq:setting:volume_growth}) and
$L_t=(T-t)\tilde L$, as relatively complete recourse typically gives (\cref{prop:benders}, with
$\tilde L=\cb/\eta$). Let $w_t:=T-t$, $\bar w:=\sum_{t<T}w_t=T(T+1)/2$, and let $\tilde\psi$ be built
as in~\eqref{eq:main:huber_potential} with $\tilde\rho:=4\tilde LD$. Since $\rho_t=w_t\tilde\rho$,
$\psi_t(s)=w_t\tilde\psi(s/w_t)$ is the perspective of $\tilde\psi$ with weight $w_t$, so
$\Psi_T(g)=\bar w\tilde\psi(g/\bar w)$, attained at $s_t=w_tg/\bar w$. For $\varepsilon\le4\tilde LD$ this gives
\begin{equation}\label{eq:main:rate_growing_lipschitz}
\EE\,\#\{k\ge1:G^k>T\varepsilon\}\le\frac{(d+1)\,\cb\,(4\tilde LD)^d\,(T+1)^{d+1}}{2^d\,\varepsilon^{d+1}},
\end{equation}
and, with these $L_t$, the deterministic bound of \citet[Thm.~2]{zhang2022stochastic} with per-stage tolerances
$\varepsilon$ is $\sum_{t<T}(1+2(T-t)\tilde LD/\varepsilon)^d$.
\end{remark}

\begin{corollary}[Discounted costs, horizon-free]\label{cor:disc}
Let the stage-$t$ loss be $\beta^{t-1}\hat\ell_t$ with $0<\beta<1$ and $\hat\ell_t$ taking values in an
interval $[a_t,a_t+\cb]$, let $\lv t0$ be the constant $\sum_{s>t}\beta^{s-1}a_s$ (e.g.\ $\hat\ell_t\in[0,\cb]$
and $\lv{}0=0$), let the cuts of index $t$ be $\beta^{t}\Lh$-Lipschitz, $X_t$ convex ($\kappa_t=1$),
$D_t\le D$, cuts valid and exact ($U=V$, $\gamma=\gamma^F=0$), and $(1-\beta)\varepsilon\le4\Lh D$.
Then, for every horizon $T$,
\begin{equation}\label{eq:main:rate_discounted}
\EE\,\#\{k\ge1:G^k>\varepsilon\}\le\frac{2(d+1)\,\cb\,(4\Lh D)^d}{(1-\beta)^{d+2}\,\varepsilon^{d+1}} .
\end{equation}
\end{corollary}

\begin{proof}
Increasing $D_t$ only weakens~\eqref{eq:setting:volume_growth}, so we may take $D_t=D$.
Then $\rho_t=\beta^{t}\hat\rho$ with $\hat\rho:=4\Lh D$, so $\psi_t(s)=\beta^{t}\hat\psi(s/\beta^{t})$,
where $\hat\psi$ is built with $\hat\rho$. Jensen with weights $\beta^{t}/B$, $B=\sum_{t<T}\beta^{t}$,
gives $\Psi_T(g)\ge B\hat\psi(g/B)$. The map $B\mapsto B\hat\psi(g/B)$ is nonincreasing, because
$\hat\psi(s)/s$ is nondecreasing, and $B\le(1-\beta)^{-1}$; so
$\Psi_T(g)\ge(1-\beta)^{-1}\hat\psi((1-\beta)g)$. Finally $M_t\le\sup(V_t-\lv t0)\le\sum_{s>t}\beta^{s-1}\cb\le\beta^{t}\cb/(1-\beta)$, so
$\sum_{t<T}M_t\le\cb/(1-\beta)^2$.
\end{proof}

With errors and general $\kappa_t$, the same proof gives~\eqref{eq:main:rate_discounted} with $\varepsilon+\Gamma$ in
the event and the bound divided by $\kappa=\min_t\kappa_t$.

\begin{corollary}[Covering rates; finite state sets]\label{cor:cover}
In the setting of \cref{thm:master}(b) for~(a) and of \cref{thm:master}(c) for~(b), with initial error at most
$\cb$ per stage and $U_t=V_t$:
\begin{enumerate}[label=(\alph*),nosep]
\item if $L_t\le L$ and $X_t$ has $\sigma$-nets of size at most $Z(1+2D/\sigma)^d$ for every
$\sigma>0$ (\cref{lem:nets}; $Z\ge\max_t|Z_t|$ bounds the number of discrete configurations, $Z=1$ for convex $X_t$), then
\begin{equation}\label{eq:main:covering_rate_relative}
\EE\,\#\{k\ge1:G^k>T\varepsilon+\Gamma\}\le Z\cb\,(T+1)(1+8LD/\varepsilon)^d/\varepsilon ;
\end{equation}
\item if every $X_t$ is finite, $\EE\sum_k(G^k-\Gamma)_+\le\cb\sum_{t<T}(T-t)|X_t|$, hence
$\EE\,\#\{k\ge1:G^k>\theta+\Gamma\}\le\cb\,T(T+1)\max_t|X_t|/(2\theta)$.
\end{enumerate}
\end{corollary}

\begin{proof}
(a) If $L>0$, take $\sigma_t=\theta/(4TL)$; if $L=0$, take a finite $\sigma_t$ so large that
$Z(1+2D/\sigma_t)^d<Z+1$, so that $X_t$ has a $\sigma_t$-net of size at most $Z$. In both cases
$\sum_tL\sigma_t\le\theta/4$ and, for $\theta=T\varepsilon$, $|C_t|\le Z(1+8LD/\varepsilon)^d$; use
$\sum_{t<T}M_t\le\cb T(T+1)/2$. (b) $M_t\le(T-t)\cb$ by~\eqref{eq:main:initial_error}, and $\Ub t=U_t$.
\end{proof}

For a finite state set $X$, horizon $T$ and losses in $[0,1]$, part~(b) gives $\EE\sum_kG^k\le|X|T(T+1)/2$,
the analogue of \citet[Thm.~4]{efroni2019tight} for RTDP. No Lipschitz property is used
there, so every cut family that satisfies tightness~\ref{O2} and the ceiling condition~\ref{O3} qualifies (\cref{sec:SDDiP}).

\begin{remark}[Absolute accuracy]\label{rem:absolute}
Replacing $\varepsilon$ with $\varepsilon/T$ in~\eqref{eq:main:rate_relative} gives, for $\varepsilon\le4LDT$,
\begin{equation}\label{eq:main:rate_absolute}
\EE\,\#\{k\ge1:G^k>\varepsilon+\Gamma\}\le\frac{(d+1)\,\cb\,T^{d+1}(T+1)\,(4LD)^d}{\kappa\,\varepsilon^{d+1}},
\end{equation}
where the extra factor $T^{d+1}$ only reflects the per-stage error $\varepsilon/T$. With the constants of
\cref{rem:lipT}, the same substitution in~\eqref{eq:main:rate_growing_lipschitz} gives the bound
$2(d+1)\,\cb\,(4\tilde LD)^d\,\bar w^{\,d+1}/\varepsilon^{d+1}$, and that of \citet[Thm.~2]{zhang2022stochastic}, with
per-stage tolerances $\varepsilon/T$, is $\sum_{t<T}(1+2(T-t)\tilde LDT/\varepsilon)^d$. \Cref{cor:cover}(a) converts likewise.
\end{remark}

\subsection{Remarks and variants}\label{sec:variants}

\paragraph{Sharper and other forms of the bound.} The middle term of~\eqref{eq:main:master_volume} is an
average and can be much smaller than $\sum_{t<T}M_t$; any probability measure
satisfying~\eqref{eq:setting:volume_growth} can be used. Hölder's inequality also turns the master bound into
a regret bound, $\EE\sum_{k\in[K]}G^k=O(K^{d/(d+1)})$ (\cref{rem:regret}).

\paragraph{Multi-cut models.} In a multi-cut model, $\lv tk=\EE_{\xi\sim\PP_{t+1}}Q^k_{t,\xi}$, where
$Q^k_{t,\xi}$ is the maximum of the cuts of scenario $\xi$. The monotone update~\ref{O4} then fails, but the
proof of \cref{thm:master}(a) only uses $\lv t{k-1}\le\lv tk\le\Ub t$, Lipschitz approximations and the lift
at the trial point. If each scenario cut at $x=\bx^k_t$ is below the scenario value
$\Bd_t(\lv{t+1}k;\cdot,\xi)$ of~\eqref{eq:setting:bellman_operator} and exact at $x$, the lift of the averaged
model at $x$ is at least $r^k_t(x)$ and $\lv tk\le V_t$, so the theorem applies with $L_t$ the
$\PP_{t+1}$-average of the per-scenario Lipschitz constants.

\paragraph{Several forward paths.} Several forward paths per iteration lift the approximations at least as
much as one path, so the bounds still hold, although they do not improve with the number of paths. Backward
passes on the not-yet-updated approximation satisfy~\ref{O2} directly. With inexact forward passes, the guarantees concern the
$\gamma^F$-greedy selections actually used; they do not transfer to a different, even exact, selection used at
exploitation time \citep[Sect.~2.1]{forcier2023trajectory}.

\paragraph{What is not covered.} The analysis does not cover deterministic selection rules such as the problem
child, for which the packing argument of \citet{lan2022complexity} and \citet{zhang2022stochastic} is the
right tool, nor cut deletion, which breaks the monotone update~\ref{O4}.

\section{The main cut families}\label{sec:instances}

We now check \cref{ass:O} for the cut families of the TFDP framework; \cref{tab:families} gives the verdicts. Three
are treated here. Benders cuts are those of SDDP. The generalised conjugacy cuts of Zhang and Sun give a partial
answer to their question on bounds polynomial in $T$. SDDiP's cuts on binary states use \cref{thm:master}(c). The
other families, and all proofs, are in \cref{app:instances}.

\subsection{Benders cuts (SDDP)}\label{sec:benders}

\begin{proposition}[Bounded subgradients]\label{prop:benders}
Fix a norm and, for $\eta>0$, let $X^\eta_t:=\{x\in\aff X_t:\dist(x,X_t)\le\eta\}$.
Assume:
\begin{itemize}[nosep]
\item $\emptyset\ne \cX_{t+1}(x,\xi)\subset X_{t+1}$ for $x\in X^\eta_t$ and $t<T$ (extended relatively
complete recourse), and~\ref{S2} and~\ref{S3} hold with $X^\eta_t$ in place of $X_t$ as the set of incoming
states;
\item $X_t$ is convex, the graph of $\cX_{t+1}(\cdot,\xi)$ over $X^\eta_t$ is convex and $\ell_{t+1}(\cdot,\cdot,\xi)$ is
jointly convex on it;
\item the approximations are convex, with $\lv t0\le V_t$ and $\lv t0\le\cB_t(\lv{t+1}0)$;
\item every cut is a Benders cut $f^k_t(y)=h(x)+\langle g,y-x\rangle$ of
$h:=\cB_t(\lv{t+1}{k})$ at $x=\bx^k_t$, with $g$ a subgradient of the restriction of $h$
to $\aff X_t$.
\end{itemize}
Then $V_t=\cB_t(V_{t+1})$ is bounded on $X^\eta_t$, such $g$ exist, and the cuts are valid on $X^\eta_t$, exact
($\gamma=0$), and $L_t$-Lipschitz on $X_t$ with
\begin{equation}\label{eq:inst:benders_lipschitz}
L_t\le\Big(\sup_{X^\eta_t}V_t-\inf_{X_t}\lv t0\Big)\Big/\eta .
\end{equation}
If $\lv t0=\sum_{s>t}a_s$ and each $\ell_s$ takes values in $[a_s,a_s+\cb]$ for incoming states in $X^\eta_{s-1}$
(not only in $X_{s-1}$), this is at most $(T-t)\cb/\eta$.
So \cref{ass:O} holds with $U_t=V_t$, provided $\lv t0$ is also $L_t$-Lipschitz (e.g.\ $\lv t0$ constant).
\end{proposition}

This is the classical bounded-subgradient argument. The sets must be chosen so that
$X_{t+1}\supset \cX_{t+1}(X^\eta_t,\xi)$, so $X_{t+1}$ may have to be larger than the set of reachable states,
and $D$ must bound the diameter of this larger set. Under the bounded-loss condition and with finitely
supported noise, the results of \cref{sec:main} hold verbatim for SDDP, with the constant growing like $T-t$ (\cref{rem:lipT}). Adaptive
partition-based cuts \citep{song2015adaptive,siddig2022adaptive,forcier2022generalized}
are Benders cuts of an aggregation of the scenarios, refined until the aggregated cut coincides at the trial
point with the Benders cut of the full problem; they are covered as well, and cuts from coarser partitions are
valid and may be added.

\begin{example}[Stage-wise absorption: slopes uniform in the horizon]\label{ex:absorb}
In \cref{prop:benders}, assume that, for every $\xi$, $\cX_{t+1}(\cdot,\xi)$ is constant on $X^\eta_t$ and
$\ell_{t+1}(\cdot,y,\xi)$ is $c$-Lipschitz there: a deviation of the incoming state is absorbed within the stage at
unit cost at most $c$. Then $\cB_t(W)$ is $c$-Lipschitz on $X^\eta_t$ for every bounded $W$, so the Benders cuts
are $c$-Lipschitz, uniformly in $t$, $k$ and $T$. This holds, for instance, for a multi-item inventory with
stocks in a box $X\subset\R^n$, bounded regular orders, and unbounded emergency orders and disposals at unit
cost at most $c$ (\cref{app:absorb}). \Cref{cor:uniform} with $L=c$, $d=n$ and $\kappa=1$ then gives, for
$\varepsilon\le4cD$, $\EE\,\#\{k\ge1:G^k>T\varepsilon\}\le(n+1)\cb(T+1)(4cD)^n/\varepsilon^{n+1}$.
\end{example}

\begin{table}[!t]
\centering\small
\setlength{\tabcolsep}{3pt}
\begin{tabular}{>{\raggedright\arraybackslash}p{2.3cm}>{\raggedright\arraybackslash}p{2.6cm}>{\raggedright\arraybackslash}p{2.4cm}>{\raggedright\arraybackslash}p{2.4cm}>{\raggedright\arraybackslash}p{2.9cm}l}
\toprule
family & \ref{O1} constant & \ref{O2} tight for & \ref{O3} ceiling & verdict & where\\
\midrule
Benders (SDDP) & $\mathrm{osc}/\eta$ & $\cB_t$, exact & $V_t$ & covered & \cref{prop:benders}\\
TFDP generic & $L_t$ \citep[Ass.~4(iii)]{forcier2023trajectory} & $\cB_t$, $\gamma_t$ & $V_t$ & covered from the second iteration & \cref{prop:TFDP}\\
SLDP reverse norm & $L_t$ & $\cB_t$, exact & $V_t$ if $L_t\ge\Lip V_t$ & covered without snapping & \cref{prop:sldp}\\
SLDP AL cuts & $\|\lambda\|_*+\varrho\le L_t$ & $\cB_t$ if $L_t\ge\Lip g$, else $\BR_t$ & $V_t$ & covered if $L_t\ge\Lip g$; else $\gamma_t$ = regularisation gap & \cref{rem:AL}\\
Zhang--Sun & $\sum_mp(l_\lambda+l_\varrho)$ & $\BR_t$, exact & $V_t$ & covered & \cref{thm:ZS}\\
SDDiP, binary & not needed & $\cB_t$, exact & $V_t$ & covered & \cref{prop:SDDiP}\\
strengthened Benders & --- & not tight & $V_t$ & only in addition to a tight family & \cref{sec:SDDiP}\\
Lagrangian, MI states & --- & duality gap & $V_t$ & not covered & \cref{sec:SDDiP}\\
MIDAS steps & fails & $\cB_t$, exact & $V_t+(T-t)\varepsilon$ & without snapping only & \cref{thm:steps}\\
\bottomrule
\end{tabular}
\caption{Cut families and \cref{ass:O}. Forward passes are greedy (exact or $\gamma^F$-optimal)
trajectories drawn from the law used in the operator of~\ref{O2}. $\mathrm{osc}$ is the oscillation
$\sup_{X^\eta_t}V_t-\inf_{X_t}\lv t0$ of \cref{prop:benders}, $\BR_t$ the regularised operator of \cref{sec:ZS}, whose value functions $V^R_t$ differ from $V_t$, $\lambda$ and $\varrho$ the multiplier and the penalty
weight of the cut, $\sum_mp(l_\lambda+l_\varrho)$ Zhang and Sun's Lipschitz bound (\cref{prop:ZS}), $g$ the nodal value function of \cref{rem:AL}; for snapping and ``from the second iteration'', see \cref{sec:other-families}.}\label{tab:families}
\end{table}

\subsection{Generalised conjugacy cuts}\label{sec:ZS}

\citet{zhang2022stochastic} extend SDDP to nonconvex, mixed-integer nonlinear problems on a recombining scenario
tree, whose transition probabilities do not depend on the parent node, so that the noise is stagewise independent. The forward step of their Algorithm~3 solves a \emph{regularised} nodal problem, in which the incoming state
$\bar x$ may be replaced by any $z$ at the price of a penalty $\varpi\mathcal R(\bar x-z)$; the backward step returns a
generalised conjugacy cut, which is Lipschitz by construction. Under their exact-penalisation assumption, the
regularised and original problems have the same optimal value. In our terms, the forward pass is greedy for the
regularised Bellman operator $\BR_t$, with value functions $V^R_t:=\BR_t(V^R_{t+1})$, $V^R_T:=0$, so we apply
\cref{sec:main} with $\BR_t$ and $V^R_t$ in the roles of $\cB_t$ and $V_t$, and $G^k$ is the regularised gap. The
cuts are tight for $\BR_t$, but they lie below the value functions $V_t$ of the original problem: this is the
ceiling~\ref{O3} with $U=V\ne V^R$. By exact penalisation, $V^R_0(x_0)=V_0(x_0)$, so $\zeta_0=0$. \Cref{prop:ZS} checks \cref{ass:S,ass:O}, and
\cref{app:ZS} gives the dictionary between the two notations. 

Let $\iota$ be the first index $i$ such that the root solution of iteration $i+1$ is $\varepsilon$-optimal in
their sense, the quantity bounded in their Theorem~3; our horizon $T$ is theirs plus one (\cref{app:ZS}).

\begin{theorem}[A bound polynomial in $T$ for Algorithm~3 of \citeauthor{zhang2022stochastic}]\label{thm:ZS}
Consider Algorithm~3 under their Assumptions~1--3 and the parameter bounds of \cref{prop:ZS}, with scenario
paths sampled stage by stage from the transition probabilities, independently across stages and iterations
(for uniform transition probabilities, their Assumption~4 implies this). Assume that, for $1\le s<T$, the
approximations of index $s$ are $L$-Lipschitz, $X_s$ has dimension at most $d$ and lies in a ball of
diameter $D$, nodal costs lie in $[0,\cb]$, and $V_s\le(T-s)\cb$ on $X_s$. Then
\begin{equation}\label{eq:inst:zs_bound_uniform}
\EE[\iota]\le\cb\,T(T+1)\,\big(1+8LD(T-1)/\varepsilon\big)^d/\varepsilon .
\end{equation}
\end{theorem}

In their Remark~3, Zhang and Sun ask ``whether there exists a complexity bound for Algorithm~3 that is polynomial
in $T$ in general''; \cref{thm:ZS} gives one, in their setting and accuracy notion, when the paths are drawn
from the transition law. \Cref{prop:ZSgeneral} gives the general form, with stage-dependent constants and a
high-probability version, and compares it with their deterministic bound. With uniform sampling on trees whose
transitions are not uniform, our argument only gives a bound with a likelihood factor exponential in $T$, and the
question remains open (\cref{rem:likelihood}).

\subsection{SDDiP: Lagrangian and integer optimality cuts on binary states}\label{sec:SDDiP}

\begin{proposition}[SDDiP]\label{prop:SDDiP}
Consider SDDiP \citep{zou2019sddip} with binary states $X_t\subset\{0,1\}^{n_t}$, mixed-integer linear
nodal problems with compact feasible sets, solved exactly, complete continuous recourse, losses in $[0,\cb]$, $\lv{}0=0$, and forward scenarios sampled from the tree's
distribution. Use, at each index $t<T$, a Lagrangian cut with an optimal multiplier or an integer optimality
cut; other valid cuts, such as strengthened Benders cuts, may be added. Then
\cref{thm:master}(c) applies with $U_t=V_t$ and $\Gamma=0$:
\begin{equation}\label{eq:inst:sddip_bound}
\EE\sum_{k\ge1}G^k\ \le\ \cb\sum_{t<T}(T-t)|X_t|\ \le\ 2^{n}\cb\,\frac{T(T+1)}2,\qquad n:=\max_{t<T}n_t,
\end{equation}
and, for $\delta\in(0,1]$, with probability at least $1-\delta$, $\sum_kG^k\le9\ln(3/\delta)\,2^{n-1}\cb\,T(T+1)$.
\end{proposition}

To our knowledge, these are the first random-sampling rates for SDDiP that are polynomial in $T$
(\citealt{zou2019sddip} prove finite convergence without a rate); they are linear in the number $|X_t|\le2^n$ of
states, hence exponential in the number $n$ of binary state variables.

For \emph{non-binary} mixed-integer states, Lagrangian cuts are not tight in general (Zou et al., Example~1):
\ref{O2} holds only with $\gamma_t$ equal to the duality gap. Nonlinear cuts restore tightness there: ReLU
Lagrangian cuts \citep{deng2024relu,bansal2026normalization}, used in a convergent multistage method by
\citet{bansal2026sddmip}, and Lagrangian cuts on an iteratively lifted binary state space \citep{yang2025globally}.
These works prove convergence without a rate; whether such cuts admit uniform Lipschitz constants, which would give
\ref{O1}, is open.

\subsection{Other families}\label{sec:other-families}

TFDP's generic cuts, Lipschitz, valid and $\gamma_t$-tight, satisfy \cref{ass:O} from the second iteration on:
TFDP starts from $-\infty$, so the analysis takes the first cuts as initial approximations (\cref{prop:TFDP});
the resulting rate is polynomial in $T$, unlike TFDP's (\cref{tab:compare}). SLDP's reverse-norm cuts satisfy it whenever their constant $L_t$ is at least the Lipschitz constant of the true
$V_t$ (\cref{prop:sldp}). Sampled SLDP and MIDAS \citep{philpott2020midas} replace a forward state by an earlier one within a
distance $\omega>0$ (``snapping''); the forward process is then not a trajectory of a policy, \cref{lem:gap} does
not apply, and these versions are not covered. Without snapping, SLDP is covered by \cref{thm:master}, and MIDAS's
non-Lipschitz step cuts by a monotone variant of the volume argument (\cref{thm:steps}).

\section{Noise that is not finitely supported}\label{sec:noise}

Nothing in \cref{sec:main} uses finite support. The gap identity (\cref{lem:gap}) involves the forward law only;
the backward pass enters only through tightness~\ref{O2} and the ceiling~\ref{O3}. We discuss four ways to obtain
such cuts when $\cB_t$ is an integral over a continuous law: exact cuts under structure, valid inexact cuts,
backward passes on a fixed SAA, and sampled cuts. In all of them the forward paths are drawn from the true law.
Noise with unbounded support adds a tail term (\cref{prop:unbounded}). Proofs are in \cref{app:noise}.

\subsection{Exact cuts for the true law}\label{sec:route1}

If the cut added at $x$ is an $L_t$-Lipschitz affine minorant of $\cB_t(\lv{t+1}{k})$ on
$X_t$ that is exact at $x$, for the true, possibly continuous law, then \cref{lem:valid}
gives \cref{ass:O} with $U=V$ and $\gamma=0$, and \cref{thm:master} and its corollaries hold
verbatim, with $V^{\pi^k}_0$, $V$ and $G$ computed under the true law. Expected Benders cuts are such cuts whenever
the expectations of the scenario values and slopes can be computed exactly (\cref{prop:expbenders}), for instance
for linear stages and a piecewise-polynomial density (\cref{rem:linexact}; see also \citealp{forcier2022generalized}); the per-iteration cost, which grows with
the number of cuts, is not bounded by our analysis. For laws without closed-form integrals, numerical quadrature
falls under \cref{sec:route2} if its error is one-sided and bounded, and under \cref{sec:saa} otherwise.

\subsection{Valid inexact cuts}\label{sec:route2}

Valid cuts that are $\bgamma^k_t$-tight at the trial point, as in TFDP \citep{forcier2023trajectory}, satisfy
\cref{ass:O} with $U=V$, so \cref{thm:master} applies with the error floor $\bar{\bGamma}_k$, which cannot be lowered in general (\cref{rem:floor}).

The most direct valid inexact cut is a lower (Jensen) discretisation (\cref{prop:jensen}): average exact affine
scenario cuts at the conditional means of the cells of a partition $\cP$ of $\Xi_{t+1}$. If the scenario value
$\xi\mapsto\Bd_t(W;y,\xi)$ of~\eqref{eq:setting:bellman_operator} is convex and $L^\xi_t$-Lipschitz for the convex
models $W$ (the class $\cW^{\rm c}_{t+1}$ of \cref{app:jensen}, which contains every $\lv{t+1}k$ when $\lv{}0$ is
convex and the cuts are affine and valid), this cut is valid for $\cB_t(W)$ and $\gamma$-tight with
$\gamma\le L^\xi_te(\cP)$, where $e(\cP):=\EE\|\bxi_{t+1}-\EE[\bxi_{t+1}\mid\cP]\|$. Jensen-type validity
also underlies backward resampling \citep{cruz2026backward}, which changes the backward discretisation across
iterations while keeping lower bounds for the continuous law; the bound on $\gamma$ is what our analysis needs
from such schemes.

Cuts from
dual-feasible approximate solutions \citep{guigues2020inexact} are also valid and inexact. Errors that decrease along a refinement schedule, with a deterministic bound $\bar{\bGamma}_k\le g_k$,
are handled by \cref{cor:anytime}(a).

\subsection{Forward passes from the true law, backward passes on a fixed SAA}\label{sec:saa}

This is common practice: the policy is trained on a discretisation and simulated under the original law
\citep{lan2023numerical}. Unlike the analyses of SDDP on the SAA problem itself \citep{lan2022complexity} and of the
SAA size needed as a function of $T$ \citep{shapiro2006complexity,park2022datadriven}, the result below is an
iteration bound for the original problem.

Let $\Ph_t$ be
a fixed law on $\Xi_t$, for instance the empirical law of $N$ i.i.d.\ draws made before the run;
everything below is conditional on $\Ph=(\Ph_t)_t$. Let
$\cBh_t:=\cB^{\Ph_{t+1}}_t$, let $\Vh_t$ be the value functions of the problem with laws $\Ph_t$, and
assume that the cut added at $x$ is an $L_t$-Lipschitz minorant of $\cBh_t(\lv{t+1}{k})$ on
$X_t$ that is exact at $x$, that $\lv t0$ is $L_t$-Lipschitz with $\lv t0\le\cBh_t(\lv{t+1}0)$, and that
$\hat M_t:=\sup_{X_t}(\Vh_t-\lv t0)<\infty$. Two one-sided biases are needed:
$\Delta^+_t$ controls the lift and $\Delta^-_t$ controls the policy. Define
\begin{align}
\cW_{t+1}&:=\{W:X_{t+1}\to\R\ \text{$L_{t+1}$-Lipschitz},\ \lv{t+1}0\le W\le\Vh_{t+1}\}\qquad(\cW_T=\{0\}), \label{eq:noise:saa_test_functions} \\
\Delta^+_t&:=\sup_{x\in X_t,\,W\in\cW_{t+1}}\big(\cB_t(W)-\cBh_t(W)\big)(x),\qquad
\Delta^-_t:=\sup_{x\in X_t}\big(\cBh_t(V_{t+1})-\cB_t(V_{t+1})\big)(x). \label{eq:noise:saa_biases}
\end{align}

\begin{theorem}[SAA backward passes, true-law forward passes]\label{thm:saa}
Let \cref{ass:S} hold for both laws, with $\Xi_t$ carrying $\PP_t$ and $\Ph_t$, and assume~\eqref{eq:setting:volume_growth}. Let
$G^k=V^{\pi^k}_0(x_0)-\lv0{k-1}(x_0)$, where $V^{\pi^k}_0$ is the expected cost under the \emph{true} laws, and let $\Psi_T$ be the aggregated lift function of
\cref{sec:lemmas} and $\bGamma^F_k$ the forward-error sum~\eqref{eq:main:error_sums}. Then,
conditionally on $\Ph$,
\begin{equation}\label{eq:noise:saa_master}
\EE\Big[\sum_{k\ge1}\Psi_T\Big(\Big(G^k-\bGamma^F_k-\sum_{t<T}(\Delta^+_t)_+\Big)_+\Big)\Big]\le\sum_{t<T}\hat M_t,
\end{equation}
the high-probability statement of \cref{thm:master}(a) holds with $\hat M_t$ in place of $M_t$, and,
for every $k\ge1$,
\begin{equation}\label{eq:noise:saa_accuracy}
V_0(x_0)-\lv0{k-1}(x_0)\le G^k,\qquad V^{\pi^k}_0(x_0)-V_0(x_0)\le G^k+\sum_{t<T}(\Delta^-_t)_+ .
\end{equation}
\end{theorem}

\begin{proof}[Proof sketch]
Validity for $\cBh$ (\cref{lem:valid}(a) for the $\Ph$-problem) gives~\ref{O3} with ceiling $U=\Vh$,
and $\lv{t+1}{k-1}\in\cW_{t+1}$. At $x=\bx^k_t$, the cut satisfies
$f^k_t(x)=\cBh_t(\lv{t+1}{k})(x)\ge\cBh_t(\lv{t+1}{k-1})(x)\ge\cB_t(\lv{t+1}{k-1})(x)-\Delta^+_t$, which
is~\ref{O2} with the deterministic error $(\Delta^+_t)_+\ge0$. Apply \cref{thm:master}(a). The policy bound
is in \cref{app:saa}.
\end{proof}

The SAA thus enters only through an additive error per stage and the ceiling $U=\Vh$, with no likelihood ratio
between the laws of whole paths.

\begin{corollary}[Wasserstein and total-variation biases]\label{cor:wasserstein}
\begin{enumerate}[label=(\alph*),nosep]
\item Let $\Xi_{t+1}$ be a subset of a normed space and $\Wa$ the Wasserstein-1 distance for its norm.
If, for all $x\in X_t$ and $W\in\cW_{t+1}\cup\{V_{t+1}\}$, the map
$\xi\mapsto \Bd_t(W;x,\xi)$ is $L^\xi_t$-Lipschitz on $\Xi_{t+1}$, then
$\Delta^\pm_t\le L^\xi_t\,\Wa(\PP_{t+1},\Ph_{t+1})$. Hence, with $\Delta_{\rm W}:=\sum_{t<T}L^\xi_t\Wa(\PP_{t+1},\Ph_{t+1})$,
\begin{equation}\label{eq:noise:saa_wasserstein}
\EE\Big[\sum_{k\ge1}\Psi_T\big((G^k-\bGamma^F_k-\Delta_{\rm W})_+\big)\Big]\le\sum_{t<T}\hat M_t,\qquad
V^{\pi^k}_0(x_0)-V_0(x_0)\le G^k+\Delta_{\rm W} .
\end{equation}
\item If, for these $x$ and $W$, $\Bd_t(W;x,\cdot)$ takes values in an interval of length $\mathrm{rg}_t$,
then $\Delta^\pm_t\le\mathrm{rg}_t\,\TV(\PP_{t+1},\Ph_{t+1})$, where $\TV(\PP,\Ph):=\sup_A|\PP(A)-\Ph(A)|$.
\end{enumerate}
\end{corollary}

The Lipschitz condition in~(a) holds, for instance, with additive noise in the dynamics and losses Lipschitz in
the next state and the noise (\cref{app:saa}), and in the inventory model of \cref{ex:absorb} (\cref{app:absorb}). For an i.i.d.\ SAA of size $N$, $\Wa(\PP_{t+1},\Ph_{t+1})$ concentrates
around its mean (\cref{app:saa}), which is at most of order $N^{-1/q}$ for a bounded noise of dimension $q\ge3$
\citep[Thm.~1]{fournier2015rate}: the bias $\Delta_{\rm W}$ is linear in $T$
under stage-uniform constants, but the rate $N^{-1/q}$ suffers from the curse of dimensionality in $q$. For state sets in affine spaces of dimension at most $d$, a
covering argument over the state and the approximation replaces $q$ by $d$: $\Delta^\pm_t=O(N^{-1/(d+2)}\sqrt{\ln N})$ (\cref{prop:uniformSAA}).

Part~(b) is useful when the two laws share their atoms, for instance for an SAA of a finitely supported law; it is
uninformative for the empirical SAA of an atomless law, whose total variation distance to that law is one.
Some bias is unavoidable, and with forward passes drawn from the SAA itself we have no guarantee for the true
problem (\cref{rem:bias}).

\subsection{Sampled cuts}\label{sec:sampled}

Cuts built from a fresh sample at each backward step are valid only in expectation. Shifting each one down by the
confidence radius $b_{k,t}:=(\mathrm{rg}_t+2L_tD_t)\sqrt{\ln(4/\delta_{k,t})/N_{k,t}}$, where $N_{k,t}$ is the sample
size, $\mathrm{rg}_t$ bounds the range in $\xi$ of the scenario values and $\sum_{k,t}\delta_{k,t}=\delta'$, makes all cuts valid on an event $A$ of probability at least
$1-\delta'$. On $A$, a stopped-potential argument gives the bounds of \cref{thm:master}(a), in expectation and with
high probability, with the error floor raised by $2\sum_{t<T}b_{k,t}$, the potential bound
$\sum_{t<T}\sup_{X_t}(V_t-\lv t0)$ and constants changed by terms of order
$\delta'$ (\cref{thm:sampled}). Sample sizes of order $g^{-2}\log k$ give a floor $g$, whatever the noise dimension.

No unconditional expectation bound follows: off $A$ an invalid cut is never removed and can exceed $V_t$, so the
ceiling~\ref{O3} can fail. Raw sampled cuts can fail it too: without the shift, SDDP can stall forever on a
$0.1$-suboptimal policy with $G^k=0$ (\cref{prop:counter}).

\section{Infinite horizon}\label{sec:infinite}

\Cref{cor:disc} is horizon-free, but its algorithm is non-stationary, with one approximation per stage.
Infinite-horizon and periodic versions of SDDP keep a single cut model shared by all times
\citep{warrington2019generalized,shapiro2020periodical,dowson2020policy,nannicini2021benders}, and
\citet{julan2023infinite} give the first complexity bounds for them. We analyse the shared-model algorithm for an
abstract oracle and general noise, as \cref{sec:setting,sec:main} do in finite horizon. Two things change: the gap
identity runs over a geometric horizon (\cref{lem:gapIH}), and all the cuts of a pass lift the same model, so a pass
is credited only with its largest residual (\cref{lem:pass}). For SDDP, the main instance, we also give Lipschitz
constants uniform in the discount factor (\cref{sec:ih-sddp}). Proofs are in \cref{app:infinite}.

\subsection{Stationary problem, shared-model algorithm and oracle}\label{sec:ih-setting}

A stationary problem has a state set $X\subset\R^n$, an initial state $x_0\in X$, i.i.d.\ noises $(\bxi_t)_{t\ge1}$
with law $\PP$ on a measurable space $\Xi$, a next-state map $\cX$, a loss $\ell$ and a discount factor
$\beta\in(0,1)$. For bounded $W$ on $X$, $x\in X$ and $\xi\in\Xi$, let
\begin{equation}\label{eq:inf:bellman_operator}
\Bd(W;x,\xi):=\inf_{y\in\cX(x,\xi)}\big[\ell(x,y,\xi)+\beta W(y)\big],\qquad
\cB W(x):=\int\Bd(W;x,\xi)\,\PP(d\xi),
\end{equation}
the operators of \cref{sec:problem} with the discount inside; $\cB W$ is defined when $\xi\mapsto\Bd(W;x,\xi)$ is
measurable, \eg on the class $\mathfrak M$ of \cref{app:ih-std}.

\begin{assumption}[Stationary problem]\label{ass:Sinf}
For all $x\in X$ and $\xi\in\Xi$:
\begin{enumerate}[label=(S\arabic*$_\infty$),ref=(S\arabic*$_\infty$),nosep]
\item\label{SI1} $X$ is compact and $\emptyset\ne\cX(x,\xi)\subset X$ (forward invariance).
\item\label{SI2} $\ell$ is Borel and $0\le\ell(x,y,\xi)\le\cb$ for $y\in\cX(x,\xi)$.
\item\label{SI3} (Measurability.) \Cref{app:measurability} holds for~\eqref{eq:inf:bellman_operator} and the models below,
with parameters (possibly including $\gamma^F_k$) in a countable disjoint union of Euclidean spaces, as the number of
cuts is random; $\cB$ has a bounded Borel fixed point $V$ with $\Bd(V;\cdot,\cdot)$ jointly measurable.
\end{enumerate}
\end{assumption}

A policy is a nonanticipative Borel rule $\bx_{t+1}=\pi_t(\bxi_{[t+1]})\in\cX(\bx_t,\bxi_{t+1})$, \emph{stationary} if
$\bx_{t+1}=\pi(\bx_t,\bxi_{t+1})$; its cost is $V^{\pi}(x_0):=\EE\sum_{t\ge0}\beta^t\ell(\bx_t,\bx_{t+1},\bxi_{t+1})$.
On $\mathfrak M$, $\cB$ is monotone and a $\beta$-contraction, $0\le V\le\cb/(1-\beta)$, and $V^\pi(x_0)\ge V(x_0)$ for
every policy (\cref{lem:std}); \cref{app:ih-std} also gives a sufficient condition for~\ref{SI3}, which is not needed
for finitely supported noise and deterministic oracles.

\paragraph{Shared-model algorithm.} The algorithm keeps one model $\uV{k-1}:X\to\R$. Let $\cA_{k-1}$ be the
$\sigma$-field generated by everything drawn during the first $k-1$ iterations: noises, pass lengths and any internal
randomness of the oracles, which must be independent of the noise and of the length of the current pass. The policy
$\pi^k$ is the stationary policy of a forward operator as in~\eqref{eq:setting:forward_error}: it selects
$y\in\cX(x,\xi)$ with $\ell(x,y,\xi)+\beta\uV{k-1}(y)\le\Bd(\uV{k-1};x,\xi)+\gamma^F_k$, for an $\cA_{k-1}$-measurable
forward error $\gamma^F_k\ge0$. Iteration $k\ge1$:
\begin{itemize}[nosep]
\item \emph{Pass length.} Variant~(D): $\tau_k:=H$, fixed. Variant~(G): $\tau_k$ geometric,
$\PP(\tau_k>s)=\beta^s$ for $s\ge0$, independent of the past and of the noise of the pass
(so $\EE\tau_k=1/(1-\beta)$).
\item \emph{Forward pass.} $\bx^k_0:=x_0$ and $\bx^k_{t+1}:=\pi^k(\bx^k_t,\bxi^k_{t+1})$ for
$t<\tau_k-1$, with fresh i.i.d.\ noise and $\pi^k$ frozen.
\item \emph{Backward pass.} $W:=\uV{k-1}$; for $t$ from $\tau_k-1$ down to $0$, an oracle returns a cut $f:X\to\R$
at $\bx^k_t$ for the \emph{current} $W$, and $W:=\max(W,f)$, possibly with further cuts (\ref{OI4}). Then $\uV k:=W$.
\end{itemize}
For comparison, variant~(N) is finite-horizon SDDP on the $H$-stage truncation, with stage losses $\beta^{t-1}\ell$
and one model per stage (\cref{sec:trunc}).
The gap is $G^k:=V^{\pi^k}(x_0)-\uV{k-1}(x_0)$, where $V^{\pi^k}(x_0)$ is the cost of $\pi^k$ given $\cA_{k-1}$, and the
residual is $r^k:=\cB\uV{k-1}-\uV{k-1}$.

\begin{assumption}[Lipschitz cut oracle, shared model]\label{ass:Oinf}
There are a metric $\dist$ on $X$, possibly extended, a constant $\Lh\ge0$, a ceiling
$U:X\to(-\infty,+\infty]$ and $\cA_{k-1}$-measurable errors $\gamma_{k,t}\ge0$ ($k\ge1$, $t\ge0$) such that, for
all $k\ge1$:
\begin{enumerate}[label=(O\arabic*$_\infty$),ref=(O\arabic*$_\infty$),nosep]
\item\label{OI1} $\uV0$ and every cut are Borel and $\Lh$-Lipschitz for $\dist$, and $\uV0$ is deterministic and bounded.
\item\label{OI2} The cut $f$ returned at $x=\bx^k_t$, $t<\tau_k$, satisfies $f(x)\ge\cB\uV{k-1}(x)-\gamma_{k,t}$.
\item\label{OI3} $U$ is deterministic, and $\uV0$ and every cut are below $U$ on $X$; the \emph{potential bound} $M_U:=\sup_X(\bar U-\uV0)$ is
finite, where $\bar U(x):=\inf_{z\in X}\{U(z)+\Lh\dist(x,z)\}$, with $\Lh\cdot\infty:=\infty$.
\item\label{OI4} Every update is $W:=\max\big(\{W,f\}\cup\{f'_i\}_{i\in[m]}\big)$ for some $m\ge0$, where the
additional cuts $f'_i$, if any, satisfy~\ref{OI1} and~\ref{OI3}.
\end{enumerate}
\end{assumption}

This is \cref{ass:O} for a single model; the errors are fixed before the pass because the proof needs them
independent of $\tau_k$. As in \cref{lem:basic}(i), the models are Borel, $\Lh$-Lipschitz and bounded, and
$\uV0\le\uV{k-1}\le W\le\uV k\le\bar U$ for every intermediate model $W$ of pass $k$. As in \cref{lem:valid}, cuts
below $\cB W$ for the current $W$, with $\uV0\le\cB\uV0$, give~\ref{OI3} with $U=V$ and $r^k\ge0$, and cuts tight
for the current $W$ satisfy~\ref{OI2} (\cref{lem:validIH}).

\paragraph{What the gap measures.} Let $\zeta:=(U(x_0)-V(x_0))_+$, which is $0$ if $U=V$. By \cref{lem:std}(iii)
and~\ref{OI3}, for every $k\ge1$,
\begin{equation}\label{eq:inf:gap_meaning}
G^k\ \ge\ V(x_0)-\uV{k-1}(x_0)\qquad\text{and}\qquad G^k\ \ge\ V^{\pi^k}(x_0)-V(x_0)-\zeta .
\end{equation}
So, when $\zeta=0$, $G^k\le\varepsilon$ means that the \emph{stationary} policy $\pi^k$ is $\varepsilon$-optimal from
$x_0$ and the lower bound $\varepsilon$-tight; this implies Ju and Lan's guarantee, $\varepsilon$-optimality of the
first decision (\cref{lem:std}(iii)).

\paragraph{Random horizon.} If $\tau\in\mathbb N^*:=\mathbb N\setminus\{0\}$ is independent of the noise and
$\PP(\tau>t)=\beta^t$, the undiscounted loss paid before $\tau$ has expectation $V^{\pi}(x_0)$
\citep{puterman1994markov}; so variant~(G) is Dowson's passes on a one-node policy graph with a self-loop of
probability $\beta$ \citep{dowson2020policy}.

\subsection{Gap identity and lift of a pass}\label{sec:ih-lemmas}

Let $\hat\rho:=4\Lh D$, with $D$ and $d$ the constants of the volume-growth condition~\eqref{eq:setting:volume_growth}
on $X$, let $\hat\varphi$ and $\hat\psi$ be the functions $\varphi_t$ and $\hat\psi_t$ of \cref{sec:lemmas} with
$\rho_t=\hat\rho$ and $d_t=d$, and write $\EE_{k-1}:=\EE[\cdot\mid\cA_{k-1}]$.

\begin{lemma}[Gap identity]\label{lem:gapIH}
Under \cref{ass:Sinf}, let $\uV{k-1}$ be bounded, let $(\bx_t)_{t\ge0}$ be the trajectory of $\pi^k$ from $x_0$
driven by i.i.d.\ noise independent of $\cA_{k-1}$, $\lambda_t:=\ell(\bx_t,\bx_{t+1},\bxi_{t+1})$, and
$e_t:=\lambda_t+\beta\uV{k-1}(\bx_{t+1})-\Bd(\uV{k-1};\bx_t,\bxi_{t+1})\in[0,\gamma^F_k]$ the forward error. For every $H\ge0$,
\begin{equation}\label{eq:inf:gap_identity}
G^k=\EE_{k-1}\Big[\sum_{t<H}\beta^t\big(r^k(\bx_t)+e_t\big)\Big]+\varepsilon^k_H,\qquad
\varepsilon^k_H\le\beta^H\Big(\frac{\cb}{1-\beta}-\inf_X\uV{k-1}\Big),
\end{equation}
where $\varepsilon^k_H:=\EE_{k-1}[\sum_{t\ge H}\beta^t\lambda_t]-\beta^H\EE_{k-1}\uV{k-1}(\bx_H)$ is what a pass of length
$H$ misses. Hence $G^k=\EE_{k-1}\sum_{t\ge0}\beta^t(r^k(\bx_t)+e_t)$, and, if $\tau\in\mathbb N^*$ is independent of the path
given $\cA_{k-1}$, with $\PP(\tau>t\mid\cA_{k-1})=\beta^t$ for all $t\ge0$, then exactly $G^k=\EE_{k-1}\big[\sum_{t<\tau}(r^k(\bx_t)+e_t)\big]$.
\end{lemma}

\begin{lemma}[A pass on a shared model lifts by its largest residual]\label{lem:pass}
Let \cref{ass:Sinf,ass:Oinf} hold and~\eqref{eq:setting:volume_growth} hold on $X$, for $\dist$, with constants
$(\mu,\kappa,D,d)$. Let $\Phi^k:=\int(\uV k-\uV0)\,d\mu$ and $s^k_t:=(r^k(\bx^k_t)-\gamma_{k,t})_+$. In variants~(D)
and~(G), for every $k\ge1$,
\begin{equation}\label{eq:inf:pass_lift}
\Phi^k-\Phi^{k-1}\ \ge\ \kappa\max_{t<\tau_k}\hat\varphi\big(s^k_t\big)\ \ge\ \kappa\hat\psi\big(\bar s^k\big),
\qquad \bar s^k:=\max_{t<\tau_k}s^k_t,
\end{equation}
and $0=\Phi^0\le\Phi^{k-1}\le\Phi^k\le\Mb_U:=\int(\bar U-\uV0)\,d\mu\le M_U$.
\end{lemma}

Only the maximum is available because the regions lifted during a pass can coincide, \eg if the path stays at one
point. A summation by parts (\cref{lem:length}) turns this maximum back into the discounted sum of \cref{lem:gapIH}:
for $\tau$ independent of bounded $s_t\ge0$, $\EE\max_{t<\tau}s_t\ge(1-\beta)\sum_t\beta^ts_t$ for all such $s$ if and
only if $\PP(\tau>t)\ge\beta^t$ for all $t$, and among these laws the geometric one has the fewest expected cuts.

\subsection{The master theorem in infinite horizon}\label{sec:ih-master}

Let $\Gamma^\beta_k:=\sum_{t\ge0}\beta^t\gamma_{k,t}+\gamma^F_k/(1-\beta)$, the \emph{error floor}, equal to
$(\gamma+\gamma^F)/(1-\beta)$ for constant errors. As in \cref{thm:master}, $\int(\bar U-\uV0)\,d\mu$ is an upper
integral if $\bar U$ is not Borel.

\begin{theorem}[Master theorem, infinite horizon]\label{thm:master-ih}
Let \cref{ass:Sinf,ass:Oinf} hold, and run variant~(G).
\begin{enumerate}[label=(\alph*),nosep]
\item \emph{Volume version.} If~\eqref{eq:setting:volume_growth} holds on $X$ as in \cref{lem:pass}, then
\begin{equation}\label{eq:inf:master_volume}
\EE\sum_{k\ge1}\kappa\,\hat\psi\big((1-\beta)(G^k-\Gamma^\beta_k)_+\big)\ \le\ \Mb_U\ \le\ M_U .
\end{equation}
In particular, if $\Gamma^\beta_k\le\Gamma^\beta$ for a deterministic $\Gamma^\beta$, then, for every $\theta>0$,
$\EE\,\#\{k\ge1:G^k>\theta+\Gamma^\beta\}\le\Mb_U/(\kappa\hat\psi((1-\beta)\theta))$, which equals
$2(d+1)\Mb_U\hat\rho^d/(\kappa(1-\beta)^{d+1}\theta^{d+1})$ when $(1-\beta)\theta\le\hat\rho$.
Moreover, for every $\delta\in(0,1]$, with probability at least $1-\delta$, simultaneously for all $\theta>0$,
$\#\{k\ge1:G^k>\theta+\Gamma^\beta_k\}\le9\ln(3/\delta)\,\Mb_U/(\kappa\hat\psi((1-\beta)\theta))$.
\item \emph{Covering version.} Fix $\theta>0$, $\sigma\ge0$ with $4\Lh\sigma\le(1-\beta)\theta$, and a finite
$\sigma$-net $C$ of $X$ for $\dist$, and let $M_C:=\sum_{c\in C}(\bar U-\uV0)(c)\le|C|M_U$. Then
\begin{equation}\label{eq:inf:master_covering}
\EE\sum_{k\ge1}\big(G^k-\Gamma^\beta_k-\theta/2\big)_+\ \le\ \frac{M_C}{1-\beta}.
\end{equation}
In particular, $\EE\,\#\{k\ge1:G^k>\theta+\Gamma^\beta_k\}\le2M_C/((1-\beta)\theta)$. Moreover, for every
$\delta\in(0,1]$, with probability at least $1-\delta$, $\#\{k\ge1:G^k>\theta+\Gamma^\beta_k\}\le18\ln(3/\delta)\,M_C/((1-\beta)\theta)$.
\item \emph{Finite state sets.} If $X$ is finite and $\dist=+\infty$ off the diagonal, then
$\EE\sum_{k\ge1}(G^k-\Gamma^\beta_k)_+\le C_{\rm fin}/(1-\beta)$, with $C_{\rm fin}:=\sum_{x\in X}(U-\uV0)(x)$. Moreover,
for every $\delta\in(0,1]$, with probability at least $1-\delta$,
$\sum_{k\ge1}(G^k-\Gamma^\beta_k)_+\le9\ln(3/\delta)\,C_{\rm fin}/(1-\beta)$.
\end{enumerate}
\end{theorem}

The proof (\cref{app:periodic,app:ih-cover}) follows that of \cref{thm:master}. We call~\eqref{eq:inf:master_volume}
the \emph{convex-regret form}; with $U=V$ and $\uV0\ge0$, $\Mb_U\le\cb/(1-\beta)$, as in \cref{cor:disc}. In~(b), if
$X$ lies in an affine space of dimension $d$, with $\dist$ given by a norm and $\diam X\le D$, the choice
$\sigma=(1-\beta)\theta/(4\Lh)$ gives $|C|\le(1+8\Lh D/((1-\beta)\theta))^d$ (\cref{lem:nets}(a); one point if
$\Lh=0$). Since $\tau_k$ is independent of $\cA_{k-1}$ and $\EE\tau_k=1/(1-\beta)$, the iterations counted in~(a)
and~(b) contain in expectation $1/(1-\beta)$ times as many trial points, each with one cut of~\ref{OI2}.

\begin{remark}[On the factor $(1-\beta)^{-(d+1)}$]\label{rem:tight}
In \cref{thm:master-ih}(a), one factor $(1-\beta)^{-1}$ comes from the potential and $(1-\beta)^{-(d+1)}$ from evaluating
$\hat\psi$ at $(1-\beta)\varepsilon$, that is, from the maximum over the pass. The factor $1-\beta$ inside $\hat\psi$ cannot be
removed from a bound of this convex-regret form (\cref{prop:tight}, for $d=1$): a value error of order $r_0$ can cause a
policy error $r_0/(1-\beta)$, the classical loss of greedy policies \citep{singh1994upper}. Whether the \emph{count}
of bad iterations can avoid it is open.
\end{remark}

\paragraph{Deterministic pass length.} Variant~(D) with $H:=\lceil\ln((d+2)\cb/((1-\beta)\theta))/(1-\beta)\rceil$,
$\uV0\ge0$, $\theta<\cb/(1-\beta)$ and a deterministic floor $\Gamma^\beta$ loses only a factor $e$: the tail of
\cref{lem:gapIH} is at most $\theta/(d+2)$, and the count of part~(a) holds with $2e(d+1)$ in place of $2(d+1)$,
for one $\theta$ at a time (\cref{thm:det}). Geometric lengths remove the logarithm in $H$, the tail and the
dependence of the pass length on $\theta$.

\subsection{Periodic problems}\label{sec:periodic}

Seasonal problems, such as hydro-thermal scheduling with weekly stages and an annual cycle, are periodic
\citep{shapiro2020periodical}. Let $S\ge1$ be the period, $\mathbb Z_S:=\mathbb Z/S\mathbb Z$ the set of \emph{phases},
and $p(t):=t\bmod S$. Phase $p$ has its own data $X_p$, $\PP_p$ on $\Xi_p$, $\cX_p(x,\xi)\subset X_{p+1}$ and
$\ell_p$, an operator $\cB_p$ as in~\eqref{eq:inf:bellman_operator}, and a model $\uV k_p$; the state $\bx_t$ lies
in $X_{p(t)}$, the noise $\bxi_{t+1}$ has law $\PP_{p(t)}$, and $x_0\in X_0$. At a time $t$ of phase $p$, the policy
$\pi^k$ is $\gamma^F_k$-greedy for $\uV{k-1}_{p+1}$, and the cut at $\bx^k_t$ is added to the model of phase $p$. The gap
$G^k:=V^{\pi^k}_0(x_0)-\uV{k-1}_0(x_0)$ is taken from phase~$0$, and the residuals are
$r^k_p:=\cB_p\uV{k-1}_{p+1}-\uV{k-1}_p$. We assume \cref{ass:Sinf,ass:Oinf} and~\eqref{eq:setting:volume_growth} phase
by phase, with common $\Lh$, $D$ and $d$, with $\cB_{p(t)}\uV{k-1}_{p(t)+1}$ in~\ref{OI2}, and with, in~\ref{SI3}, a
bounded Borel fixed point $V=(V_p)_p$ of the monotone $\beta$-contraction $(W_p)_p\mapsto(\cB_pW_{p+1})_p$, each
$\Bd_p(V_{p+1};\cdot,\cdot)$ being jointly measurable. The constants become $\kappa:=\min_p\kappa_p$,
$\Phi^k:=\sum_p\int(\uV k_p-\uV0_p)\,d\mu_p$, $\Mb_U:=\sum_p\int(\bar U_p-\uV0_p)\,d\mu_p$ and
$\zeta:=(U_0(x_0)-V_0(x_0))_+$; $S=1$ is the stationary case. \Cref{lem:gapIH,lem:pass} hold with the same proofs,
and \cref{lem:pass} improves, as each phase has its own model:
\begin{equation}\label{eq:inf:periodic_lift}
\Phi^k-\Phi^{k-1}\ \ge\ \kappa\sum_{p\in\mathbb Z_S}\hat\psi\big(\bar s^k_p\big),\qquad
\bar s^k_p:=\max\{(r^k_p(\bx^k_t)-\gamma_{k,t})_+:t<\tau_k,\ p(t)=p\}\quad(\max\emptyset:=0).
\end{equation}

\begin{theorem}[Periodic problems, geometric pass length]\label{thm:periodic}
Under \cref{ass:Sinf,ass:Oinf} and~\eqref{eq:setting:volume_growth} phase by phase, the periodic algorithm with
geometric pass lengths satisfies:
\begin{enumerate}[label=(\roman*),nosep]
\item $G^k=\EE_{k-1}\big[\sum_{t<\tau_k}(r^k_{p(t)}(\bx^k_t)+e_t)\big]$, with $e_t\in[0,\gamma^F_k]$, and
$G^k\ge\max\{V^{\pi^k}_0(x_0)-V_0(x_0)-\zeta,\,V_0(x_0)-\uV{k-1}_0(x_0)\}$;
\item $\EE\sum_{k\ge1}\kappa\,\chi_S\big((G^k-\Gamma^\beta_k)_+\big)\le\Mb_U$, with
$\chi_S(g)\ge\hat\psi((1-\beta)g)$ defined in~\eqref{eq:appD:periodic_chi}; with $q_S:=S(1-\beta)/(1-\beta^S)$
and $\Gamma^\beta_k\le\Gamma^\beta$ deterministic, for every $\theta>0$ with $(1-\beta)\theta\le\hat\rho$,
\begin{equation}\label{eq:inf:periodic_rate}
\EE\,\#\{k\ge1:G^k>\theta+\Gamma^\beta\}\le\frac{2(d+1)\,\Mb_U\,\hat\rho^d}{\kappa\,(1-\beta)^{d+1}\theta^{d+1}}\min\Big\{1,\frac{q_S^{d+1}}{S}\Big\};
\end{equation}
\item for $\delta\in(0,1]$, with probability at least $1-\delta$, simultaneously for all $\theta>0$,
$\#\{k\ge1:G^k>\theta+\Gamma^\beta_k\}\le9\ln(3/\delta)\,\Mb_U/(\kappa\chi_S(\theta))$.
\end{enumerate}
\end{theorem}

\Cref{thm:master-ih}(a) is the case $S=1$, since $\chi_1(g)=\hat\psi((1-\beta)g)$; parts~(b) and~(c) hold with one net
per phase (\cref{app:ih-cover}). Since $1\le q_S\le\min\{S,1+S(1-\beta)\}$, for SDDP ($\Mb_U\le S\cb/(1-\beta)$,
$\kappa=1$, $\Gamma^\beta=0$) the bound~\eqref{eq:inf:periodic_rate} is at most
$2(d+1)\cb(4\Lh D)^d\min\{S,(1+S(1-\beta))^{d+1}\}/((1-\beta)^{d+2}\theta^{d+1})$. This price of periodicity, the
factor $\min\{S,q_S^{d+1}\}$ compared with $S=1$, is at most $S$. It is close to $1$ when $(d+1)S(1-\beta)$ is small,
since $q_S^{d+1}\le e^{(d+1)S(1-\beta)}$. For instance, weekly stages, an annual cycle and an annual discount rate of
$5\%$ give $S=52$ and $q_S<1.03$, so $q_S^{d+1}\approx1.05$ for $d=1$ but $1.3$ for $d=10$.

\subsection{SDDP}\label{sec:ih-sddp}

For SDDP with exact Benders cuts, \cref{thm:master-ih} gives \cref{thm:geom}, whose Lipschitz constant $\Lh$ we then
bound, uniformly in $\beta$ under a reachability condition (\cref{prop:lipiter}).

\begin{assumption}[Shared-model SDDP]\label{ass:I}
\begin{enumerate}[label=(I\arabic*),ref=(I\arabic*),nosep]
\item\label{I1} \Cref{ass:Sinf} holds, $X$ is nonempty and convex, $d:=\dim\aff X$, $D\ge\diam X$ (Euclidean norm), and
$\avg_X$ is the average for Lebesgue measure on $\aff X$ (the value at the point if $d=0$). For each $\xi$, the graph
$\{(x,y):x\in X,\ y\in \cX(x,\xi)\}$ is convex, $\cX(x,\xi)$ is closed, and $\ell(\cdot,\cdot,\xi)$ is
convex on the graph and lower semicontinuous in $y$.
\item\label{I2} $\uV0:X\to\R$ is convex and continuous, and $0\le\uV0\le\cB\uV0$ (\eg $\uV0=0$).
\item\label{I3} The forward pass is exact ($\gamma^F_k=0$), and a cut added at $x\in X$ to the current model $W$ is an
affine $f$ with $f(x)=\cB W(x)$ and $f\le\cB W$ on $X$ (on $\Oe$ under~\ref{Ieta}).
\item\label{I4} $\uV0$ and every cut are $\Lh$-Lipschitz on $X$.
\end{enumerate}
\end{assumption}

Condition~\ref{I1} makes the problem convex, and~\ref{I2}--\ref{I3} describe SDDP with exact Benders cuts computed
for the current model; the noise need not be finitely supported (\cref{sec:route1}). Condition~\ref{I4} assumes a
Lipschitz constant for the cuts, which \cref{prop:lipiter} derives, as \cref{prop:benders} does in finite horizon,
from extended relatively complete recourse:
\begin{enumerate}[label=(E$_\eta$),ref=(E$_\eta$),nosep,leftmargin=*]
\item\label{Ieta} For some $\eta>0$, with $\Oe:=\{z\in\aff X:\dist(z,X)\le\eta\}$, the conditions of~\ref{I1} on the
graph, $\cX$ and $\ell$ hold with $\Oe$ in place of $X$ for the first argument, and~\ref{SI1}--\ref{SI3} hold with
$\Oe$ in place of $X$ as the set of incoming states; in particular $\emptyset\ne \cX(z,\xi)\subset X$ and
$\ell(z,y,\xi)\in[0,\cb]$ for $z\in\Oe$.
\end{enumerate}
Condition~\ref{Ieta} is used only there; it is stronger than Assumption~2.2 of \citet{julan2023infinite}, because
successors of the neighbourhood must stay in $X$.

\begin{corollary}[Shared-model SDDP, geometric pass length]\label{thm:geom}
Under \cref{ass:I}, \cref{ass:Oinf} holds with $U=V$, $\gamma_{k,t}=0$, the Euclidean distance
and $\Lh$ of~\ref{I4}, and~\eqref{eq:setting:volume_growth} with $\kappa=1$ (\cref{lem:VG}(a)); moreover $r^k\ge0$,
$\zeta=0$ and $\Mb_U\le\Mb:=\avg_X(V-\uV0)\le\cb/(1-\beta)$. Hence variant~(G) satisfies:
\begin{enumerate}[label=(\roman*),nosep]
\item $G^k=\EE_{k-1}\big[\sum_{t<\tau_k}r^k(\bx^k_t)\big]$ and
$G^k\ge\max\{V^{\pi^k}(x_0)-V(x_0),\,V(x_0)-\uV{k-1}(x_0)\}\ge0$;
\item $\EE\sum_{k\ge1}\hat\psi((1-\beta)G^k)\le\Mb$, and, for every $\varepsilon>0$ with
$(1-\beta)\varepsilon\le\hat\rho$,
\begin{equation}\label{eq:inf:geometric_rate}
\EE\,\#\{k\ge1:G^k>\varepsilon\}\le\frac{2(d+1)\,\Mb\,\hat\rho^d}{(1-\beta)^{d+1}\varepsilon^{d+1}}
\le\frac{2(d+1)\,\cb\,(4\Lh D)^d}{(1-\beta)^{d+2}\varepsilon^{d+1}};
\end{equation}
\item for $\delta\in(0,1]$, with probability at least $1-\delta$, for all $\varepsilon>0$, $\#\{k\ge1:G^k>\varepsilon\}\le K_\varepsilon:=9\ln(3/\delta)\,\Mb/\hat\psi((1-\beta)\varepsilon)$
and $V(x_0)-\uV k(x_0)\le\varepsilon$ for all $k\ge K_\varepsilon$.
\end{enumerate}
\end{corollary}

\paragraph{Lipschitz constants.} Under~\ref{Ieta}, the cuts of SDDP have a crude Lipschitz constant of order $\cb/((1-\beta)\eta)$, as in
\cref{prop:benders}, which turns the bounds of \cref{cor:disc,thm:geom} into bounds of order
$(1-\beta)^{-(2d+2)}\varepsilon^{-(d+1)}$. A constant uniform in $\beta$ needs a reachability condition.

\begin{definition}[Robust reachability]\label{def:reach}
Let $m\ge1$ be an integer and $\chi\ge0$. Condition $(R_{m,\chi})$ holds if, for every $z\in\Oe$ and
every $y\in X$, there are a policy started at $z$ and a stopping time $\sigma$ of the noise
filtration such that $1\le\sigma\le m$ and $\|\bx_\sigma-y\|\le\chi$ almost surely. Condition $(R_{1,0})$
means $\cX(z,\xi)=X$ for all $z\in\Oe$ and $\xi$.
\end{definition}

\begin{proposition}[Lipschitz bounds for the shared model]\label{prop:lipiter}
Assume \ref{I1}--\ref{I3} and~\ref{Ieta}, and let $\uV0$ be $L_0$-Lipschitz on $X$. In variants~(D) and~(G):
\begin{enumerate}[label=(\roman*),nosep]
\item every cut is $\cb/((1-\beta)\eta)$-Lipschitz on $X$, so~\ref{I4} holds with
$\Lh=\max\{L_0,\cb/((1-\beta)\eta)\}$;
\item if $(R_{m,\chi})$ holds with $\chi<\eta$ and $L_0\le m\cb/(\eta-\chi)$ (\eg $\uV0=0$), every cut is
$m\cb/(\eta-\chi)$-Lipschitz, so~\ref{I4} holds with $\Lh=m\cb/(\eta-\chi)$, \emph{uniformly in $\beta$};
\item under $(R_{1,0})$, every cut is $\cb/\eta$-Lipschitz, so~\ref{I4} holds with $\Lh=\max\{L_0,\cb/\eta\}$.
\end{enumerate}
\end{proposition}

Under~\ref{Ieta}, \ref{I3} asks for cuts valid on $\Oe$, as subgradient cuts of the convex function $\cB W$ on $\Oe$
are (proof of \cref{thm:geom}). The $\beta$-uniform bound~(ii) uses the shared model: the slope of a cut of $\cB W$
is bounded by an oscillation of $\cB W$ in which $\inf W$ cancels, which fails for stage-wise models
(\cref{app:ih-lip}; \cref{rem:lipex} gives examples of $(R_{m,0})$). For periodic problems, \cref{prop:lipiter} holds
phase by phase; its bound~(ii) needs a reachability condition that returns to the right phase, which forces $m\ge S+1$
unless one step suffices.

\begin{corollary}[Relative accuracy, $\beta$-uniform constants]\label{cor:rel}
Under \ref{I1}--\ref{I3} and~\ref{Ieta} with $\uV0=0$ and $(R_{m,0})$, variant~(G) satisfies, for
$\erel\le4mD/\eta$,
\begin{equation}\label{eq:inf:relative_rate}
\EE\,\#\Big\{k:G^k>\erel\frac{\cb}{1-\beta}\Big\}\le\frac{2(d+1)\,(4mD/\eta)^d}{(1-\beta)\,\erel^{d+1}},
\end{equation}
with $1/(1-\beta)$ cuts per iteration in expectation. With the crude constant of
\cref{prop:lipiter}(i), the right-hand side is $2(d+1)(4D/\eta)^d/((1-\beta)^{d+1}\erel^{d+1})$, for
$\erel\le4D/((1-\beta)\eta)$.
\end{corollary}

\begin{proof}
\Cref{thm:geom}(ii) with $\Lh=m\cb/\eta$ (\cref{prop:lipiter}(ii)) and $\varepsilon=\erel\cb/(1-\beta)$; the
condition $(1-\beta)\varepsilon\le\hat\rho$ reads $\erel\le4mD/\eta$.
\end{proof}

This is the analogue of \cref{cor:uniform}, with $T$ replaced by the effective horizon $1/(1-\beta)$ and $\erel\cb$ in
the role of the per-stage error $\varepsilon$ ($\erel$ is a fraction of the largest cost $\cb/(1-\beta)$); for absolute
accuracy, the factor $(1-\beta)^{-(d+2)}$ of~\eqref{eq:inf:geometric_rate} matches the $T^{d+1}(T+1)$
of~\eqref{eq:main:rate_absolute}.

\subsection{Other oracles and non-stationary truncation}\label{sec:ih-oracles}\label{sec:trunc}

\paragraph{Other oracles.} The oracles of \cref{sec:instances,sec:noise} carry over to the shared model, with
the constants and the part of \cref{thm:master-ih} given in \cref{prop:ih-families}. Valid inexact cuts and SAA backward passes with forward
passes from the true law are covered (\cref{prop:ih-saa}), and so are shifted sampled cuts, through a stopped version
of the proof (\cref{prop:ih-sampled}); unbounded reachable sets are not. For SDDiP on binary states with exact forward
passes, part~(c) gives $\EE\sum_kG^k\le2^n\cb/(1-\beta)^2$. The $\beta$-uniform constants of \cref{sec:ih-sddp} use
convexity and affine cuts. For nonconvex oracles, a $\beta$-uniform constant is known only under stage-wise absorption
(\cref{ex:absorb}), where reverse-norm cuts with $\Lh=c$ qualify, and the regularisation constant of Zhang and Sun may have to grow with $1/(1-\beta)$.

\paragraph{Non-stationary truncation.} Variant~(N) gives, by \cref{cor:disc} and a tail $\varepsilon_H$, the rate of variant~(D) with $H$
cuts per iteration, but it guarantees only a non-stationary policy that uses $H$ functions (\cref{thm:trunc}). Using
the normalised time-1 approximation $\mathcal V^k_1:=\beta^{-1}\lv1k$ as a stationary model instead is not justified.
On a deterministic instance, its greedy policy stays $\beta^2/(4(1-\beta))$-suboptimal at every iteration, while the
truncated gap tends to zero for cut oracles with bounded slopes (\cref{prop:counter-ih}).

\section{Discussion}\label{sec:discussion}

\subsection{Comparison with the SDDP literature}\label{sec:compare-sddp}

\Cref{tab:compare} compares the bounds under stage-uniform constants. The accuracy notions differ; by
\cref{rem:compare}, $G^k\le\theta$ implies each of the others. Lan's corrected bound uses his own recursion for the accuracy, with a factor
$\prod_{t=2}^{T-1}N_t$ \citep{lan2022correction}. Iteration counts are comparable because an iteration costs
$\Theta(\sum_tN_t)$ subproblems in all rows, for the deterministic Algorithm~2 of \citet[Rem.~2]{zhang2022stochastic}
as for SDDP (one forward subproblem and $N_t$ backward subproblems per stage). The third row includes a factor $2^n$
absent from \citet[Thm.~4.11]{forcier2023trajectory}: their proof drops a $+1$ in the bound on the number of
effective iterations, as does that of their Cor.~4.7.

\begin{table}[tb]
\centering\small
\setlength{\tabcolsep}{3pt}
\begin{tabularx}{\textwidth}{@{}ll>{\raggedright\arraybackslash}X@{}}
\toprule
forward pass & iterations, accuracy $T\varepsilon$ & source\\
\midrule
problem child / explorative & $T(1+2LD/\varepsilon)^d$ & \citet[Cor.~2]{zhang2022stochastic}, \citet[Thm.~2]{lan2022complexity}, see also \citet[Cor.~4.7]{forcier2023trajectory}\\
uniform sampling, $N_t$ atoms & $O\big(T(1+2LD/\varepsilon)^d\prod_{t<T}N_t\big)$ & \citet[Thm.~3]{zhang2022stochastic}, \citet{lan2022correction}\\
sampling, general law & $2^n(T-1)\big(4DL(T-1)/(T\varepsilon-\gamma_\Sigma)\big)^{n+2(T-1)}$ & \citet[Thm.~4.11]{forcier2023trajectory}\\
sampling (this paper) & $(d+1)\,\cb\,(T+1)(4LD)^d/\varepsilon^{d+1}$ & \cref{cor:uniform}\\
\bottomrule
\end{tabularx}
\caption{Finite horizon, stage-uniform constants, each state set in a ball of diameter $D$. Accuracy $T\varepsilon$
in the sense of each source, which $G^k\le T\varepsilon$ implies (\cref{rem:compare}); for an absolute accuracy
$\varepsilon$, replace $\varepsilon$ by $\varepsilon/T$ in every row (\cref{rem:absolute}). Our bound counts all bad
iterations, the others bound the first good one; only the first row certifies its output. $\gamma_\Sigma$ is the
total error and $n$ the state dimension of \citet{forcier2023trajectory}.}\label{tab:compare}
\end{table}

\paragraph{Comparing the upper bounds.} Against the problem child, our bounds lose a factor of order $\cb T/\theta$
at the accuracy $\theta$, not a product over stages, whether the Lipschitz constants are uniform or grow like $T-t$
(\cref{tab:compare,rem:lipT,rem:absolute}); at $\theta=T\varepsilon$ with uniform constants, this factor is about
$(d+1)2^d\cb/\varepsilon$, independent of $T$. This ratio, which no known lower bound matches, comes from the potential: a region of $X_t$ can receive up to
$M_t/\theta$ lifts of size $\theta$, whereas in the packing argument a saturated point stays saturated forever.
For an absolute accuracy, linear dependence on $T$ is
not expected either: for $d\ge5$, even the deterministic Algorithm~2 of Zhang and Sun needs a number of iterations
superlinear in $T$, of order $(DL(T-1)/(8\varepsilon))^{(d-2)/2}$, to certify $\varepsilon$-optimality with its
stopping rule $\mathrm{UB}-\mathrm{LB}<\varepsilon$ \citep[Thm.~5]{zhang2022stochastic}. This is a lower bound for
certified termination, not for the number of bad iterations of a randomised method, and it is consistent with our
upper bound. Polynomiality in $T$ also requires $L_t$ and the state geometry to be controlled in $T$, which can fail: the
Hoffman recursion of SLDP can make $L_t$ grow geometrically, and the regularisation weight of Zhang and Sun may
have to grow with $T$ \citep[after Assumption~2]{zhang2022stochastic}. Our bounds are exponential in $d$, unlike
the dimension-free bounds of \citet{barros2026bidirectional,barros2026dimension} under strong convexity.

\paragraph{Infinite horizon.} The guarantee of \citet{julan2023infinite} is implied by $G^k\le\varepsilon$
(\cref{sec:ih-setting}). Counted in cuts, with a common $\beta$-uniform Lipschitz constant, our bound exceeds that of
their explorative Inf-EDDP by a factor of order $1/\erel$, as in finite horizon, and that of their
continual-exploration variant by a further $1/(1-\beta)$, while against their sampled variant it improves from
exponential in $1/(1-\beta)$ to polynomial, for a stronger guarantee (\cref{app:julan}).

\subsection{Comparison with reinforcement learning}\label{sec:compare-rl}

SDDP can be seen as a convex, continuous-state version of RTDP \citep{barto1995rtdp}. For finite MDPs with a
known model, \citet{efroni2019tight} prove the counterparts of our
gap identity, high-probability lemma and finite-state bound, and a uniform-PAC corollary (their Corollary~5); their sign convention, for rewards and optimistic upper values,
mirrors SDDP's lower values. In the discounted case, for the variant of RTDP that follows the frozen greedy policy
during an episode of geometric length and then backs up the visited states in reverse order (variant~(G) with
counting measure), \cref{thm:master-ih}(c) gives $\EE\sum_kG^k\le|X|\cb/(1-\beta)^2$; classical RTDP, which updates the value
before acting, needs its own timing argument. \citet{strehl2006pac} proved PAC bounds, in a different protocol,
for a modified RTDP with a known model that updates only when the decrease exceeds $\varepsilon(1-\beta)$, the
factor $1-\beta$ of \cref{thm:geom}, and asked whether the original RTDP is PAC, a question that Efroni et al.\
answered.

Compared with the covering argument of \citet[App.~F]{domingues2021kernel} (see also
\citealp{sinclair2020adaptive,sinclair2023adaptive}), which \cref{thm:master}(b) follows, our constants are
larger by a factor of about $2^d/(d+1)$; the volume version adds a convex bound valid for all thresholds at once, the
average-gap bound of \cref{rem:regret} and a potential measured by an average, and requires~\eqref{eq:setting:volume_growth}. Domingues et al.\ and Sinclair et al.\ also treat unknown models and
state no known-model planning bound. The upper bounds of Domingues et al., minima of Lipschitz ``cones'' centred
at visited states, are mirrored reverse-norm cuts, so the case of SLDP reverse-norm cuts with finitely supported
noise can essentially be read off their Appendix~F, specialised to a known model. Finally, heuristic search value
iteration \citep{smith2004hsvi,delage2022hsvi} also keeps convex piecewise-linear bounds, but explores
deterministically and terminates through coverings whose size is exponential in the depth.

\subsection{Conclusion}\label{sec:conclusion}

For TFDP-type algorithms with randomly sampled forward passes, the number of iterations whose gap exceeds an error of
$\varepsilon$ per stage is polynomial in the horizon, and linear in it under stage-uniform Lipschitz constants. The
argument needs only the forward law in the gap identity, and Lipschitz cuts that are tight at the trial point and lie
below a fixed ceiling. Three questions remain: whether the nonlinear cuts that restore tightness on mixed-integer
states have uniformly bounded Lipschitz constants (\cref{sec:SDDiP}), whether nonconvex oracles have constants
uniform in $\beta$ (\cref{sec:ih-oracles}), and whether the number of bad iterations can avoid the factor $1-\beta$
of \cref{rem:tight}.

\par\medskip\noindent\textbf{Use of large language models.}
Large language models were used during this work to suggest proof ideas, to check arguments, and to help with writing.
All statements and proofs have been verified manually by the author, who takes full responsibility for any remaining errors.

\begin{multicols}{2}
\bibliographystyle{abbrvnat}
\bibliography{references}

\begin{thebibliography}{51}
\providecommand{\natexlab}[1]{#1}
\providecommand{\url}[1]{\texttt{#1}}
\expandafter\ifx\csname urlstyle\endcsname\relax
  \providecommand{\doi}[1]{doi: #1}\else
  \providecommand{\doi}{doi: \begingroup \urlstyle{rm}\Url}\fi

\bibitem[Ahmed et~al.(2022)Ahmed, Cabral, and {Freitas Paulo da
  Costa}]{ahmed2022sldp}
S.~Ahmed, F.~G. Cabral, and B.~{Freitas Paulo da Costa}.
\newblock Stochastic {Lipschitz} dynamic programming.
\newblock \emph{Mathematical Programming}, 191:\penalty0 755--793, 2022.

\bibitem[Baldoni et~al.(2011)Baldoni, Berline, De~Loera, K{\"o}ppe, and
  Vergne]{baldoni2011integrate}
V.~Baldoni, N.~Berline, J.~A. De~Loera, M.~K{\"o}ppe, and M.~Vergne.
\newblock How to integrate a polynomial over a simplex.
\newblock \emph{Mathematics of Computation}, 80\penalty0 (273):\penalty0
  297--325, 2011.

\bibitem[Bansal and
  K{\"u}{\c{c}}{\"u}kyavuz(2026{\natexlab{a}})]{bansal2026normalization}
A.~Bansal and S.~K{\"u}{\c{c}}{\"u}kyavuz.
\newblock Normalization of {ReLU} dual for cut generation in stochastic
  mixed-integer programs.
\newblock arXiv:2602.05974, 2026{\natexlab{a}}.

\bibitem[Bansal and
  K{\"u}{\c{c}}{\"u}kyavuz(2026{\natexlab{b}})]{bansal2026sddmip}
A.~Bansal and S.~K{\"u}{\c{c}}{\"u}kyavuz.
\newblock {SDDmiP.jl}: A software package with a provably convergent {B}enders
  algorithm for multi-stage stochastic mixed-integer programming.
\newblock arXiv:2608.05567, 2026{\natexlab{b}}.

\bibitem[Barros and Guigues(2026)]{barros2026bidirectional}
P.~Barros and V.~Guigues.
\newblock Bidirectional {SDDP} with dimension-free complexity for solving
  strongly convex stochastic dynamic programming equations.
\newblock arXiv:2606.10161v2, 2026.

\bibitem[Barros et~al.(2026)Barros, Guigues, Liang, and
  Monteiro]{barros2026dimension}
P.~Barros, V.~Guigues, J.~Liang, and R.~D.~C. Monteiro.
\newblock Dimension-free complexity guarantees for dual dynamic programming.
\newblock arXiv:2606.10203v2, 2026.

\bibitem[Barto et~al.(1995)Barto, Bradtke, and Singh]{barto1995rtdp}
A.~G. Barto, S.~J. Bradtke, and S.~P. Singh.
\newblock Learning to act using real-time dynamic programming.
\newblock \emph{Artificial Intelligence}, 72\penalty0 (1--2):\penalty0 81--138,
  1995.

\bibitem[Birge and Wets(1986)]{birge1986designing}
J.~R. Birge and R.~J.-B. Wets.
\newblock Designing approximation schemes for stochastic optimization problems,
  in particular for stochastic programs with recourse.
\newblock \emph{Mathematical Programming Study}, 27:\penalty0 54--102, 1986.

\bibitem[Cruz et~al.(2026)Cruz, Diniz, and Bahiense]{cruz2026backward}
C.~Cruz, A.~L. Diniz, and L.~Bahiense.
\newblock Backward resampling in stochastic dual dynamic programming:
  Application to power generation planning.
\newblock \emph{European Journal of Operational Research}, 335\penalty0
  (1):\penalty0 192--215, 2026.
\newblock \doi{10.1016/j.ejor.2026.05.011}.

\bibitem[Dann et~al.(2017)Dann, Lattimore, and Brunskill]{dann2017unifying}
C.~Dann, T.~Lattimore, and E.~Brunskill.
\newblock Unifying {PAC} and regret: uniform {PAC} bounds for episodic
  reinforcement learning.
\newblock In \emph{Advances in Neural Information Processing Systems 30
  (NeurIPS 2017)}, 2017.
\newblock arXiv:1703.07710.

\bibitem[Delage et~al.(2024)Delage, Buffet, Dibangoye, and
  Saffidine]{delage2022hsvi}
A.~Delage, O.~Buffet, J.~S. Dibangoye, and A.~Saffidine.
\newblock {HSVI} can solve zero-sum partially observable stochastic games.
\newblock \emph{Dynamic Games and Applications}, 14\penalty0 (4):\penalty0
  751--805, 2024.

\bibitem[Deng and Xie(2024)]{deng2024relu}
H.~Deng and W.~Xie.
\newblock On the {ReLU} {L}agrangian cuts for stochastic mixed integer
  programming.
\newblock arXiv:2411.01229, 2024.

\bibitem[Domingues et~al.(2021)Domingues, M{\'e}nard, Pirotta, Kaufmann, and
  Valko]{domingues2021kernel}
O.~D. Domingues, P.~M{\'e}nard, M.~Pirotta, E.~Kaufmann, and M.~Valko.
\newblock Kernel-based reinforcement learning: a finite-time analysis.
\newblock In \emph{Proceedings of the 38th International Conference on Machine
  Learning (ICML 2021)}, volume 139 of \emph{Proceedings of Machine Learning
  Research}, 2021.

\bibitem[Dowson(2020)]{dowson2020policy}
O.~Dowson.
\newblock The policy graph decomposition of multistage stochastic programming
  problems.
\newblock \emph{Networks}, 76\penalty0 (1):\penalty0 3--23, 2020.

\bibitem[Efroni et~al.(2019)Efroni, Merlis, Ghavamzadeh, and
  Mannor]{efroni2019tight}
Y.~Efroni, N.~Merlis, M.~Ghavamzadeh, and S.~Mannor.
\newblock Tight regret bounds for model-based reinforcement learning with
  greedy policies.
\newblock In \emph{Advances in Neural Information Processing Systems 32
  (NeurIPS 2019)}, 2019.
\newblock arXiv:1905.11527.

\bibitem[Forcier and Lecl{\`e}re(2022)]{forcier2022generalized}
M.~Forcier and V.~Lecl{\`e}re.
\newblock Generalized adaptive partition-based method for two-stage stochastic
  linear programs: geometric oracle and analysis.
\newblock \emph{Operations Research Letters}, 50\penalty0 (5):\penalty0
  552--557, 2022.

\bibitem[Forcier and Lecl{\`e}re(2023)]{forcier2023trajectory}
M.~Forcier and V.~Lecl{\`e}re.
\newblock Trajectory following dynamic programming algorithms without finite
  support assumptions.
\newblock \emph{Journal of Convex Analysis}, 30\penalty0 (3):\penalty0
  951--999, 2023.

\bibitem[Fournier and Guillin(2015)]{fournier2015rate}
N.~Fournier and A.~Guillin.
\newblock On the rate of convergence in {Wasserstein} distance of the empirical
  measure.
\newblock \emph{Probability Theory and Related Fields}, 162:\penalty0 707--738,
  2015.

\bibitem[F{\"u}llner and Rebennack(2025)]{fullner2025sddp}
C.~F{\"u}llner and S.~Rebennack.
\newblock Stochastic dual dynamic programming and its variants: a review.
\newblock \emph{SIAM Review}, 67\penalty0 (3):\penalty0 415--539, 2025.
\newblock \doi{10.1137/23M1575093}.

\bibitem[F{\"u}llner et~al.(2024)F{\"u}llner, Sun, and
  Rebennack]{fullner2024lipschitz}
C.~F{\"u}llner, X.~A. Sun, and S.~Rebennack.
\newblock On {Lipschitz} regularization and {Lagrangian} cuts in multistage
  stochastic mixed-integer linear programming.
\newblock Optimization Online preprint, 2024.
\newblock Updated February 2025.

\bibitem[Girardeau et~al.(2015)Girardeau, Lecl{\`e}re, and
  Philpott]{girardeau2015convergence}
P.~Girardeau, V.~Lecl{\`e}re, and A.~B. Philpott.
\newblock On the convergence of decomposition methods for multistage stochastic
  convex programs.
\newblock \emph{Mathematics of Operations Research}, 40\penalty0 (1):\penalty0
  130--145, 2015.
\newblock \doi{10.1287/moor.2014.0664}.

\bibitem[Guigues(2020)]{guigues2020inexact}
V.~Guigues.
\newblock Inexact cuts in stochastic dual dynamic programming.
\newblock \emph{SIAM Journal on Optimization}, 30\penalty0 (1):\penalty0
  407--438, 2020.

\bibitem[Hoeffding(1963)]{hoeffding1963probability}
W.~Hoeffding.
\newblock Probability inequalities for sums of bounded random variables.
\newblock \emph{Journal of the American Statistical Association}, 58\penalty0
  (301):\penalty0 13--30, 1963.

\bibitem[Ju and Lan(2023)]{julan2023infinite}
C.~Ju and G.~Lan.
\newblock Dual dynamic programming for stochastic programs over an infinite
  horizon.
\newblock arXiv:2303.02024v4, 2023.
\newblock First posted March 2023; version 4, 10 August 2026.

\bibitem[Kallenberg(2002)]{kallenberg2002foundations}
O.~Kallenberg.
\newblock \emph{Foundations of Modern Probability}.
\newblock Springer, New York, 2nd edition, 2002.

\bibitem[Kuhn(2005)]{kuhn2005generalized}
D.~Kuhn.
\newblock \emph{Generalized Bounds for Convex Multistage Stochastic Programs},
  volume 548 of \emph{Lecture Notes in Economics and Mathematical Systems}.
\newblock Springer, Berlin, 2005.

\bibitem[Lan(2022{\natexlab{a}})]{lan2022complexity}
G.~Lan.
\newblock Complexity of stochastic dual dynamic programming.
\newblock \emph{Mathematical Programming}, 191:\penalty0 717--754,
  2022{\natexlab{a}}.

\bibitem[Lan(2022{\natexlab{b}})]{lan2022correction}
G.~Lan.
\newblock Correction to: {Complexity} of stochastic dual dynamic programming.
\newblock \emph{Mathematical Programming}, 194:\penalty0 1187--1189,
  2022{\natexlab{b}}.

\bibitem[Lan and Shapiro(2024)]{lan2023numerical}
G.~Lan and A.~Shapiro.
\newblock Numerical methods for convex multistage stochastic optimization.
\newblock \emph{Foundations and Trends in Optimization}, 6\penalty0
  (2):\penalty0 63--144, 2024.
\newblock Remark numbers refer to arXiv:2303.15672.

\bibitem[McDiarmid(1989)]{mcdiarmid1989method}
C.~McDiarmid.
\newblock On the method of bounded differences.
\newblock In \emph{Surveys in Combinatorics, 1989}, volume 141 of \emph{London
  Mathematical Society Lecture Note Series}, pages 148--188. Cambridge
  University Press, 1989.

\bibitem[Nannicini et~al.(2021)Nannicini, Traversi, and
  Wolfler~Calvo]{nannicini2021benders}
G.~Nannicini, E.~Traversi, and R.~Wolfler~Calvo.
\newblock A {Benders} squared ({B}$^2$) framework for infinite-horizon
  stochastic linear programs.
\newblock \emph{Mathematical Programming Computation}, 13\penalty0
  (4):\penalty0 645--681, 2021.
\newblock \doi{10.1007/s12532-020-00195-2}.

\bibitem[Park et~al.(2022)Park, Jia, and Hanasusanto]{park2022datadriven}
H.~Park, Z.~Jia, and G.~A. Hanasusanto.
\newblock Data-driven stochastic dual dynamic programming: performance
  guarantees and regularization schemes.
\newblock Optimization Online preprint 21376, 2022.
\newblock Last updated May 2025.

\bibitem[Pereira and Pinto(1991)]{pereira1991multi}
M.~V.~F. Pereira and L.~M. V.~G. Pinto.
\newblock Multi-stage stochastic optimization applied to energy planning.
\newblock \emph{Mathematical Programming}, 52\penalty0 (1--3):\penalty0
  359--375, 1991.
\newblock \doi{10.1007/BF01582895}.

\bibitem[Philpott and Guan(2008)]{philpott2008convergence}
A.~B. Philpott and Z.~Guan.
\newblock On the convergence of stochastic dual dynamic programming and related
  methods.
\newblock \emph{Operations Research Letters}, 36\penalty0 (4):\penalty0
  450--455, 2008.
\newblock \doi{10.1016/j.orl.2008.01.013}.

\bibitem[Philpott et~al.(2020)Philpott, Wahid, and Bonnans]{philpott2020midas}
A.~B. Philpott, F.~Wahid, and J.~F. Bonnans.
\newblock {MIDAS}: a mixed integer dynamic approximation scheme.
\newblock \emph{Mathematical Programming}, 181\penalty0 (1):\penalty0 19--50,
  2020.
\newblock Algorithm and equation numbers refer to the preprint (HAL
  hal-01401950v1, 2016).

\bibitem[Puterman(1994)]{puterman1994markov}
M.~L. Puterman.
\newblock \emph{Markov Decision Processes: Discrete Stochastic Dynamic
  Programming}.
\newblock John Wiley \& Sons, New York, 1994.

\bibitem[Rockafellar and Wets(1998)]{rockafellar1998variational}
R.~T. Rockafellar and R.~J.-B. Wets.
\newblock \emph{Variational Analysis}, volume 317 of \emph{Grundlehren der
  Mathematischen Wissenschaften}.
\newblock Springer, Berlin, 1998.

\bibitem[Shapiro(2006)]{shapiro2006complexity}
A.~Shapiro.
\newblock On complexity of multistage stochastic programs.
\newblock \emph{Operations Research Letters}, 34\penalty0 (1):\penalty0 1--8,
  2006.

\bibitem[Shapiro(2011)]{shapiro2011analysis}
A.~Shapiro.
\newblock Analysis of stochastic dual dynamic programming method.
\newblock \emph{European Journal of Operational Research}, 209\penalty0
  (1):\penalty0 63--72, 2011.
\newblock \doi{10.1016/j.ejor.2010.08.007}.

\bibitem[Shapiro and Ding(2020)]{shapiro2020periodical}
A.~Shapiro and L.~Ding.
\newblock Periodical multistage stochastic programs.
\newblock \emph{SIAM Journal on Optimization}, 30\penalty0 (3):\penalty0
  2083--2102, 2020.

\bibitem[Siddig and Song(2022)]{siddig2022adaptive}
M.~Siddig and Y.~Song.
\newblock Adaptive partition-based {SDDP} algorithms for multistage stochastic
  linear programming with fixed recourse.
\newblock \emph{Computational Optimization and Applications}, 81\penalty0
  (1):\penalty0 201--250, 2022.

\bibitem[Sinclair et~al.(2020)Sinclair, Wang, Jain, Banerjee, and
  Yu]{sinclair2020adaptive}
S.~R. Sinclair, T.~Wang, G.~Jain, S.~Banerjee, and C.~L. Yu.
\newblock Adaptive discretization for model-based reinforcement learning.
\newblock In \emph{Advances in Neural Information Processing Systems 33
  (NeurIPS 2020)}, 2020.

\bibitem[Sinclair et~al.(2023)Sinclair, Banerjee, and Yu]{sinclair2023adaptive}
S.~R. Sinclair, S.~Banerjee, and C.~L. Yu.
\newblock Adaptive discretization in online reinforcement learning.
\newblock \emph{Operations Research}, 71\penalty0 (5):\penalty0 1636--1652,
  2023.
\newblock \doi{10.1287/opre.2022.2396}.

\bibitem[Singh and Yee(1994)]{singh1994upper}
S.~P. Singh and R.~C. Yee.
\newblock An upper bound on the loss from approximate optimal-value functions.
\newblock \emph{Machine Learning}, 16\penalty0 (3):\penalty0 227--233, 1994.

\bibitem[Smith and Simmons(2004)]{smith2004hsvi}
T.~Smith and R.~Simmons.
\newblock Heuristic search value iteration for {POMDPs}.
\newblock In \emph{Proceedings of the 20th Conference on Uncertainty in
  Artificial Intelligence (UAI 2004)}, pages 520--527, 2004.
\newblock arXiv:1207.4166.

\bibitem[Song and Luedtke(2015)]{song2015adaptive}
Y.~Song and J.~Luedtke.
\newblock An adaptive partition-based approach for solving two-stage stochastic
  programs with fixed recourse.
\newblock \emph{SIAM Journal on Optimization}, 25\penalty0 (3):\penalty0
  1344--1367, 2015.

\bibitem[Strehl et~al.(2006)Strehl, Li, and Littman]{strehl2006pac}
A.~L. Strehl, L.~Li, and M.~L. Littman.
\newblock {PAC} reinforcement learning bounds for {RTDP} and {Rand-RTDP}.
\newblock In \emph{Learning for Search: Papers from the 2006 AAAI Workshop},
  Technical Report WS-06-11. AAAI Press, 2006.

\bibitem[Warrington et~al.(2019)Warrington, Beuchat, and
  Lygeros]{warrington2019generalized}
J.~Warrington, P.~N. Beuchat, and J.~Lygeros.
\newblock Generalized dual dynamic programming for infinite horizon problems in
  continuous state and action spaces.
\newblock \emph{IEEE Transactions on Automatic Control}, 64\penalty0
  (12):\penalty0 5012--5023, 2019.
\newblock \doi{10.1109/TAC.2019.2907812}.

\bibitem[Yang and Yang(2025)]{yang2025globally}
H.~Yang and H.~Yang.
\newblock Globally converging algorithm for multistage stochastic mixed-integer
  programs via enhanced {L}agrangian cuts.
\newblock Optimization Online, 2025.

\bibitem[Zhang and Sun(2022)]{zhang2022stochastic}
S.~Zhang and X.~A. Sun.
\newblock Stochastic dual dynamic programming for multistage stochastic
  mixed-integer nonlinear optimization.
\newblock \emph{Mathematical Programming}, 196:\penalty0 935--985, 2022.

\bibitem[Zou et~al.(2019)Zou, Ahmed, and Sun]{zou2019sddip}
J.~Zou, S.~Ahmed, and X.~A. Sun.
\newblock Stochastic dual dynamic integer programming.
\newblock \emph{Mathematical Programming}, 175:\penalty0 461--502, 2019.
\newblock Theorem and equation numbers refer to the preprint (Optimization
  Online 5436, 2017).

\end{thebibliography}
\end{multicols}

\appendix
\counterwithin{theorem}{subsection}
\section{Proofs for Sections~\ref{sec:setting} and~\ref{sec:main}}\label{app:main}

\subsection{Proofs of Lemmas~\ref{lem:valid}, \ref{lem:VG} and~\ref{lem:nets}}\label{app:setting}

\begin{proof}[Proof of \cref{lem:valid}]
(a) Induction on $k$ and, within an iteration, backward induction on $t$. At $t=T$, $\lv{T}{}=V_T=0$. If
$\lv{t+1}{k}\le V_{t+1}$ on $X_{t+1}$, then by forward invariance~\ref{S1} and monotonicity,
$f\le\cB_t(\lv{t+1}{k})\le\cB_t(V_{t+1})=V_t$ on $X_t$, hence $\lv t{k}\le V_t$. By~\ref{S2} and $V_T=0$, each $V_t$ is bounded on $X_t$, so $M_t\le\sup V_t-\inf\lv t0<\infty$. For the sub-Bellman inequality, a
cut created at iteration $j\le k$ satisfies $f\le\cB_t(\lv{t+1}{j})\le\cB_t(\lv{t+1}k)$, because
$\lv{t+1}{j}\le\lv{t+1}k$ and $\cB_t$ is monotone; with $\lv t0\le\cB_t(\lv{t+1}0)\le\cB_t(\lv{t+1}k)$ this gives
$\lv tk\le\cB_t(\lv{t+1}k)$. (b) The operator $\cB_t$ is monotone.
\end{proof}

\begin{proof}[Proof of \cref{lem:VG}]
(a) For $\lambda\in[0,1]$, $x+\lambda(X_t-x)\subset X_t\cap\{y:\|y-x\|\le\lambda D_t\}$ by convexity, and its
relative volume is $\lambda^{d_t}$. Take $\lambda=\min\{1,s/D_t\}$. (b) If $x\in K^j$, the $\dist_t$-ball of radius $s$
contains $K^j\cap\{\|y-x\|\le s\}$, whose $\mu_t$-measure is at least
$|I_t|^{-1}\min\{1,(s/\diam K^j)^{\dim K^j}\}\ge|I_t|^{-1}\min\{1,(s/D_t)^{d_t}\}$ by~(a). (c) Apply~(b) to singletons.
\end{proof}

\begin{lemma}[Net sizes]\label{lem:nets}
Let $\sigma>0$. (a) If $X_t$ is a compact subset of an affine space of dimension $d_t$, with $\dist_t$ given by a norm
and diameter at most $D_t$, then $X_t$ has a $\sigma$-net of size at most $(1+2D_t/\sigma)^{d_t}$. (b) If
$X_t=\bigcup_{z\in Z_t}\{z\}\times K_z$, with each $K_z$ compact in an affine space of dimension at most $d_t$ and of
diameter at most $D_t$, and $\dist_t$ is the extended metric of \cref{lem:VG}(b), then $X_t$ has a $\sigma$-net of
size at most $|Z_t|(1+2D_t/\sigma)^{d_t}$.
\end{lemma}

\begin{proof}[Proof of \cref{lem:nets}]
(a) A maximal $\sigma$-separated subset of $X_t$ is a $\sigma$-net. The balls of radius $\sigma/2$ around its
points, taken in the affine hull, are disjoint and lie in a ball of radius $D_t+\sigma/2$ centred at any point of
$X_t$; comparing volumes gives $((D_t+\sigma/2)/(\sigma/2))^{d_t}$. (b) Apply~(a) to each fibre.
\end{proof}

\subsection{Proof of Theorem~\ref{thm:master}(b) and~(c)}\label{app:cover}

With $\theta$, $\sigma_t$ and $C_t$ as in \cref{thm:master}(b), so that $2\sum_{t<T}L_t\sigma_t\le\theta/2$, let
$Y_k:=\sum_{t<T}\sum_{c\in C_t}(\Ub t-\lv tk)(c)$. By \cref{lem:basic}(i), $(Y_k)$ is adapted, nonincreasing and
nonnegative, and $Y_0$ is deterministic. Fix $k\ge1$ and $t<T$, let $x=\bx^k_t$, and choose $c\in C_t$ with
$\dist_t(x,c)\le\sigma_t$. By~\ref{O2} and~\ref{O4}, $(\lv tk-\lv t{k-1})(x)\ge(r^k_t(x)-\gamma_t)_+$. Both
approximations are $L_t$-Lipschitz, so $(\lv tk-\lv t{k-1})(c)\ge(r^k_t(x)-\gamma_t-2L_t\sigma_t)_+$, and the lifts at
the other net points are nonnegative. Hence, with $R^k$ the residual sum~\eqref{eq:setting:residuals},
\begin{equation}\label{eq:appA:covering_lift}
Y_{k-1}-Y_k\ \ge\ \sum_{t<T}\big(r^k_t(\bx^k_t)-\gamma_t-2L_t\sigma_t\big)_+\ \ge\ \Big(R^k-\sum_{t<T}\gamma_t-\theta/2\Big)_+ .
\end{equation}
By Jensen for $(\cdot)_+$ and \cref{lem:gap}, $\EE[Y_{k-1}-Y_k\mid\cA_{k-1}]\ge(G^k-\Gamma-\theta/2)_+$. Summing over
$k\in[K]$ gives the first claim, since $Y_0-\EE Y_K\le Y_0$. The count follows from
$(\theta/2)\,\#\{k\ge1:G^k>\theta+\Gamma\}\le\sum_{k\ge1}(G^k-\Gamma-\theta/2)_+$. The high-probability claim is
\cref{lem:dbp} applied to $(Y_k)$, with $C=Y_0$ (deterministic because $\lv{}0$, $\Ub t$ and the nets $C_t$ are) and
$f_k=(G^k-\Gamma-\theta/2)_+$. For~(c), take $\sigma_t=0$ and $c=x$, so that the Lipschitz property is not used, and
repeat the computation with $\theta=0$, $f_k=(G^k-\Gamma)_+$ and $C=C_{\rm fin}$. \qed

\subsection{Error floor and regret}\label{app:floor}

\begin{remark}[The error floor is attained]\label{rem:floor}
If every $X_t=\{\bar x_t\}$ is a point, $\ell_t\equiv\cb$, $\lv{}0\equiv0$, and the oracle returns the valid,
$\gamma_t$-tight constant cut $\cB_t(\lv{t+1}{k})(\bar x_t)-\gamma_t$, with $0\le\gamma_t\le\cb$, then
$G^k=\sum_{t<T}\gamma_t$ for all $k\ge2$. If $T=1$, $\cX_1(x_0,\xi)=[0,1]$, $\ell_1(x_0,y,\xi)=\gamma^Fy$,
$\lv{}0\equiv0$ and the forward oracle always returns $y=1$, then $G^k=\gamma^F$ for all $k\ge1$.
\end{remark}

\begin{remark}[Regret]\label{rem:regret}
In the setting of \cref{cor:uniform} with $\gamma=\gamma^F=0$ and $\kappa=1$, \cref{lem:valid}(a) and
\cref{lem:gap} give $0\le G^k\le M^V:=\sum_{t<T}\sup_{X_t}(V_t-\lv t0)$. The function $\psi$ of \cref{cor:uniform},
with $\rho=4LD$, satisfies $\psi(s)\ge s^{d+1}/(2(d+1)s_{\max}^d)$ on $[0,s_{\max}]$ when $s_{\max}\ge\rho$. Hence
$\Psi_T(g)\ge g^{d+1}/(2(d+1)T^d\max\{\rho,M^V/T\}^d)$ for $g\le M^V$, and Hölder's inequality with the master
bound~\eqref{eq:main:master_volume} gives
\begin{equation}\label{eq:appA:average_gap}
\EE\sum_{k\in[K]}G^k\le K^{d/(d+1)}\big(2(d+1)T^d\max\{\rho,M^V/T\}^d\,M^V\big)^{1/(d+1)} .
\end{equation}
\end{remark}

\subsection{Conditions (S2) and (S3)}\label{app:measurability}

Condition~\ref{S2} is equivalent to the boundedness of $\cB_t(0)$ when the losses are bounded below. For signed
losses and noise that is not finitely supported, the boundedness of $\cB_t(0)$ alone does not suffice: the cost of a
greedy path can then fail to be integrable.

The Lipschitz envelope $\Ub t$ of~\ref{O3} is Borel for a norm, for the extended metric of \cref{lem:VG}(b) and on
finite sets. In general, in \cref{thm:master}(a), $\int(\Ub t-\lv t0)\,d\mu_t$ denotes the upper integral
$\inf\{\int g\,d\mu_t:g\ \text{Borel},\ g\ge\Ub t-\lv t0\}$, which suffices since the approximations are Borel
by~\ref{O1} and lie below $\Ub t$.

The approximations are of the form $\lv tk=\mathsf v^k_t(\mathsf p^k_t,\cdot)$, with $\mathsf p^k_t$ a random vector of
a Euclidean space $\mathsf P_k$ whose dimension may grow with $k$ and $\mathsf v^k_t:\mathsf P_k\times X_t\to\R$ Borel,
with $\mathsf v^k_t(\mathsf p,\cdot)$ bounded for every $\mathsf p$. For these functions, the map
$(\mathsf p,x,\xi)\mapsto\Bd_t(\mathsf v^k_{t+1}(\mathsf p,\cdot);x,\xi)$ is jointly measurable and, for every $\gamma>0$, admits
a jointly measurable $\gamma$-minimising selection on the set where it is finite, and for $\gamma=0$ one when the
infimum is attained. The map $(x,\xi)\mapsto\Bd_t(W;x,\xi)$ is also jointly measurable for $W=V_{t+1}$ and for the
bounded continuous test functions $W$ of \cref{sec:noise}. Joint measurability cannot be asked for every bounded Borel
$\mathsf v^k_{t+1}$, since the projection of a Borel set need not be Borel. \emph{Sufficient condition.} It holds when
$(x,\xi;y)\mapsto\ell_{t+1}(x,y,\xi)+\delta_{\cX_{t+1}(x,\xi)}(y)$ is a normal integrand and each $\mathsf v^k_{t+1}$
is Borel in $\mathsf p$ and continuous in $y\in\R^{n_{t+1}}$, as is a maximum of finitely many cuts with Borel
coefficients that are Lipschitz for a norm. The sum is then a normal integrand with parameter $(\mathsf p,x,\xi)$, so
its optimal value is measurable and its $\gamma$-argmin mappings ($\gamma\ge0$) are measurable and closed-valued. They
admit measurable selections where they are nonempty, which for $\gamma>0$ is where the optimal value is finite
\citep[Ex.~14.29, Prop.~14.44(c), Thm.~14.37, Prop.~14.33 and Cor.~14.6]{rockafellar1998variational}. The same holds
for every bounded lower semicontinuous $W$ \citep[Ex.~14.30]{rockafellar1998variational}, in particular for these test
functions and for a Lipschitz $V_{t+1}$. For other cut families, such as the step cuts of \cref{sec:midas}, \ref{S3} is
an assumption.

\section{Cut families: complements and proofs for Section~\ref{sec:instances}}\label{app:instances}

\subsection{Benders cuts}\label{app:benders}

\begin{proof}[Proof of \cref{prop:benders}]
By the standard argument (convex combinations of $\epsilon$-optimal points stay feasible by graph convexity),
$h=\cB_t(\lv{t+1}k)$ is convex on the convex set $X^\eta_t$. By induction, every cut of index $t+1$ is valid on
$X_{t+1}$, so $\lv{t+1}0\le\lv{t+1}k\le V_{t+1}$ there (\cref{lem:valid}(a)). Since $\cX_{t+1}(z,\xi)\subset X_{t+1}$
for $z\in X^\eta_t$, monotonicity and~\ref{S2} on $X^\eta_t$ give, on $X^\eta_t$,
$-\infty<\cB_t(0)+\inf\lv{t+1}0\le h\le\cB_t(V_{t+1})=V_t\le\sup_{X^\eta_t}\cB_t(0)+\sup_{X_{t+1}}V_{t+1}<\infty$.
Moreover $h\ge\lv tk\ge\lv t0$ on $X_t$ by the sub-Bellman inequality. For $x\in X_t$, $X^\eta_t$ contains a relative
neighbourhood of $x$ in $\aff X_t$, so the relative subdifferential is nonempty. The subgradient inequality gives
validity on $X^\eta_t$, and tightness holds by construction. For a unit vector $u$ in the direction space of
$\aff X_t$, $h(x+\eta u)\ge h(x)+\eta\langle g,u\rangle$, so $\langle g,u\rangle\le(\sup_{X^\eta_t}V_t-\inf_{X_t}\lv t0)/\eta$.
Since $\pm u$ are both admissible, $f^k_t$ is $L_t$-Lipschitz on $X_t$. Under the bounded-loss condition,
the same upper bound on $V_t$ and backward induction give $\sup_{X^\eta_t}V_t\le\sum_{s>t}(a_s+\cb)$.
\end{proof}

\subsection{Details of \texorpdfstring{\cref{ex:absorb}}{the absorption example}}\label{app:absorb}

\emph{Absorption.} Comparing the infima defining $\Bd_t(W;z,\xi)$ and $\Bd_t(W;z',\xi)$, over the same set, shows
that $\cB_t(W)$ is $c$-Lipschitz on $X^\eta_t$ for every bounded $W$; so the Benders cuts are $c$-Lipschitz
($\eta\langle g,v\rangle\le h(x+\eta v)-h(x)\le c\eta$ for unit directions $v$ of $\aff X_t$).

\emph{Inventory.} With the norm $\|\cdot\|_1$, let $X_t=X:=\prod_{i\in[n]}[-b'_i,b_i]$, $t\in[T]$, be the stocks of
$n$ items (negative stocks are backlogs), $x_0\in X$, $\bxi_t\in\Xi:=\prod_i[0,\bar\xi_i]$ the demand, and
$y=x+u+e-w-\xi$, with regular orders $u\in\prod_i[0,\bar u_i]$ and \emph{unbounded} emergency orders $e\ge0$ and
disposals $w\ge0$, at unit costs $c_{o,i},c_{e,i},c_{w,i}\ge0$. Then $\cX_{t+1}(x,\xi)=X$ for every $x\in\R^n$, and
$\ell_{t+1}(x,y,\xi)=\mathrm{oc}(y-x+\xi)+\mathrm{hc}(y)$. Here $\mathrm{hc}\ge0$ is a continuous convex holding and
backlog cost, and $\mathrm{oc}(r)\ge0$, the cheapest cost of a net order $u+e-w=r$, is convex (a linear programme in
$r$) and $c$-Lipschitz, $c:=\max_i\max\{c_{e,i},c_{w,i}\}$ (absorb a change of $r_i$ with $e_i$ or $w_i$). With
$\lv{}0=0$, the losses lie in $[0,\cb]$, $\cb:=c(D+\|\bar\xi\|_1)+\max_X\mathrm{hc}$, where $D:=\sum_i(b_i+b'_i)=\diam X$.

\emph{Continuous demand} (\cref{sec:noise}). As $\Bd_t(W;x,\xi)=:\vartheta_W(x-\xi)$ depends on $x-\xi$ only, it is
$c$-Lipschitz in $\xi$ for every bounded $W$ and convex in $\xi$ for convex $W$. So \cref{cor:wasserstein} applies
with $\mathrm{rg}_t=c\|\bar\xi\|_1$, and the conditions of \cref{prop:jensen,ass:scen} hold on $\cW^{\rm c}_{t+1}$ with
$L_t=L^\xi_t=c$, or $c\sqrt n$ for the Euclidean norms of \cref{sec:sampled}, under which $D$ still bounds the
diameters; for the oracle, $\partial\vartheta_W$ has a closed graph, hence a Borel selection
\citep[Ex.~14.9 and Cor.~14.6]{rockafellar1998variational}.

\subsection{TFDP's generic cuts}\label{app:TFDP}\label{sec:TFDP}

TFDP \citep[Algorithm~1]{forcier2023trajectory} uses a $\gamma^F$-forward operator (their Def.~2.1), cuts of $\cB_t$
applied to the already updated approximation (our $\lv{t+1}{k}$) that are $L_t$-Lipschitz on the reachable set, valid
and $\gamma_t$-tight (their Assumption~4(i)--(iii)), and a monotone update (their line~11). With TFDP's indexing,
take $X_t$ equal to their reachable set $X^r_t$, which is forward invariant by definition. Their Assumptions~1
and~3(i) make $X^r_t$ bounded but not necessarily closed; we assume that it is closed, hence compact, as~\ref{S1}
requires. TFDP starts from $V^0\equiv-\infty$, so $M_t=\infty$; we start the analysis after the first iteration,
conditionally on $\cA_1$, with the first cuts $\lv t1$ as initial approximations.

\begin{proposition}[TFDP with random node selection]\label{prop:TFDP}
Consider TFDP's Algorithm~1 under their Assumptions~1--5, with closed reachable sets, with~\ref{S2} (\eg losses
bounded below and $\cB_t(0)$ bounded), and with the noise of the forward pass drawn from the law of $\bxi_t$,
independently of the past, as in \citet[Thm.~4.11]{forcier2023trajectory}. Take their constants: $n_t\le n$,
$D_t=D$, $L_t=\tilde L_t=L$ (Lipschitz constants of the cuts and of $V_t$), $\gamma_t\le\gamma$. Then
\cref{ass:S,ass:O} hold with $U=V$ from the second iteration on, with
$\sup_{X_t}(V_t-\lv t1)\le\sum_{t\le s<T}(\gamma_s+(\tilde L_s+L_s)D_s)$, and
\begin{equation}\label{eq:appB:tfdp_rate}
\EE\,\#\{k\ge2:G^k>\varepsilon+\Gamma\}\ \le\ \frac{T(T+1)(\gamma+2LD)(1+8LDT/\varepsilon)^n}{\varepsilon},
\qquad\Gamma=\sum_{t<T}\gamma_t+\sum_{t\in[T]}\gamma^F_t.
\end{equation}
\end{proposition}

\begin{proof}[Proof of \cref{prop:TFDP}]
The setting satisfies~\ref{O1},~\ref{O2} (\cref{lem:valid}(b)),~\ref{O3} with $U_t=V_t$ (\cref{lem:valid}(a))
and~\ref{O4}. Condition~\ref{S3} holds under TFDP's normal-integrand assumptions if the cut oracle is a Borel
function of its inputs: the cuts are then Lipschitz and Borel in finitely many random parameters (the trial points
and noises drawn so far), and $V_{t+1}$ is Lipschitz (\cref{app:measurability}). With random node selection from the
law of $\bxi_t$, the forward noise has the law used in the Bellman operators, even without finite support. Since
$V_t$ is $\tilde L_t$-Lipschitz \citep[Assumption~3(ii)]{forcier2023trajectory}, we have, at $x=\bx^1_t$,
\begin{equation}\label{eq:appB:tfdp_initial_gap}
\sup_{X_t}(V_t-\lv t1)\le(V_t-\lv t1)(x)+(\tilde L_t+L_t)D_t
\le\sup_{X_{t+1}}(V_{t+1}-\lv{t+1}1)+\gamma_t+(\tilde L_t+L_t)D_t ,
\end{equation}
using $\lv t1(x)\ge\cB_t(\lv{t+1}1)(x)-\gamma_t$ and $V_t(x)-\cB_t(\lv{t+1}1)(x)\le\sup(V_{t+1}-\lv{t+1}1)$.
So $M^{(1)}_t:=\sup(V_t-\lv t1)\le\sum_{t\le s<T}(\gamma_s+(\tilde L_s+L_s)D_s)$, which is deterministic. The
reachable sets need not be convex, so we apply \cref{thm:master}(b) conditionally on $\cA_1$, with nets of size
$(1+8LDT/\varepsilon)^n$ (\cref{lem:nets}(a), $\sigma_t=\varepsilon/(4LT)$) and $\sum_{t<T}M^{(1)}_t\le T(T+1)(\gamma+2LD)/2$.
\end{proof}

\subsection{Reverse-norm and augmented Lagrangian cuts (SLDP)}\label{app:sldp}\label{sec:sldp}

SLDP uses the reverse-norm cut \citep[Def.~1]{ahmed2022sldp} $f^k_t(y):=\cB_t(\lv{t+1}{k})(x)-L_t\|y-x\|$ at
$x=\bx^k_t$, averaged over the noise as on their p.~772; their $\rho_t$ is our $L_t$. The state sets $X_t$, possibly
nonconvex, are the compact projections of their Compact State Complete Continuous Recourse condition (p.~774).

\begin{proposition}[Reverse-norm cuts]\label{prop:sldp}
Let $\lv t0$ be $L_t$-Lipschitz with $\lv t0\le V_t$.
\begin{enumerate}[label=(\roman*),nosep]
\item \ref{O1} holds with constant $L_t$,~\ref{O2} holds with $\gamma=0$, and~\ref{O4} holds.
\item If $L_t\ge\Lip_{X_t}(V_t)$ for every $t$, then every cut is below $V_t$ on $X_t$, so~\ref{O3} holds with
$U_t=V_t=\Ub t$ and $M_t=\sup(V_t-\lv t0)$.
\end{enumerate}
\end{proposition}

\begin{proof}[Proof of \cref{prop:sldp}]
(i) holds by construction. (ii) Induction: if $\lv{t+1}{k}\le V_{t+1}$ on $X_{t+1}$, then
$\cB_t(\lv{t+1}{k})(x)\le V_t(x)$, so $f^k_t(y)\le V_t(x)-L_t\|y-x\|\le V_t(y)$; hence $\lv t{k}\le V_t$.
\end{proof}

So \cref{thm:master} only needs $L_t\ge\Lip V_t$. SLDP's own validity condition \citep[eq.~(20)]{ahmed2022sldp},
$f^k_t\le\cB_t(\lv{t+1}{k})$ for the \emph{current} approximation, holds when $L_t\ge\Lip\cB_t(W)$ for every
$L_{t+1}$-Lipschitz $W$, which their Hoffman recursion~(17) guarantees. It can fail when $L_t$ only dominates
$\Lip V_t$. Take identity dynamics on $X_1=X_2=[-1,1]$, $T=3$, $\cX_1(x_0)=\{0\}$, zero losses, $L_2=1$,
$\lv20(y)=\max(-1,-|y|)$ and $L_1=0$: the first cut of index~$1$, at $\bx^1_1=0$, is the constant
$0>-|y|=\cB_1(\lv21)(y)$ for $y\ne0$. Without snapping, \cref{thm:master} gives explicit rates, without the error
term of \citet[Lemma~5]{ahmed2022sldp}.

\begin{remark}[Augmented Lagrangian cuts]\label{rem:AL}
An augmented Lagrangian cut $g^{\mathrm{AL}}(b_0;\lambda,\varrho)+\lambda^\top(b_0-b)-\varrho\|b_0-b\|$ of a nodal
value function $g$ at the trial state $b_0$, with multiplier $\lambda$, penalty $\varrho$ and augmented Lagrangian
dual function $g^{\mathrm{AL}}$ \citep[Def.~3]{ahmed2022sldp}, is valid for $g$ (weak duality), is
$(\|\lambda\|_*+\varrho)$-Lipschitz, and is tight if and only if $g^{\mathrm{AL}}(b_0;\lambda,\varrho)=g(b_0)$.
Uniformly Lipschitz cuts require restricting the dual to $\|\lambda\|_*+\varrho\le L_t$. The restricted dual value
is at least $\inf_b\{g(b)+L_t\|b-b_0\|\}$, which equals $g(b_0)$ when $L_t\ge\Lip(g)$ on the whole domain of $g$ \citep[p.~768]{ahmed2022sldp};
the cut then satisfies \cref{ass:O} with $\gamma=0$ and $U=V$. Otherwise the cut is tight only for a
Lipschitz-regularised operator, the setting of \cref{sec:ZS} with a norm regularisation. With SLDP's unregularised
forward pass,~\ref{O2} is then guaranteed only with $\gamma_t$ equal to a regularisation gap, which is not
controlled a priori. \citet{fullner2024lipschitz} study Lipschitz regularisation and the tightness of Lagrangian
cuts in this setting.
\end{remark}

\subsection{Zhang and Sun}\label{app:ZS}

\paragraph{Setting and dictionary.} \citet{zhang2022stochastic} work on a recombining tree (their Assumption~3) with
stages $0\le t\le T-1$, compact, possibly nonconvex state spaces and nonnegative lower semicontinuous nodal costs
(possibly $+\infty$, which we encode in the feasible sets). Their stage $t$ is our stage $t+1$, and their root is our
stage~1. Their state space at stage $t$, written $\mathcal S_t$ here, is our $X_{t+1}$; their approximations
$\mathcal Q^i_t$ and original expected cost-to-go functions $\mathcal Q_t$ (their~(3)) are our $\lv{t+1}i$ and
$V_{t+1}$. Within a stage we use their node notation: $n$ is a node (not a dimension), $m$ ranges over its children,
$p_{nm}$ are the transition probabilities, $F_m$ the nodal feasible sets, $f_m$ the nodal costs and $\mathcal Q^i_m$
the approximations; $\mathcal R_n$, $\varpi_n$ and $\varrho$ are their $\psi_n$, $\sigma_n$ and $\rho$. Their
regularised value functions $\mathcal Q^R_n$ (their~(4)) are our $V^R$, the value functions of the regularised
operator $\BR_t$, with stage cost $\min_{z,y}\{f(z,y,x)+\varpi\mathcal R(\bar x-z)\}$ at the incoming state $\bar x$
and next-state set $\mathrm{proj}_xF_t$. The forward step of their Algorithm~3 (line~7) is greedy for $\BR_t$. The
backward step (line~14) returns the generalised conjugacy cut $C^i_m$ of each child $m$, with constant $v^i_m$ (the
value of their problem~(B) at the parent state) and parameters $(\lambda,\varrho)$ bounded by $l_{n,\lambda}$ and
$l_{n,\varrho}$ (their (15)--(16)); the update is their~(20). Exact penalisation (their Assumption~2) means
$v^{\mathrm{reg}}=v^{\mathrm{prim}}$, and a root solution is $\varepsilon$-optimal (their Def.~1) if its regularised
cost is at most $v^{\mathrm{reg}}+\varepsilon$.

\begin{proposition}[Zhang--Sun cuts satisfy the oracle]\label{prop:ZS}
Under Assumptions~1 and~3 of \citet{zhang2022stochastic}, with their parameter bounds $l_{n,\varrho}\ge\varpi_n$, or,
in the convex case, $\mathcal R_n=\|\cdot\|$, $l_{n,\lambda}\ge\varpi_n$ and $l_{n,\varrho}=0$, Algorithm~3 with
$\BR_t$ in the role of $\cB_t$ satisfies \cref{ass:S,ass:O} with $\dist_t=\|\cdot\|$; $L$ equal to their constants
$L_n=\sum_mp_{nm}(l_{m,\lambda}+l_{m,\varrho})$ \citep[Prop.~5]{zhang2022stochastic}; $\gamma=\gamma^F=0$; and
$U=V$, the value functions of the \emph{original} problem.
\end{proposition}

\begin{lemma}[Validity for the regularised value functions]\label{lem:ZSreg}
If moreover $\mathcal R_m$ is subadditive (\eg a norm) and $l_{m,\lambda}+l_{m,\varrho}\le\varpi_m$ (with the bounds
of \cref{prop:ZS}, this means $(l_{m,\lambda},l_{m,\varrho})=(0,\varpi_m)$, or $(\varpi_m,0)$ in the convex case),
then the cuts are also valid for the regularised functions, $\mathcal Q^i_n\le\mathcal Q^R_n$, and $U$ can be taken
equal to $V^R$.
\end{lemma}

\begin{proof}[Proof of \cref{prop:ZS}]
\ref{S1}: the regularised next-state set $\mathrm{proj}_xF_t\subset\mathcal S_t$ is nonempty and does not depend on
the incoming state. \ref{S2}: $f\ge0$, so~\ref{S2} reduces to the boundedness of $\BR_t(0)$, and
$\BR_t(0)(\bar x)\le\min f+\varpi\sup_{\mathcal S-\mathcal S}\mathcal R<\infty$ uniformly in $\bar x$, because
Assumption~1 makes the problem proper. \ref{S3}: the oracles return minimisers, and the tree is finite. \ref{O1}
and~\ref{O3}: by Prop.~5 of \citet{zhang2022stochastic}, $\mathcal Q^i_n$ is $L_n$-Lipschitz and lies below
$\mathcal Q_n$ on $\mathcal S_n$. Its proof shows the same for each averaged cut $\sum_mp_{nm}C^i_m$, which lies below
$\sum_mp_{nm}\mathcal Q_m=\mathcal Q_n$ by their (12), (16) and~(B). \ref{O2}: the first display chain of the proof of
their Lemma~3 (their App.~A.3.1) gives, for each child $m$, at the parent state $\bar x$,
\begin{equation}\label{eq:appB:zs_tightness}
v^i_m\ \ge\ \min_{(x,y)\in F_m,\,z\in\mathcal S_n}\ f_m(z,y,x)+\varpi_m\mathcal R_m(\bar x-z)+\mathcal Q^i_m(x).
\end{equation}
In the nonconvex case, $(\lambda,\varrho)=(0,\varpi_m)$ is feasible in their~(B); in the convex case, use their
Lemma~2 and $l_{m,\lambda}\ge\varpi_m$. Averaging over $m$, the averaged cut at $\bar x$ is at least
$\BR(\mathcal Q^i)(\bar x)\ge\BR(\mathcal Q^{i-1})(\bar x)$. The root step (lines~19--20) is exact. \ref{O4} is
their~(20).
\end{proof}

\begin{proof}[Proof of \cref{lem:ZSreg}]
Let $G_m(x):=\min_{(x',y)\in F_m}f_m(x,y,x')+\mathcal Q^i_m(x')$; by weak duality, $C^i_m\le G_m$ on $\mathcal S_n$
\citep[proof of Prop.~5]{zhang2022stochastic}. For $x,z\in\mathcal S_n$,
$C^i_m(x)-C^i_m(z)=\langle\hat\lambda,x-z\rangle+\hat\varrho[\mathcal R_m(\bar x-z)-\mathcal R_m(\bar x-x)]\le\varpi_m\mathcal R_m(x-z)$:
in the case $(0,\varpi_m)$, $\hat\lambda=0$, $\hat\varrho\le\varpi_m$ and $\mathcal R_m$ is subadditive; in the case
$(\varpi_m,0)$, $\hat\varrho=0$, $\|\hat\lambda\|_*\le\varpi_m$ and $\mathcal R_m=\|\cdot\|$. Hence
$C^i_m(x)\le\inf_z\{G_m(z)+\varpi_m\mathcal R_m(x-z)\}$, the regularised nodal value at $x$ with approximation
$\mathcal Q^i_m$. Averaging over $m$ gives $\le\BR(\mathcal Q^i)(x)\le\BR(\mathcal Q^R)(x)=\mathcal Q^R_n(x)$, by
backward induction over the tree and induction over $i$ (with $\mathcal Q^0=0\le\mathcal Q^R$).
\end{proof}

\begin{proposition}[General form of \cref{thm:ZS}]\label{prop:ZSgeneral}
In the setting of \cref{thm:ZS}, without the assumptions on Lipschitz constants, diameters and costs, let $L_s$ be
the Lipschitz constant of the approximations of index $s$ (\cref{prop:ZS}), let $X_s\subset\bigcup_{j\le J_s}B_{s,j}$
with balls of diameters $D_{s,j}$, and let $I_s(\varepsilon_s):=\sum_{j\le J_s}(1+2L_sD_{s,j}/\varepsilon_s)^{d_s}$.
Put $M_0:=v^{\mathrm{prim}}$, the optimal value, and $M_s:=\sup_{x\in X_s}\inf_{z\in X_s}\{V_s(z)+L_s\|x-z\|\}$.
Then, for all $(\varepsilon_s)_{1\le s<T}\subset(0,\infty)$ with $\sum_s\varepsilon_s\le\varepsilon/4$,
\begin{equation}\label{eq:appB:zs_bound}
\EE[\iota]\ \le\ \frac2\varepsilon\Big(M_0+\sum_{s=1}^{T-1}I_s(\varepsilon_s)\,M_s\Big),
\end{equation}
and, for $\delta\in(0,1]$, with probability at least $1-\delta$,
$\iota\le\frac{18\ln(3/\delta)}{\varepsilon}\big(M_0+\sum_sI_s(\varepsilon_s)M_s\big)$. The same bounds
hold for the number of iterations at which the regularised policy is not $\varepsilon$-optimal for the
regularised problem, or the lower bound is not $\varepsilon$-tight.
\end{proposition}

The bound has the structure of the deterministic Theorem~2 of \citet{zhang2022stochastic},
$i\le\sum_sI_s(\varepsilon_s)$ with $\sum_s\varepsilon_s=\varepsilon$ (our net count is their ball count), with an
extra factor $2M_s/\varepsilon$ and tolerances smaller by a factor~$4$; their random-sampling bound has an extra
factor of order $\prod_sN_s$.

\begin{proof}[Proof of \cref{thm:ZS,prop:ZSgeneral}]
Zhang and Sun's oracles are random functions measurable with respect to $\sigma$-fields $\Sigma^{\rm oracle}_i$. We
assume the sampled paths independent of $\Sigma^{\rm oracle}_\infty$ and include $\Sigma^{\rm oracle}_\infty$ in
$\cA_0$, so that the oracle randomness meets the requirement of \cref{sec:algo}. By \cref{prop:ZS},
\cref{thm:master}(b) applies to the regularised problem, with $G^k$ the regularised policy cost minus the lower bound,
also with several paths per iteration (\cref{sec:variants}). By their Assumption~2,
$U_0(x_0)=v^{\mathrm{prim}}=v^{\mathrm{reg}}=V^R_0(x_0)$, so $\zeta_0=0$. Hence $G^k\le\varepsilon$ implies: the
regularised policy has cost at most $v^{\mathrm{reg}}+\varepsilon$; its root decision satisfies their Def.~1; and
$v^{\mathrm{prim}}-\lv0{k-1}(x_0)\le\varepsilon$. So
$\iota\le\min\{k\ge1:G^k\le\varepsilon\}-1\le\#\{k\ge1:G^k>\varepsilon\}$ (their iteration $i+1$ is our iteration
$k=i+1$). Take the net $\{x_0\}$ of $X_0$, with $M_0=U_0(x_0)=v^{\mathrm{prim}}$ since $\lv{}0=0$, and, for
$1\le s<T$, $\sigma_s:=\varepsilon_s/L_s$ if $L_s>0$ and $\sigma_s:=\max_jD_{s,j}$ if $L_s=0$, so that
$2\sum_{s<T}L_s\sigma_s\le2\sum_s\varepsilon_s\le\varepsilon/2$. If $L_s>0$, a maximal $\sigma_s$-separated subset of
$X_s\cap B_{s,j}$ has at most $(1+2D_{s,j}/\sigma_s)^{d_s}$ points, by the packing argument of \cref{lem:nets}(a); if
$L_s=0$, one point of each nonempty $X_s\cap B_{s,j}$ suffices. In both cases $X_s$ has a $\sigma_s$-net of size at
most $I_s(\varepsilon_s)$, and $M_s$ is the potential bound of~\ref{O3} with $U=V$ (under \cref{lem:ZSreg}, it may be
replaced by $\sup_{X_s}V^R_s$). This proves \cref{prop:ZSgeneral}. For the explicit bounds, $M_0\le T\cb$ and
$\sum_{1\le s<T}M_s\le\cb\,T(T-1)/2$, so, at an accuracy $\theta$ and with $I:=\max_sI_s(\varepsilon_s)$,
$\EE[\iota]\le(2/\theta)\big(T\cb+I\cb\,T(T-1)/2\big)\le I\cb\,T(T+1)/\theta$. For $\varepsilon$-optimality, take
$\theta=\varepsilon$ and $\varepsilon_s=\varepsilon/(4(T-1))$, so $I=(1+8LD(T-1)/\varepsilon)^d$; this
is~\eqref{eq:inst:zs_bound_uniform}. For the accuracy $T\varepsilon$ of \cref{tab:summary}, take $\theta=T\varepsilon$
and $\varepsilon_s=T\varepsilon/(4(T-1))\ge\varepsilon/4$, so $I\le(1+8LD/\varepsilon)^d$ and
$\EE[\iota]\le\cb\,(T+1)(1+8LD/\varepsilon)^d/\varepsilon$, linear in $T$.
\end{proof}

\begin{remark}[Uniform sampling on non-uniform trees]\label{rem:likelihood}
If the forward noise is drawn from $q_t$ instead of the transition law $p_t$, with $p_t\le\upsilon_tq_t$, the lifts
are obtained along $q$-paths while $G^k$ is a $p$-expectation. Since $(r^k_t)_+\ge0$, the expectation under $p$ of
the positive residuals is at most $\Upsilon:=\prod_{s<T-1}\upsilon_s$ times that under $q$, and the proof of
\cref{thm:master} gives $\EE\sum_k\Psi_T(G^k/\Upsilon)\le\sum_{t<T}M_t$ (exact cuts). Under the uniform Assumption~4
of Zhang and Sun, $\upsilon_s=N_s\max p_{nm}$ over the nodes $n$ of stage $s$, each with $N_s$ children $m$, so
$\Upsilon$ is exponential in $T$ unless the transitions are nearly uniform.
\end{remark}

\subsection{SDDiP}\label{app:SDDiP}

\begin{proof}[Proof of \cref{prop:SDDiP}]
The validity condition of \citet[(3.4)]{zou2019sddip} is validity for the \emph{true} value function on
$\{0,1\}^{n_t}$, which is~\ref{O3} with $U_t=V_t$. Their tightness~(3.5) is tightness for the nodal problem with the
updated approximation, which is~\ref{O2} with $\gamma=0$ by \cref{lem:valid}(b). Their Theorem~3 proves both
properties for Lagrangian cuts; for their integer optimality cut~(4.2) both are immediate. Additional valid cuts
satisfy~\ref{O3}. On a finite set, the Lipschitz part of~\ref{O1} is not needed (\cref{thm:master}(c)), $\lv{}0=0$
is deterministic, and $M_t\le(T-t)\cb$.
\end{proof}

\subsection{MIDAS step cuts}\label{app:midas}\label{sec:midas}

Step cuts are the one family in this paper that is not Lipschitz: this subsection shows that~\ref{O1} is not
necessary, since monotone approximations can replace Lipschitz ones in the volume argument.
MIDAS \citep{philpott2020midas}, written here for minimisation, uses step cuts $f(y)=v\,\one\{y<x+\omega\one\}$
(componentwise inequality), with $v=\cB_t(\lv{t+1}{k})(x)$ at the trial point $x$ and step width $\omega>0$ (their
$\delta$). Every $\lv tk$ is then nonincreasing, but~\ref{O1} and \cref{lem:lift} fail. Take $X_t=[0,1]$,
$\lv t{k-1}=v'\one\{y<1-\omega'\}$ with $v'\ge v$ and $\omega'>0$, and the trial point $x=1$ (at distance
$\omega+\omega'$ from the earlier trial point $1-\omega'-\omega$): the new cut equals $v$ on $X_t$, but the lift is
positive only on $[1-\omega',1]$, whose measure $\omega'$ is arbitrary. Without snapping (\cref{sec:other-families}),
monotonicity of the approximations can replace their Lipschitz continuity (\cref{thm:steps}; MIDAS's own guarantee
is asymptotic, with floor $(T+1)\varepsilon$). TFDP's suggestion \citep[App.~A.3]{forcier2023trajectory} to
interpolate affinely between the steps gives cuts that are Lipschitz in $\ell_\infty$ with constant of order
$(T-t)\cb/\omega$, which \cref{thm:master}(b) covers.

\begin{theorem}[Monotone step cuts]\label{thm:steps}
Let \cref{ass:S} hold, with losses in $[0,\cb]$ and $\lv t0\equiv0$ for $t<T$. Fix $\omega>0$ and let
$X^+_t:=X_t+[0,\omega]^{n_t}$. Let every cut be a step cut $f(y)=v\,\one\{y<x+\omega\one\}$ as in \cref{sec:midas},
extended to $\R^{n_t}$ by the same formula, with $\cB_t(\lv{t+1}{k-1})(x)-\gamma_t\le v\le\cB_t(\lv{t+1}{k})(x)$ at
$x=\bx^k_t$, and let the update be monotone. Let the forward pass be as in \cref{sec:algo}, with no snapping. Then
\begin{equation}\label{eq:appB:midas_bound}
\EE\sum_{k\ge1}(G^k-\Gamma)_+\ \le\ \cb\sum_{t<T}(T-t)\frac{\vol(X^+_t)}{\omega^{n_t}}
\ \le\ \cb\,\frac{T(T+1)}2\Big(1+\frac D\omega\Big)^{n},
\end{equation}
where the last bound holds if each $X_t$ lies in a box of side $D$ and $n_t\le n$. If moreover each $X_t$ is a box,
$V_t$ is nonincreasing, and $|V_t(x)-V_t(y)|\le\varepsilon$ when $\|x-y\|_\infty<\omega$
\citep[Prop.~1]{philpott2020midas}, then $\lv tk\le V_t+(T-t)\varepsilon$ on $X_t$, so
$V^{\pi^k}_0(x_0)-V_0(x_0)\le G^k+T\varepsilon$.
\end{theorem}

\begin{proof}[Proof of \cref{thm:steps}]
Since $\lv{}0=0$, losses are nonnegative and $v\le\cB_t(\lv{t+1}{k})(x)\le\cb+\sup\lv{t+1}{k}$, backward induction
gives $0\le\lv tk\le(T-t)\cb$ on $\R^{n_t}$. \Cref{lem:gap} uses nothing beyond boundedness, so it holds. Fix $k\ge1$,
$t<T$, $x=\bx^k_t$ and $r_t=r^k_t(x)$. On the box $Q:=\{y:x\le y<x+\omega\one\}\subset X^+_t$, $f^k_t(y)=v$ and
$\lv t{k-1}(y)\le\lv t{k-1}(x)$, because $\lv t{k-1}$ is nonincreasing. Hence
$\lv tk(y)-\lv t{k-1}(y)\ge v-\lv t{k-1}(x)\ge(r_t-\gamma_t)_+$ on $Q$, using $\lv tk\ge\lv t{k-1}$. Since
$\vol(Q)=\omega^{n_t}$, the potential $\Phi^k:=\sum_{t<T}\int_{X^+_t}\lv tk$ satisfies
$\Phi^k-\Phi^{k-1}\ge\sum_{t<T}\omega^{n_t}(r_t-\gamma_t)_+$ and $\Phi^k\le\cb\sum_{t<T}(T-t)\vol(X^+_t)$. Normalise
stage $t$ by $\omega^{n_t}$ and proceed as in the proof of \cref{thm:master}(c). For the last claim, let
$\zeta_t:=\sup_k\sup_{X_t}(\lv tk-V_t)_+$. Then $v\le\cB_t(V_{t+1}+\zeta_{t+1})(x)=V_t(x)+\zeta_{t+1}$. If $y\in X_t$
and $y<x+\omega\one$, then $z:=y\vee x\in X_t$ (box), $z\ge y$ and $\|z-x\|_\infty<\omega$, so
$V_t(y)\ge V_t(z)\ge V_t(x)-\varepsilon$ and $f(y)\le V_t(y)+\varepsilon+\zeta_{t+1}$ for a step cut $f$; elsewhere
$f=0\le V_t$. Hence $\zeta_t\le\varepsilon+\zeta_{t+1}$, and $\zeta_T=0$; this is the minimisation form of
\citet[Lemma~6]{philpott2020midas}. Finally, \cref{lem:gap} gives $V^{\pi^k}_0(x_0)-V_0(x_0)\le G^k+\zeta_0$.
\end{proof}

\section{Proofs and complements for Section~\ref{sec:noise}}\label{app:noise}

\subsection{Exact cuts for the true law}\label{app:route1}

Write $E_t$ for the direction space of $\aff X_t$.

\begin{proposition}[Expected Benders cuts]\label{prop:expbenders}
Let $W$ be continuous on $X_{t+1}$ and $x\in X_t$. Suppose there are Borel functions $a:\Xi_{t+1}\to\R$ and
$g:\Xi_{t+1}\to E_t$ with $a(\xi)=\Bd_t(W;x,\xi)$, $\|g(\xi)\|\le L_t$ and
$\Bd_t(W;y,\xi)\ge a(\xi)+\langle g(\xi),y-x\rangle$ for all $y\in X_t$ and $\xi\in\Xi_{t+1}$. Then
$f(y):=\EE a(\bxi_{t+1})+\langle\EE g(\bxi_{t+1}),y-x\rangle$ is an $L_t$-Lipschitz affine minorant of $\cB_t(W)$ on
$X_t$ with $f(x)=\cB_t(W)(x)$.
\end{proposition}

\begin{proof}
Integrate the pointwise inequality against $\PP_{t+1}$; it is an equality at $y=x$, and $\|\EE g\|\le L_t$.
\end{proof}

\begin{remark}[Linear stages]\label{rem:linexact}
For linear stages with finite optimal values, data affine in $\xi$, a deterministic full-row-rank recourse matrix and
a density of $\PP_{t+1}$ that is polynomial on each cell of a polyhedral subdivision of a polytope $\Xi_{t+1}$, the
expectations of \cref{prop:expbenders} are finite sums of integrals of polynomials over polytopes, computable exactly for rational data
\citep{baldoni2011integrate}. Indeed, for fixed $x$, the noises for which a given basis is optimal form a polytope, on
which the value and the subgradient given by the dual solution are polynomials of degree at most two in $\xi$;
conclude by inclusion--exclusion over these polytopes and the density cells. TFDP
\citep[Thm.~3.7]{forcier2023trajectory} builds a stronger exact cut in a related setting.
\end{remark}

\subsection{Jensen lower discretisation}\label{app:jensen}

For $t+1<T$ and convex $X_{t+1}$, let $\cW^{\rm c}_{t+1}$ be the set of convex $L_{t+1}$-Lipschitz functions
$W:X_{t+1}\to\R$ with $\lv{t+1}0\le W\le V_{t+1}$, and let $\cW^{\rm c}_T:=\{0\}$. If $\lv{t+1}0\le V_{t+1}$ is convex
and $L_{t+1}$-Lipschitz, and every cut of index $t+1$ is affine, $L_{t+1}$-Lipschitz and valid, then every
$\lv{t+1}k$ lies in $\cW^{\rm c}_{t+1}$. The conditions of \cref{prop:jensen,ass:scen} are imposed on this class
only, since a nonconvex $W$ need not have tight affine minorants. The classes $\cW_{t+1}$
of~\eqref{eq:noise:saa_test_functions} and $\cW'_{t+1}$ of \cref{prop:uniformSAA} need no convexity: they only enter
suprema that bound biases.

\begin{proposition}[Jensen lower discretisation]\label{prop:jensen}
Fix $t<T$. Assume that $X_t$, $X_{t+1}$ and $\Xi_{t+1}\subset\R^q$ are convex, that $\PP_{t+1}$ has a finite first
moment, and that, for every $W\in\cW^{\rm c}_{t+1}$ and every $y\in X_t$, the map $\xi\mapsto \Bd_t(W;y,\xi)$ is
convex on $\Xi_{t+1}$ and $L^\xi_t$-Lipschitz. Let $\cP$ be a finite partition of $\Xi_{t+1}$ into convex Borel cells
$C$, $p_C:=\PP_{t+1}(C)$ and, if $p_C>0$, $\xi_C:=\EE[\bxi_{t+1}\mid\bxi_{t+1}\in C]$. For such a $W$ and $x\in X_t$,
let $f_C$ be an affine $L_t$-Lipschitz minorant of $\Bd_t(W;\cdot,\xi_C)$ on $X_t$ with
$f_C(x)=\Bd_t(W;x,\xi_C)$. Then $f:=\sum_{C:\,p_C>0}p_Cf_C$ is affine, $L_t$-Lipschitz, valid for $\cB_t(W)$ on $X_t$
and $\gamma$-tight at $x$ with $\gamma\le L^\xi_t\,e(\cP)$, where
$e(\cP):=\EE\|\bxi_{t+1}-\EE[\bxi_{t+1}\mid\cP]\|\le\max_{C\in\cP}\diam C$. If $\Xi_{t+1}$ lies in a cube $Q$ of side
$D_\infty>0$ and, for $h>0$ and $m:=\lceil D_\infty/h\rceil$, $\cP$ consists of the nonempty traces on $\Xi_{t+1}$ of
the $m^q$ boxes obtained by partitioning each side of $Q$ into $m$ intervals of length at most $h$, then
$|\cP|\le m^q$ and, for the Euclidean norm, $e(\cP)\le h\sqrt q/2$.
\end{proposition}

\begin{proof}
A probability law carried by a convex set $C$ has its mean in $C$, and Jensen's inequality holds for convex functions
on $C$. Cells with $p_C=0$ do not contribute to $\cB_t(W)$; below, sums run over the cells with $p_C>0$, whose
weights add up to one, so $f$ is affine and $L_t$-Lipschitz. For $y\in X_t$,
$\Bd_t(W;y,\xi_C)\le\EE[\Bd_t(W;y,\bxi_{t+1})\mid\bxi_{t+1}\in C]$, and
$f(y)\le\sum_Cp_C\Bd_t(W;y,\xi_C)\le\cB_t(W)(y)$. At $x$,
$\cB_t(W)(x)-f(x)=\sum_Cp_C\big(\EE[\Bd_t(W;x,\bxi_{t+1})\mid C]-\Bd_t(W;x,\xi_C)\big)\le L^\xi_t\sum_Cp_C\EE[\|\bxi_{t+1}-\xi_C\|\mid C]=L^\xi_te(\cP)$.
For boxes of sides at most $h$, each coordinate of $\bxi_{t+1}$ lies, given the cell, in an interval of length at
most $h$, so its conditional variance is at most $h^2/4$, and
$e(\cP)\le(\EE\|\bxi_{t+1}-\EE[\bxi_{t+1}\mid\cP]\|^2)^{1/2}\le h\sqrt q/2$.
\end{proof}

An arbitrary grid of half-open cubes of side $h$ gives $|\cP|\le(m+1)^q$. For $W\in\cW^{\rm c}_{t+1}$, convexity in
$\xi$ holds when the noise enters the right-hand side or the technology matrix of a linear stage; it can fail for a
nonconvex $W$ (if the next state equals the noise, $\Bd_t(W;y,\xi)=\ell_{t+1}(y,\xi,\xi)+W(\xi)$). When the noise
enters the cost, Edmundson--Madansky weights give the lower bound instead
\citep[][Prop.~2.6]{birge1986designing,kuhn2005generalized,forcier2023trajectory}.

\subsection{SAA backward passes: proofs}\label{app:saa}

\begin{proof}[End of the proof of \cref{thm:saa}]
By \cref{lem:basic}(iii), $V^{\pi^k}_0(x_0)\ge V_0(x_0)$, the first inequality of~\eqref{eq:noise:saa_accuracy}.
Moreover $\lv0{k-1}(x_0)\le\Vh_0(x_0)$, and by backward induction $\sup_{X_t}(\Vh_t-V_t)_+\le\sum_{t\le s<T}(\Delta^-_s)_+$.
Indeed $\Vh_t-V_t=[\cBh_t(\Vh_{t+1})-\cBh_t(V_{t+1})]+[\cBh_t(V_{t+1})-\cB_t(V_{t+1})]$, and the first bracket is at
most $\sup_{X_{t+1}}(\Vh_{t+1}-V_{t+1})_+$, because $\cBh_t$ is monotone, commutes with constants, and
$\cX_{t+1}(x,\xi)\subset X_{t+1}$. Hence
$G^k=V^{\pi^k}_0(x_0)-\lv0{k-1}(x_0)\ge V^{\pi^k}_0(x_0)-V_0(x_0)-\sum_{t<T}(\Delta^-_t)_+$.
\end{proof}

\begin{proof}[Proof of \cref{cor:wasserstein}]
(a) For a coupling $(\bxi,\hat\bxi)$ of $(\PP_{t+1},\Ph_{t+1})$ on $\Xi_{t+1}^2$ and $f=\Bd_t(W;x,\cdot)$,
$|\int f\,d\PP_{t+1}-\int f\,d\Ph_{t+1}|=|\EE[f(\bxi)-f(\hat\bxi)]|\le L^\xi_t\EE\|\bxi-\hat\bxi\|$; take the
infimum over couplings. (b) With $\underline f:=\inf f$, $f-\underline f\in[0,\mathrm{rg}_t]$, so
$\int(f-\underline f)\,d(\PP_{t+1}-\Ph_{t+1})\le \mathrm{rg}_t\sup_A(\PP_{t+1}-\Ph_{t+1})(A)$, and symmetrically.
\end{proof}

\paragraph{When the bias of \cref{cor:wasserstein}(a) is small.} Take additive noise in the dynamics,
$\cX_{t+1}(x,\xi)=\{A_{t+1}x+u+\xi:u\in\mathcal U_{t+1}\}$ with $\mathcal U_{t+1}$ compact, and a loss such that
$\xi\mapsto\ell_{t+1}(x,A_{t+1}x+u+\xi,\xi)$ is $L^\ell_{t+1}$-Lipschitz for all $x\in X_t$ and $u\in\mathcal U_{t+1}$
(\eg $L^{\ell,y}$-Lipschitz in the next state and $L^{\ell,\xi}$-Lipschitz in $\xi$, with
$L^\ell_{t+1}=L^{\ell,y}+L^{\ell,\xi}$). If $X_{t+1}$ is forward invariant for $\xi\in\Xi_{t+1}$, and $V_{t+1}$ and the
elements of $\cW_{t+1}$ are $L_{t+1}$-Lipschitz for the norm of $\Xi_{t+1}$, the Lipschitz condition holds with
$L^\xi_t\le L^\ell_{t+1}+L_{t+1}$. Let $\Ph_t$ be the empirical law of $N$ i.i.d.\ draws from $\PP_t$, with
$\diam\Xi_t\le D_\xi$. Replacing one sample point moves a mass $1/N$ by at most $D_\xi$, so it changes
$\Wa(\PP_t,\Ph_t)$ by at most $D_\xi/N$. McDiarmid's inequality \citep{mcdiarmid1989method} with a union bound over
$t$ then gives, with probability at least $1-\delta$, $\Wa(\PP_t,\Ph_t)\le\EE\Wa(\PP_t,\Ph_t)+D_\xi\sqrt{\ln(T/\delta)/(2N)}$
for all $t\in[T]$; the mean is bounded in \cref{sec:saa}.

\begin{proposition}[i.i.d.\ SAA: a bias governed by the state dimension]\label{prop:uniformSAA}
Let $\|\cdot\|$ be a norm, for which $L_{t+1}$, $L^x_t$ and the diameters are taken. For each $t<T$, let
$\cW'_{t+1}$ be the set of $L_{t+1}$-Lipschitz $W$ on $X_{t+1}$ (set $L_T=D_T:=0$, so that the terms of index $T$ vanish) with
$\underline w_{t+1}\le W\le\underline w_{t+1}+B_{t+1}$, and assume that it contains $\cW_{t+1}\cup\{V_{t+1}\}$.
Assume that, for every $\sigma>0$, $X_s$ has a $\sigma$-net for $\|\cdot\|$ of size at most $(1+2D_s/\sigma)^d$ for
$s=t$ and, if $t+1<T$, for $s=t+1$ (see \cref{lem:nets}(a)). Assume that, for all $x,x'\in X_t$,
$\xi\in\Xi_{t+1}$ and $W\in\cW'_{t+1}$, $\Bd_t(W;x,\cdot)$ takes values in an interval of length $\mathrm{rg}_t$ and
$|\Bd_t(W;x,\xi)-\Bd_t(W;x',\xi)|\le L^x_t\|x-x'\|$. Let $\Ph_{t+1}$ be the empirical law of $N$ i.i.d.\ draws from
$\PP_{t+1}$. Then for every $u>0$ and $\delta\in(0,1)$, with probability at least $1-\delta$, for all $t<T$,
\begin{gather}
\sup_{x\in X_t,\,W\in\cW'_{t+1}}\big|\cB_t(W)(x)-\cBh_t(W)(x)\big|
\le4u+\mathrm{rg}_t\sqrt{\frac{\ln\nu_t(u)+\ln(2T/\delta)}{2N}}, \label{eq:appC:uniform_saa_deviation} \\
\ln\nu_t(u):=d\ln\Big(1+\frac{2L^x_tD_t}{u}\Big)+\Big(1+\frac{8L_{t+1}D_{t+1}}{u}\Big)^{d}\ln\Big(1+\frac{2B_{t+1}}{u}\Big). \label{eq:appC:covering_number}
\end{gather}
With $u=N^{-1/(d+2)}$, $\Delta^\pm_t=O\big(N^{-1/(d+2)}\sqrt{\ln N}\big)$, with constants depending on $d$,
$L_{t+1}D_{t+1}$, $L^x_tD_t$, $B_{t+1}$, $\mathrm{rg}_t$ and $\ln(T/\delta)$, but not on $q$.
\end{proposition}

With losses in $[0,\cb]$ and $\lv{}0=0$, one can take $\underline w=0$ and $B_{t+1}=(T-t-1)\cb$ if $V_{t+1}$ is
$L_{t+1}$-Lipschitz.
\begin{proof}
\emph{Nets.} By assumption, take a $(u/L^x_t)$-net $\mathcal N$ of $X_t$ and a $(u/(4L_{t+1}))$-net
$\{z_j\}_{j\in[n_Z]}$ of $X_{t+1}$, with $n_Z\le(1+8L_{t+1}D_{t+1}/u)^d$ (a radius $u/0$ is infinite, and one point
is then a net; this covers $t=T-1$, where $L_T=0$). \emph{Cover of $\cW'_{t+1}$.} Code $W$ by the integers
$\lfloor(W(z_j)-\underline w)/(u/2)\rfloor\in\{m\in\mathbb N:m\le\lfloor2B/u\rfloor\}$. If $W,W'$ have the same code,
then for $y\in X_{t+1}$ and $z_j$ with $\|y-z_j\|\le u/(4L_{t+1})$,
$|W(y)-W'(y)|\le|W(y)-W(z_j)|+|W(z_j)-W'(z_j)|+|W'(z_j)-W'(y)|<u$. One member of $\cW'_{t+1}$ in each nonempty code
class gives a $u$-cover $\cW_u$ in sup norm with $|\cW_u|\le(1+2B/u)^{n_Z}$. \emph{Concentration.} For
$(x',W')\in\mathcal N\times\cW_u$, $f=\Bd_t(W';x',\cdot)$ takes values in an interval of length $\mathrm{rg}_t$;
Hoeffding's inequality \citep{hoeffding1963probability} and a union bound over these $\nu_t(u)\ge|\mathcal N||\cW_u|$
functions and over $t$ give, with probability at least $1-\delta$,
$|\int f\,d(\PP_{t+1}-\Ph_{t+1})|\le \mathrm{rg}_t\sqrt{(\ln\nu_t(u)+\ln(2T/\delta))/(2N)}$ for all of them.
\emph{Approximation.} For $(x,W)$, pick $x'\in\mathcal N$ with $\|x-x'\|\le u/L^x_t$ and $W'\in\cW_u$ with
$\|W-W'\|_\infty\le u$. Since $\cX_{t+1}(x',\xi)\subset X_{t+1}$, $|\Bd_t(W;x',\xi)-\Bd_t(W';x',\xi)|\le u$, so
$|\Bd_t(W;x,\xi)-\Bd_t(W';x',\xi)|\le2u$ for every $\xi\in\Xi_{t+1}$, and both integrals move by at most $2u$. For the
rate, $\ln\nu_t(u)=O(u^{-d}\ln(1/u))$, and $u=N^{-1/(d+2)}$ balances the two terms.
\end{proof}

\begin{remark}[Bias, and SAA forward passes]\label{rem:bias}
In the instance of \cref{prop:counter} (\cref{app:sampled}), with a fixed SAA whose mean exceeds $0.6$, SDDP ends
with a policy that is $0.1$-suboptimal for the true problem, with forward passes from either law; this is consistent
with \cref{thm:saa}, since $\Delta^-_1>0.1$. For the Jensen law $\sum_Cp_C\delta_{\xi_C}$ of \cref{prop:jensen},
$\Delta^-_t\le0$ if $\Bd_t(V_{t+1};x,\cdot)$ is convex for every $x\in X_t$ (the lower bound is certified), and
$\Delta^+_t\le L^\xi_te(\cP)$ if $\Bd_t(W;x,\cdot)$ is $L^\xi_t$-Lipschitz for $W\in\cW_{t+1}$, as in
\cref{cor:wasserstein}(a). If the forward passes are drawn from $\Ph$ (classical SDDP on the SAA), \cref{thm:master}
controls the SAA gap; transferring it to the true cost needs regularity of the policy in $\xi$, which greedy argmin
selections do not have in general.
\end{remark}

\subsection{Sampled cuts}\label{app:sampled}

At iteration $k$ and for the cut of index $t<T$ (backward), draw $(\bxi^k_{t,j})_{j\in[N]}$ i.i.d.\ from
$\PP_{t+1}$ ($N=N_{k,t}$ deterministic), independently of everything drawn before. Here the state sets are convex
with normalised Lebesgue measure (\cref{lem:VG}(a)), the cuts are affine, and $\lv t0\le\cB_t(\lv{t+1}0)$.

\begin{assumption}[Scenario oracle]\label{ass:scen}
For every $W\in\cW^{\rm c}_{t+1}$ (\cref{app:jensen}), every $x\in X_t$ and every $\xi\in\Xi_{t+1}$, the oracle
returns $a(\xi)=\Bd_t(W;x,\xi)$ and $g(\xi)\in E_t$ with $\|g(\xi)\|\le L_t$ and
$\Bd_t(W;y,\xi)\ge a(\xi)+\langle g(\xi),y-x\rangle$ for all $y\in X_t$, Borel in its inputs, and
$\sup_\xi a(\xi)-\inf_\xi a(\xi)\le \mathrm{rg}_t$.
\end{assumption}

For instance, $\mathrm{rg}_t\le L^\xi_t\diam\Xi_{t+1}$, or $\mathrm{rg}_t\le(T-t)\cb$ if losses lie in $[0,\cb]$ and
$\lv{}0=0$. The raw sampled cut is $\hat f(y):=\bar a+\langle\bar g,y-x\rangle$, with $\bar a,\bar g$ the sample
means of $a(\bxi^k_{t,i})$ and $g(\bxi^k_{t,i})$; it is valid in expectation only. Let
$b(N,\delta):=(\mathrm{rg}_t+2L_tD_t)\sqrt{\ln(4/\delta)/N}$ (see \cref{lem:conc}).

\begin{theorem}[Shifted sampled cuts]\label{thm:sampled}
Let \cref{ass:S,ass:scen} hold, with a Euclidean norm on $E_t$ in which $D_t\ge\diam X_t$, and assume convex state
sets, convex $L_t$-Lipschitz $\lv t0$ with $\lv t0\le\cB_t(\lv{t+1}0)$, and exact or $\gamma^F$-optimal forward passes from the true laws, with deterministic
$\gamma^{F,k}_t$. Choose $\delta_{k,t}>0$ with $\delta':=\sum_{k\ge1,t<T}\delta_{k,t}\le1$. Let
$b_{k,t}:=b(N_{k,t},\delta_{k,t})$, add the shifted cut $\hat f-b_{k,t}$ at iteration $k$ and index $t$, and set
$\Gamma_k:=2\sum_{t<T}b_{k,t}$; let $\Gamma^F_k$ be the forward-error sum~\eqref{eq:main:error_sums}. There is an
event $A$ with $\PP(A)\ge1-\delta'$ on which every cut is valid ($\lv tk\le V_t$ for all $k,t$), and
\begin{enumerate}[label=(\roman*),nosep]
\item $\EE\big[\one_A\sum_{k\ge1}\Psi_T((G^k-\Gamma^F_k-\Gamma_k)_+)\big]\le(1+\delta'/2)M^V$, where
$M^V:=\sum_{t<T}\sup_{X_t}(V_t-\lv t0)$;
\item for $\delta\in(0,1]$, with probability at least $1-\delta-\delta'$, simultaneously for all $\theta>0$,
$\#\{k\ge1:G^k>\theta+\Gamma^F_k+\Gamma_k\}\le(9\ln(3/\delta)+\delta'/2)M^V/\Psi_T(\theta)$.
\end{enumerate}
With $\delta_{k,t}=\delta'/(Tk(k+1))$, the condition $\Gamma_k\le g_k$ for a target $g_k>0$ holds as soon as
$N_{k,t}\ge\big(2T(\mathrm{rg}_t+2L_tD_t)/g_k\big)^2\ln\big(4Tk(k+1)/\delta'\big)$.
\end{theorem}

Let $\mathcal G$ be the $\sigma$-field generated by everything drawn before the sample of iteration $k$ and index $t$;
$x:=\bx^k_t$ and $W:=\lv{t+1}{k}$ are $\mathcal G$-measurable. The conditional mean of the raw cut,
$m(y):=\EE[\hat f(y)\mid\mathcal G]=\cB_t(W)(x)+\langle\EE[g\mid\mathcal G],y-x\rangle$, satisfies $m\le\cB_t(W)$ on
$X_t$.

\begin{lemma}[Uniform deviation of a sampled affine cut]\label{lem:conc}
Let \cref{ass:scen} hold with a Euclidean (Hilbert) norm on $E_t$, and let $\delta\in(0,1]$. For every event
$E\in\mathcal G$ on which $W\in\cW^{\rm c}_{t+1}$, almost surely on $E$,
$\PP\big(\sup_{y\in X_t}|\hat f(y)-m(y)|>b(N,\delta)\mid\mathcal G\big)\le\delta$.
\end{lemma}

\begin{proof}
Conditionally on $\mathcal G$ the pairs $(a,g)(\bxi^k_{t,i})$ are i.i.d.\ with the law obtained by freezing $(x,W)$.
For $y\in X_t$, $\|y-x\|\le D_t$, so $\sup_y|\hat f(y)-m(y)|\le|\bar a-\EE a|+D_t\|\bar g-\EE g\|$. Hoeffding's
inequality gives $\PP(|\bar a-\EE a|>\mathrm{rg}_t\sqrt{\ln(4/\delta)/(2N)})\le\delta/2$. The map
$(\xi_i)_i\mapsto\|\bar g-\EE g\|$ has bounded differences $2L_t/N$, and $\EE\|\bar g-\EE g\|\le L_t/\sqrt N$ by the
Hilbert structure, so McDiarmid's inequality gives $\PP(\|\bar g-\EE g\|>L_t(1+\sqrt{2\ln(2/\delta)})/\sqrt N)\le\delta/2$.
Finally $1+\sqrt{2v}\le2\sqrt{v+\ln2}$ for $v\ge0$, since $4v+4\ln2-(1+\sqrt{2v})^2=(\sqrt{2v}-1)^2+4\ln2-2>0$; with
$v=\ln(2/\delta)$ the sum of the two deviations is at most $b(N,\delta)$.
\end{proof}

\begin{proof}[Proof of \cref{thm:sampled}]
\emph{Good events.} Let $A_{k,t}:=\{\sup_{X_t}|\hat f-m|\le b_{k,t}\}$, $\Omega_0$ the whole space and
$\Omega_k:=\Omega_{k-1}\cap\bigcap_{t<T}A_{k,t}$, which is $\cA_k$-measurable. On $\Omega_{k-1}$, all cuts of
iterations $\le k-1$ are valid, so \cref{lem:valid}(a) gives $\lv{}j\le V$ for $j\le k-1$. On
$E:=\Omega_{k-1}\cap\bigcap_{s>t}A_{k,s}\in\mathcal G$, the cuts of indices $s>t$ of iteration $k$ are valid too, so
$W=\lv{t+1}{k}$, the maximum of the convex $\lv{t+1}0$ and of valid affine $L_{t+1}$-Lipschitz cuts, lies in
$\cW^{\rm c}_{t+1}$. On $E\cap A_{k,t}$, with $b=b_{k,t}$, we have $\hat f-b\le m\le\cB_t(W)$ on $X_t$ (validity) and
$(\hat f-b)(x)=\bar a-b\ge\cB_t(W)(x)-2b$. By \cref{lem:conc}, the probability that $E$ holds while $A_{k,t}$ fails
is at most $\delta_{k,t}$, also conditionally on $\cA_{k-1}$. Summing over the first failing index,
$\PP(\Omega_{k-1}\setminus \Omega_k\mid\cA_{k-1})\le\sum_{t<T}\delta_{k,t}$. Let $A:=\bigcap_k\Omega_k$; then
$\PP(A^c)\le\delta'$.

\emph{Lift on $\Omega_k$.} There, every cut so far is valid and $L_t$-Lipschitz, and the shifted cut at $x$ is at
least $\cB_t(\lv{t+1}{k-1})(x)-2b_{k,t}$. As in the proof of \cref{thm:master},
$\Phi^k-\Phi^{k-1}\ge\Psi_T((R^k-\Gamma_k)_+)$. On $\Omega_{k-1}$, $\lv{}{k-1}\le V$, so
$r^k_t\in[0,\sup(V_t-\lv t0)]$ and $R^k\in[0,M^V]$. Also $\Psi_T(u)\le\psi_0(u)\le u/2$
by~\eqref{eq:main:huber_potential}.

\emph{Stopped potential.} Let $\tau_{\rm f}:=\inf\{k:\Omega_k\text{ fails}\}$, a stopping time, and
$Y_k:=\Phi^{k\wedge(\tau_{\rm f}-1)}$. The process $(Y_k)$ is adapted, nondecreasing and bounded by $M^V$, and
$Y_k-Y_{k-1}=\one_{\Omega_k}(\Phi^k-\Phi^{k-1})$. Using $\Omega_{k-1}\in\cA_{k-1}$, conditional Jensen and \cref{lem:gap},
\begin{equation}\label{eq:appC:stopped_potential}
\EE[Y_k-Y_{k-1}\mid\cA_{k-1}]\ge\one_{\Omega_{k-1}}\Psi_T\big((G^k-\Gamma^F_k-\Gamma_k)_+\big)-\tfrac{M^V}2\,\PP(\Omega_{k-1}\setminus \Omega_k\mid\cA_{k-1}).
\end{equation}
Summing and taking expectations gives~(i), since $\EE\lim_kY_k\le M^V$,
$\sum_k\PP(\Omega_{k-1}\setminus \Omega_k)\le\delta'$ and $\one_A\le\one_{\Omega_{k-1}}$. For~(ii), apply \cref{lem:dbp}
to $M^V-Y_k$: with probability at least $1-\delta$ its regret is at most $9M^V\ln(3/\delta)$, and
by~\eqref{eq:appC:stopped_potential} it is at least
$\sum_{k\ge1}\one_{\Omega_{k-1}}\Psi_T((G^k-\Gamma^F_k-\Gamma_k)_+)-M^V\delta'/2$; intersect with $A$. Finally,
$2\sum_{t<T}b_{k,t}\le g_k$ if each $b_{k,t}\le g_k/(2T)$; solving for $N$ gives the sample-size condition.
\end{proof}

\begin{proposition}[Raw sampled cuts can stall SDDP]\label{prop:counter}
Let $T=2$, $X_0=\{x_0\}$, $\cX_1(x_0,\xi)=[0,1]$, $\ell_1\equiv0$, $X_1=[0,1]$, $\cX_2(x,\xi)=\{0\}$,
$\ell_2(x,0,\xi)=\xi x+0.6(1-x)$ with $\bxi_2\sim U[0,1]$, and $\lv{}0\equiv0$. Then $V_1(x)=0.6-0.1x$ and
$V_0(x_0)=0.5$, attained only at $x_1=1$. Run SDDP with exact forward passes and raw sampled cuts at index~1, built
from fresh samples of sizes $N_k\ge1$. If the cut of iteration $j$ is built from a sample with mean $\bar\xi_j>0.6$, then at every
iteration $k>j$ the policy plays $x_1=0$, $V^{\pi^k}_0(x_0)=0.6=V_0(x_0)+0.1$ and $\lv0{k-1}(x_0)=0.6>V_0(x_0)$; in
particular $G^k=0$ although $\pi^k$ is $0.1$-suboptimal. This happens with probability at least
$\PP(\bar\xi_1>0.6)>0$ for every choice of sample sizes, and almost surely if $\sup_kN_k<\infty$.
\end{proposition}

\begin{proof}[Proof of \cref{prop:counter}]
For every $\xi$, $y\mapsto \Bd_1(0;y,\xi)=\xi y+0.6(1-y)$ is affine, so its exact scenario cut is itself, and a raw
sampled cut is the line $f_j(y)=\bar\xi_jy+0.6(1-y)$, through $(0,0.6)$ and $(1,\bar\xi_j)$. Hence $\lv1k(0)=0.6$ for
$k\ge1$. If $\bar\xi_j>0.6$ for some $j\le k-1$, then $\lv1{k-1}(y)\ge0.6+(\bar\xi_j-0.6)y>0.6$ for $y\in(0,1]$, so the
forward pass of iteration $k$ plays $x_1=0$. Index~0 is deterministic with $d_0=0$; its exact cut gives
$\lv0{k-1}(x_0)=\min_{[0,1]}\lv1{k-1}=0.6$, while $V^{\pi^k}_0(x_0)=V_1(0)=0.6$. Cuts are never removed, so this
persists. The mean of $N$ i.i.d.\ $U[0,1]$ variables has a positive density on $(0,1)$, so $\PP(\bar\xi>0.6)>0$. The
cuts use independent samples, so if $N_k\le\bar N$ for all $k$, an exceedance occurs almost surely.
\end{proof}

\subsection{Noise with unbounded support}\label{app:unbounded}

With convex $V_t$ and an unbounded reachable set, costs and residuals are typically unbounded (a convex function
bounded on $\R^n$ is constant); a tail term replaces the uniform bound.

\begin{proposition}[Leaving the compact set]\label{prop:unbounded}
Let $X'_t\supset X_t$ be Borel sets with $\emptyset\ne \cX_t(x,\xi)\subset X'_t$ for $x\in X'_{t-1}$, $\xi\in\Xi_t$,
where the $X_t$ are convex and compact with normalised Lebesgue measure. Let \cref{ass:S} hold with $X'_t$ in place
of $X_t$, except that compactness is dropped and the uniform bound in~\ref{S2} is replaced by integrability of
$\ell_t$, $V_t$ and $\lv t0$ along the forward paths. Assume that $\lv t0\le V_t$ on $X'_t$ and that $\lv t0$ is
$L_t$-Lipschitz on $X_t$, that the cuts are valid on $X'_t$ (below $V_t$), $L_t$-Lipschitz on $X_t$ and $\gamma$-tight
as in \cref{sec:route2}, and that for all $k$
\begin{equation}\label{eq:appC:outside_residual}
\Delta^{\rm out}_k:=\EE\Big[\sum_{t<T}r^k_t(\bx^k_t)\,\one\{\bx^k_t\notin X_t\}\ \Big|\ \cA_{k-1}\Big]\le\Delta^{\rm out} .
\end{equation}
Then $\EE\big[\sum_{k\ge1}\Psi_T((G^k-\bar{\bGamma}_k-\Delta^{\rm out})_+)\big]\le M^V:=\sum_{t<T}\sup_{X_t}(V_t-\lv t0)$.
The bound on $\Delta^{\rm out}_k$ holds in either of two cases: $r^k_t\le M'_t$ on $X'_t$ for some $M'_t\ge0$ and
$\PP(\bx^k_t\notin X_t\mid\cA_{k-1})\le o_t$, with $\Delta^{\rm out}=\sum_tM'_to_t$; or
$\EE[(V_t-\lv t0)(\bx^k_t)^2\mid\cA_{k-1}]\le\bar m_t^2$ and the same $o_t$, with $\Delta^{\rm out}=\sum_t\bar m_t\sqrt{o_t}$.
\end{proposition}

\begin{proof}
Validity on $X'_t$ and $\lv t0\le V_t$ give $\lv t0\le\lv tk\le V_t$ on $X'_t$. Since $\cX_{t+1}(x,\xi)\subset X'_{t+1}$
for $x\in X'_t$ and $\cB_t$ is monotone, $r^k_t\le\cB_t(V_{t+1})-\lv t0=V_t-\lv t0$ there (residuals may be negative).
\Cref{lem:gap} holds with $X'_t$ in place of $X_t$, by the same telescoping: $\lv t{k-1}(\bx^k_t)$ is integrable because
$\lv t0\le\lv t{k-1}\le V_t$, and so is every $r^k_t(\bx^k_t)$, by the one-step
identity~\eqref{eq:main:one_step_expectation}. The potential is computed on the compact sets $X_t$, where it is at
most $M^V$, and \cref{lem:lift} is applied only at trial points in $X_t$. Let
$\tilde R^k:=\sum_{t<T}r^k_t(\bx^k_t)\one\{\bx^k_t\in X_t\}$, which is integrable and at most $M^V$. Since
$\sum_{t<T}(r^k_t(\bx^k_t)-\bgamma^k_t)_+\one\{\bx^k_t\in X_t\}\ge(\tilde R^k-\bGamma_k)_+$, we get
$\Phi^k-\Phi^{k-1}\ge\Psi_T((\tilde R^k-\bGamma_k)_+)$. Then
$\EE[\tilde R^k\mid\cA_{k-1}]=\EE[R^k\mid\cA_{k-1}]-\Delta^{\rm out}_k\ge G^k-\Gamma^F_k-\Delta^{\rm out}$; conclude as
in \cref{thm:master}. For the second sufficient condition, use Cauchy--Schwarz and $r^k_t\le V_t-\lv t0$.
\end{proof}

\section{Proofs and complements for Section~\ref{sec:infinite}}\label{app:infinite}

\subsection{Standard facts, measurability and the four lemmas}\label{app:ih-std}

Let $\mathfrak M$ be the set of bounded Borel $W:X\to\R$ such that $(x,\xi)\mapsto\Bd(W;x,\xi)$ is jointly measurable; by~\ref{SI3},
it contains the models of the algorithm and $V$, and $\cB$ is defined on it.

\begin{lemma}[Standard facts]\label{lem:std}
Under \cref{ass:Sinf}:
(i) for $W,W'\in\mathfrak M$, $\cB W$ is Borel and bounded, $W\le W'$ implies $\cB W\le\cB W'$, $\cB(W+c)=\cB W+\beta c$ for
constants $c$, so $\sup|\cB W-\cB W'|\le\beta\sup|W-W'|$; and $\mathfrak M$ is closed under uniform convergence;
(ii) $0\le V\le\cb/(1-\beta)$, $V$ is the only bounded fixed point of $\cB$ in $\mathfrak M$, every $W\in\mathfrak M$ with
$W\le\cB W$ satisfies $W\le V$, and $\cB^jW\to V$ uniformly whenever $\cB^jW\in\mathfrak M$ for all $j\ge0$;
(iii) $V^{\pi}(x_0)\ge\EE[\ell(x_0,\bx_1,\bxi_1)+\beta V(\bx_1)]\ge V(x_0)$ for every policy $\pi$, also conditionally on
a $\sigma$-field independent of the noise on which $\pi$ may depend (as $\pi^k$ depends on $\cA_{k-1}$).
\end{lemma}

\begin{proof}
(i) Since $\emptyset\ne\cX(x,\xi)\subset X$ and $\ell\in[0,\cb]$, $\beta\inf W\le\Bd(W;x,\xi)\le\cb+\beta\sup W$, and
Fubini's theorem makes $\cB W$ Borel; monotonicity and the shift are immediate, and give the contraction. If
$W_j\to W$ uniformly, $|\Bd(W_j;x,\xi)-\Bd(W;x,\xi)|\le\beta\sup|W_j-W|$, so $\Bd(W;\cdot,\cdot)$ is a pointwise limit of
jointly measurable functions. (ii) $V=\cB V\ge\beta\inf V$ gives $\inf V\ge0$, and $V\le\cb+\beta\sup V$ gives $\sup V\le\cb/(1-\beta)$.
If $W\le\cB W$ and $c:=\sup(W-V)_+$, then $W\le V+c$ and $W\le\cB W\le\cB(V+c)=V+\beta c$, so $c\le\beta c$ and
$c=0$. Uniqueness and convergence follow from the contraction. (iii) Let $\lambda_t:=\ell(\bx_t,\bx_{t+1},\bxi_{t+1})$ and
$\mathcal H_t$ the $\sigma$-field of the conditioning and of $\bxi_{[t]}$. As $\lambda_t+\beta V(\bx_{t+1})\ge\Bd(V;\bx_t,\bxi_{t+1})$,
the freezing property gives $\EE[\lambda_t+\beta V(\bx_{t+1})\mid\mathcal H_t]\ge V(\bx_t)$. Iterate from $t=1$ and let the
horizon tend to infinity ($V$ is bounded): $V^\pi(x_0)\ge\EE[\lambda_0+\beta V(\bx_1)]$, and $t=0$ gives $\ge V(x_0)$.
\end{proof}

\paragraph{A sufficient condition for~\ref{SI3}.} Assume the sufficient condition of \cref{app:measurability} and,
moreover, that $(x,y)\mapsto\ell(x,y,\xi)+\delta_{\cX(x,\xi)}(y)$ is lower semicontinuous on $X\times X$ for every
$\xi$. For a bounded lower semicontinuous $W$, $\Bd(W;\cdot,\cdot)$ is then jointly measurable (\cref{app:measurability}),
and $\Bd(W;\cdot,\xi)$, an infimum over the compact $X$, is lower semicontinuous and at least $\beta\inf W$, so $\cB W$
is lower semicontinuous (Fatou). Hence the $\cB^j0$ lie in $\mathfrak M$, and their uniform limit $V$ is in $\mathfrak M$
and satisfies $\cB V=V$ (\cref{lem:std}(i)--(ii)). For finitely supported noise and deterministic oracles, \ref{SI3} is not needed:
$\cB$ is defined on all bounded functions, and all the random variables take countably many values.

\begin{lemma}[Sufficient conditions]\label{lem:validIH}
Assume \cref{ass:Sinf}, and let every cut be computed for the current model $W$. (a) If $\uV0\le V$ and every cut
satisfies $f\le\cB W$ on $X$, then every model is $\le V$, so~\ref{OI3} holds with $U=V$ and $M_U\le\cb/(1-\beta)-\inf_X\uV0$. If moreover
$\uV0\le\cB\uV0$ (e.g.\ $\uV0=0$), the assumption $\uV0\le V$ holds by \cref{lem:std}(ii), and every model satisfies
$W\le\cB W$, so $r^k\ge0$. (b) Under~\ref{OI4}, if
$f(x)\ge\cB W(x)-\gamma_{k,t}$ at $x=\bx^k_t$, then~\ref{OI2} holds.
\end{lemma}

\begin{proof}
(a) If $\uV0\le\cB\uV0$, then $\uV0\le V$ by \cref{lem:std}(ii). By induction on the cuts, let $W\le V$ be the current
model and $f$ a cut; then $f\le\cB W\le\cB V=V$, so $\max(W,f)\le V$. If moreover $W\le\cB W$, then, with
$W':=\max(W,f)\ge W$, $W\le\cB W\le\cB W'$ and $f\le\cB W\le\cB W'$, so $W'\le\cB W'$. (b) By~\ref{OI4},
$W\ge\uV{k-1}$, and $\cB$ is monotone.
\end{proof}

\begin{proof}[Proof of \cref{lem:gapIH}]
Write $\uV{}:=\uV{k-1}$, $r_t:=r^k(\bx_t)$ and $\mathcal H_t:=\cA_{k-1}\vee\sigma(\bxi_{[t]})$. The bounds on $e_t$ are the
definitions of $\pi^k$ and $\Bd$. By the freezing property and~\ref{SI3},
$\EE[\Bd(\uV{};\bx_t,\bxi_{t+1})\mid\mathcal H_t]=\cB\uV{}(\bx_t)$, so
$\EE[\lambda_t+\beta\uV{}(\bx_{t+1})\mid\mathcal H_t]=\uV{}(\bx_t)+r_t+\EE[e_t\mid\mathcal H_t]$. Multiply by $\beta^t$,
take $\EE_{k-1}$, sum over $H\le t<H'$ and let $H'\to\infty$; all terms are bounded given $\cA_{k-1}$, and
$\beta^{H'}\EE_{k-1}\uV{}(\bx_{H'})\to0$, so $\varepsilon^k_H=\EE_{k-1}\sum_{t\ge H}\beta^t(r_t+e_t)$. For $H=0$ this is
$G^k$, as $\bx_0=x_0$; subtracting gives~\eqref{eq:inf:gap_identity}, and $\lambda_t\le\cb$ bounds $\varepsilon^k_H$. For
$\tau$, Fubini and independence give $\EE_{k-1}\sum_{t<\tau}(r_t+e_t)=\sum_t\beta^t\EE_{k-1}(r_t+e_t)$.
\end{proof}

\begin{proof}[Proof of \cref{lem:pass}]
The bounds stated after \cref{ass:Oinf} hold as in the proof of \cref{lem:basic}(i) (an $\Lh$-Lipschitz $f\le U$
satisfies $f\le\bar U$, and $\bar U\le M_U+\sup\uV0$), and give those on $\Phi^k$. Fix $t<\tau_k$, $x:=\bx^k_t$ and the
cut $f$ at $x$. By~\ref{OI2} and~\ref{OI4}, $\uV k(x)\ge f(x)\ge\cB\uV{k-1}(x)-\gamma_{k,t}=\uV{k-1}(x)+r^k(x)-\gamma_{k,t}$
and $\uV k\ge\uV{k-1}$, so $\uV k(x)\ge\uV{k-1}(x)+s^k_t$. \Cref{lem:lift}, applied to $V=\uV{k-1}$ and $W=\uV k$, gives
$\Phi^k-\Phi^{k-1}\ge\kappa\hat\varphi(s^k_t)$. Take the maximum over $t$; $\hat\varphi$ is nondecreasing and
$\hat\psi\le\hat\varphi$.
\end{proof}

\begin{lemma}[Pass lengths]\label{lem:length}
Let $(r_t)_{t\ge0}\subset[0,R]$, $\bar r_0:=0$, $\bar r_n:=\max_{t<n}r_t$ ($n\ge1$), and $\tau$ a random variable in
$\mathbb N^*$.
\begin{enumerate}[label=(\roman*),nosep]
\item $\EE\,\bar r_\tau=\sum_{s\ge0}\PP(\tau>s)\,(\bar r_{s+1}-\bar r_s)$.
\item If $\PP(\tau>s)\ge\beta^s$ for all $s\ge0$, then
$\EE\,\bar r_\tau\ge(1-\beta)\sum_{n\ge1}\beta^{n-1}\bar r_n\ge(1-\beta)\sum_{t\ge0}\beta^tr_t$. Conversely, if $R>0$ and
$\EE\,\bar r_\tau\ge(1-\beta)\sum_t\beta^tr_t$ for every such sequence, then $\PP(\tau>s)\ge\beta^s$ for all $s$.
Under this condition $\EE\tau\ge1/(1-\beta)$, with equality iff $\PP(\tau>s)=\beta^s$ for all $s$.
\item For a deterministic $\tau=H$:
$\bar r_H\ge\frac{1-\beta}{1-\beta^H}\sum_{t<H}\beta^tr_t\ge(1-\beta)\sum_{t<H}\beta^tr_t$.
\end{enumerate}
\end{lemma}

\begin{proof}
(i) Write $\bar r_n=\sum_{s<n}(\bar r_{s+1}-\bar r_s)$ and use Tonelli; the increments are nonnegative.
(ii) Hence $\EE\bar r_\tau\ge\sum_s\beta^s(\bar r_{s+1}-\bar r_s)=(1-\beta)\sum_{n\ge1}\beta^{n-1}\bar r_n$ (the
series converge since $\bar r$ is bounded); then use $\bar r_n\ge r_{n-1}$. For the converse, take
$r_t=R\one\{t\ge s\}$: then $\bar r_n=R\one\{n>s\}$, $\EE\bar r_\tau=R\,\PP(\tau>s)$ and
$(1-\beta)\sum_{t\ge s}\beta^tR=R\beta^s$. Finally $\EE\tau=\sum_s\PP(\tau>s)\ge\sum_s\beta^s$.
(iii) $\sum_{t<H}\beta^tr_t\le\bar r_H(1-\beta^H)/(1-\beta)$.
\end{proof}

\subsection{Periodic problems and proof of the master theorem}\label{app:periodic}

\paragraph{Periodic model.} \Cref{lem:std} and its proof hold for the periodic model, with
$\mathbf T:(W_p)_{p\in\mathbb Z_S}\mapsto(\cB_pW_{p+1})_{p\in\mathbb Z_S}$ in place of $\cB$, the norm
$\max_p\sup_{X_p}|W_p|$, the classes $\mathfrak M_p$, and, in~(iii), the cost from phase~$0$ and $\cB_{p(t)}$ at time $t$. In
the proof of \cref{lem:validIH}(a), adding a cut $f\le\cB_pW_{p+1}$ to $W_p$ gives a family $W'$ with
$W'_{p'}\le\cB_{p'}W_{p'+1}\le\cB_{p'}W'_{p'+1}$ for all $p'$, and $f\le\cB_pV_{p+1}=V_p$. \Cref{lem:gapIH,lem:pass}
hold with $r_t:=r^k_{p(t)}(\bx_t)$ and $\uV{k-1}_{p(t)}$ at time $t$. We identify $\mathbb Z_S$ with
$\{p\in\mathbb N:p<S\}$. The number of visits to a phase before the end of a geometric pass is not geometric, but
the summation by parts of \cref{lem:length} still applies.

\begin{lemma}[Pass lengths, phase by phase]\label{lem:length-phase}
Let $(r_t)_{t\ge0}\subset[0,R]$, let $\tau\in\mathbb N^*$ satisfy $\PP(\tau>s)=\beta^s$ for all $s\ge0$, and, for
$p\in\mathbb Z_S$, let $\bar r_{(p)}:=\max\{r_t:t<\tau,\ p(t)=p\}$, with $\max\emptyset:=0$. Then
$\EE\,\bar r_{(p)}\ge(1-\beta^S)\sum_{t\ge0:\,p(t)=p}\beta^tr_t$.
\end{lemma}

\begin{proof}
Let $q_j:=r_{p+jS}$ for $j\ge0$, $\bar q_0:=0$, $\bar q_n:=\max_{j<n}q_j$ for $n\ge1$, and $N:=\#\{j\ge0:p+jS<\tau\}$, so
that $\bar r_{(p)}=\bar q_N$ and $\PP(N>j)=\beta^{p+jS}$. The identity of \cref{lem:length}(i) holds for every random
variable with values in $\mathbb N$, so $\EE\,\bar r_{(p)}=\sum_{j\ge0}\beta^{p+jS}(\bar q_{j+1}-\bar q_j)
=\beta^p(1-\beta^S)\sum_{n\ge1}\beta^{S(n-1)}\bar q_n\ge\beta^p(1-\beta^S)\sum_{j\ge0}\beta^{jS}q_j$, since $\bar q$ is
bounded and $\bar q_n\ge q_{n-1}$.
\end{proof}

Let
\begin{equation}\label{eq:appD:periodic_chi}
\chi_S(g):=\max\Big\{\hat\psi\big((1-\beta)g\big),\ S\,\hat\psi\Big(\frac{(1-\beta^S)\,g}{S}\Big)\Big\}\qquad(g\ge0).
\end{equation}
The first term credits a pass with its largest residual, as in \cref{thm:master-ih}, against a potential $S$ times
larger; alone, it gives \cref{thm:det} with the periodic $\Mb_U$ for deterministic pass lengths. The second credits a
pass with one residual per phase, but a pass visits a given phase only about $1/(S(1-\beta))$ times, so the factor
$1-\beta$ becomes $(1-\beta^S)/S=(1-\beta)/q_S$, with $q_S$ as in \cref{thm:periodic}. Since $1-\beta^S=(1-\beta)\sum_{i<S}\beta^i\le
S(1-\beta)$, $q_S\ge1$; since $1-\beta^S\ge1-\beta$, $q_S\le S$; with $x:=S(1-\beta)$, $\ln\beta\le\beta-1$ gives
$\beta^S\le e^{-x}$, so $q_S\le x/(1-e^{-x})\le1+x$, because $(1+x)e^{-x}\le1$.

\begin{proof}[Proof of \cref{thm:periodic}, hence of \cref{thm:master-ih}(a)]
Fix $k\ge1$, couple the pass with an infinite trajectory $(\bx_t)_{t\ge0}$ of $\pi^k$ from $x_0$, independent of
$\tau_k$ given $\cA_{k-1}$, and let $r_t:=r^k_{p(t)}(\bx_t)$ and $s_t:=(r_t-\gamma_{k,t})_+$. On the
$\cA_{k-1}$-measurable event where $\sum_t\beta^t\gamma_{k,t}=\infty$, $\Gamma^\beta_k=\infty$ and there is nothing to
prove, since $\Phi^k\ge\Phi^{k-1}$; we work on its complement. Models, residuals and $s_t$ are bounded (\cref{lem:pass}).

\emph{Gap identity.} (i) is \cref{lem:gapIH} with $\tau=\tau_k$, with~\eqref{eq:inf:gap_meaning}, which follows from
\cref{lem:std}(iii) and $\uV{k-1}_0(x_0)\le\bar U_0(x_0)\le U_0(x_0)$. By \cref{lem:gapIH} with $H=0$ and $e_t\le\gamma^F_k$,
\begin{equation}\label{eq:appD:inexact_gap}
\EE_{k-1}\sum_{t\ge0}\beta^t(r_t-\gamma_{k,t})\ \ge\ G^k-\Gamma^\beta_k .
\end{equation}

\emph{Pathwise lift.} As in the proof of \cref{lem:pass}, at $t<\tau_k$, with $x:=\bx^k_t$ and $p:=p(t)$,
$\uV k_p(x)\ge\uV{k-1}_p(x)+s_t$ and $\uV k_p\ge\uV{k-1}_p$, so
$\int(\uV k_p-\uV{k-1}_p)\,d\mu_p\ge\kappa\hat\psi(\bar s_p)$ with $\bar s_p:=\max\{s_t:t<\tau_k,\ p(t)=p\}$ ($\max\emptyset:=0$).
Summing over $p$,
\begin{equation}\label{eq:appD:periodic_pathwise}
\Phi^k-\Phi^{k-1}\ \ge\ \Lambda_k:=\kappa\sum_{p\in\mathbb Z_S}\hat\psi(\bar s_p),
\end{equation}
which is~\eqref{eq:inf:periodic_lift}, and implies \cref{lem:pass} ($S=1$).

\emph{Lower bounds.} First, $\Lambda_k\ge\kappa\hat\psi(\bar s)$ with $\bar s:=\max_{t<\tau_k}s_t=\max_p\bar s_p$. Given
$\cA_{k-1}$ and the path, \cref{lem:length}(ii) gives $\EE[\bar s\mid\cA_{k-1},\text{path}]\ge(1-\beta)\sum_t\beta^ts_t$, so
$\EE_{k-1}\bar s\ge(1-\beta)(G^k-\Gamma^\beta_k)_+$ by $s_t\ge\max\{0,r_t-\gamma_{k,t}\}$ and~\eqref{eq:appD:inexact_gap}, and
conditional Jensen ($\hat\psi$ is convex and $\bar s$ is bounded) gives
$\EE_{k-1}\Lambda_k\ge\kappa\hat\psi\big((1-\beta)(G^k-\Gamma^\beta_k)_+\big)$. Second, by \cref{lem:length-phase},
$a_p:=\EE_{k-1}\bar s_p\ge0$ satisfies $\sum_pa_p\ge(1-\beta^S)\EE_{k-1}\sum_t\beta^ts_t\ge(1-\beta^S)(G^k-\Gamma^\beta_k)_+$,
and conditional Jensen, convexity and monotonicity of $\hat\psi$ give
$\EE_{k-1}\Lambda_k\ge\kappa\sum_p\hat\psi(a_p)\ge\kappa S\hat\psi\big(\sum_pa_p/S\big)\ge\kappa S\hat\psi\big((1-\beta^S)(G^k-\Gamma^\beta_k)_+/S\big)$.
Hence $\EE_{k-1}[\Phi^k-\Phi^{k-1}]\ge\EE_{k-1}\Lambda_k\ge\kappa\chi_S((G^k-\Gamma^\beta_k)_+)$.

\emph{Summation and rates.} Sum over $k$, with $0=\Phi^0\le\Phi^k\le\Mb_U$. On $\{G^k>\theta+\Gamma^\beta\}$ the $k$-th
term is at least $\kappa\chi_S(\theta)$, as $\chi_S$ is nondecreasing. If
$(1-\beta)\theta\le\hat\rho$, then $(1-\beta^S)\theta/S=(1-\beta)\theta/q_S\le\hat\rho$ too, both evaluations of
$\hat\psi$ are on its first branch, and $\chi_S(\theta)=\max\{1,S/q_S^{d+1}\}((1-\beta)\theta)^{d+1}/(2(d+1)\hat\rho^d)$.
For~(iii), apply \cref{lem:dbp} with $Y_k:=\Mb_U-\Phi^k$ ($Y_0=\Mb_U$ is deterministic) and
$f_k:=\kappa\chi_S((G^k-\Gamma^\beta_k)_+)$, which is $\cA_{k-1}$-measurable. For $S=1$, $q_S=1$ and these are the
statements of \cref{thm:master-ih}(a); the cut count follows from $\EE[\tau_k\mid\cA_{k-1}]=1/(1-\beta)$. For SDDP,
$\Mb_U\le S\cb/(1-\beta)$ and $q_S\le\min\{S,1+S(1-\beta)\}$ give the bound stated after \cref{thm:periodic}.
\end{proof}

\subsection{Covering versions}\label{app:ih-cover}

\begin{proof}[Proof of \cref{thm:master-ih}(b) and~(c)]
Let $Y_k:=\sum_{c\in C}(\bar U-\uV k)(c)$; by the bounds after \cref{ass:Oinf}, $(Y_k)$ is adapted, nonincreasing and
nonnegative, and $Y_0=M_C$ is deterministic. Fix $k\ge1$ and $t<\tau_k$, let $x:=\bx^k_t$, and choose $c\in C$ with
$\dist(x,c)\le\sigma$. As in the proof of \cref{lem:pass}, $(\uV k-\uV{k-1})(x)\ge s^k_t$; both models are
$\Lh$-Lipschitz, so $(\uV k-\uV{k-1})(c)\ge(s^k_t-2\Lh\sigma)_+$, and the other lifts are nonnegative. Different $t$
may give the same $c$, so only the maximum is available: $Y_{k-1}-Y_k\ge(\bar s^k-2\Lh\sigma)_+$. By conditional Jensen for $(\cdot)_+$, then
$\EE_{k-1}\bar s^k\ge(1-\beta)(G^k-\Gamma^\beta_k)_+$ (proof of \cref{thm:periodic}) and $2\Lh\sigma\le(1-\beta)\theta/2$,
\[
\EE_{k-1}[Y_{k-1}-Y_k]\ \ge\ \big(\EE_{k-1}\bar s^k-2\Lh\sigma\big)_+\ \ge\ (1-\beta)\big(G^k-\Gamma^\beta_k-\theta/2\big)_+=:f_k .
\]
Sum over $k\in[K]$ and let $K\to\infty$: $\EE\sum_kf_k\le Y_0=M_C$. The count follows from
$(\theta/2)\,\#\{k\ge1:G^k>\theta+\Gamma^\beta_k\}\le\sum_{k\ge1}(G^k-\Gamma^\beta_k-\theta/2)_+$, and the high-probability
claim is \cref{lem:dbp} applied to $(Y_k)$, with $C=M_C$ and these $f_k$. The net size is \cref{lem:nets}(a) with
$2D/\sigma=8\Lh D/((1-\beta)\theta)$. For~(c), $\bar U=U$ because $\dist=+\infty$ off the diagonal; take $C=X$, $c=x$,
$\sigma=0$ and $\theta=0$, so that the Lipschitz property is not used, with $f_k=(1-\beta)(G^k-\Gamma^\beta_k)_+$ and
$Y_0=C_{\rm fin}$. For periodic problems, take one net $C_p$ per phase, $Y_k:=\sum_p\sum_{c\in C_p}(\bar U_p-\uV k_p)(c)$
and $c\in C_{p(t)}$; the argument is unchanged.
\end{proof}

\subsection{Deterministic pass length}\label{app:det}

\begin{theorem}[Deterministic pass length]\label{thm:det}
Let \cref{ass:Sinf,ass:Oinf} and~\eqref{eq:setting:volume_growth} hold as in \cref{lem:pass}, with $\uV0\ge0$, and
let $\hat\rho$, $\hat\psi$, $\Gamma^\beta_k$ and $\Mb_U$ be as in \cref{sec:ih-lemmas,sec:ih-master}. Run variant~(D) with pass
length $H$; let $\varepsilon_H:=\beta^H\cb/(1-\beta)$ and $\varsigma\in(0,1)$.
\begin{enumerate}[label=(\roman*),nosep]
\item $\EE[\Phi^k-\Phi^{k-1}\mid\cA_{k-1}]\ge\kappa\hat\psi\big((1-\beta)(G^k-\Gamma^\beta_k-\varepsilon_H)_+\big)$ for $k\ge1$, so
the sum over $k$ of the expectations of the right-hand sides is at most $\Mb_U$. If $\Gamma^\beta_k\le\Gamma^\beta$
(deterministic), then, for $\theta>0$ with $\varepsilon_H\le\varsigma\theta$ and $(1-\beta)\theta\le\hat\rho$,
\begin{equation}\label{eq:appD:deterministic_rate}
\EE\,\#\{k\ge1:G^k>\theta+\Gamma^\beta\}\le\frac{2(d+1)\,\Mb_U\,\hat\rho^d}{\kappa(1-\varsigma)^{d+1}(1-\beta)^{d+1}\theta^{d+1}};
\end{equation}
for $\delta\in(0,1]$, with probability at least $1-\delta$, simultaneously for all $\theta$ with
$\varepsilon_H\le\varsigma\theta$, $\#\{k\ge1:G^k>\theta+\Gamma^\beta_k\}\le9\ln(3/\delta)\,\Mb_U/(\kappa\hat\psi((1-\varsigma)(1-\beta)\theta))$.
\item For $\Gamma^\beta_k\le\Gamma^\beta$, $0<\theta<\cb/(1-\beta)$ with $(1-\beta)\theta\le\hat\rho$ and
$H:=\big\lceil\ln\big((d+2)\cb/((1-\beta)\theta)\big)/(1-\beta)\big\rceil$,
\begin{equation}\label{eq:appD:deterministic_rate_tuned}
\EE\,\#\{k\ge1:G^k>\theta+\Gamma^\beta\}\le\frac{2e\,(d+1)\,\Mb_U\,\hat\rho^d}{\kappa(1-\beta)^{d+1}\theta^{d+1}},
\end{equation}
which is at most $2e(d+1)\cb(4\Lh D)^d/(\kappa(1-\beta)^{d+2}\theta^{d+1})$ if $\Mb_U\le\cb/(1-\beta)$, and the expected number
of cuts (trial-point cuts; the optional cuts of~\ref{OI4} are not counted) before the first iteration with $G^k\le\theta+\Gamma^\beta$ is at most $H$ times~\eqref{eq:appD:deterministic_rate_tuned}.
\end{enumerate}
\end{theorem}

\begin{proof}[Proof of \cref{thm:det}]
(i) Let $s_t:=s^k_t$ and $r_t:=r^k(\bx^k_t)$. By \cref{lem:pass} and \cref{lem:length}(iii),
$\Phi^k-\Phi^{k-1}\ge\kappa\hat\psi\big((1-\beta)\sum_{t<H}\beta^ts_t\big)$. By \cref{lem:gapIH},
$\EE_{k-1}\sum_{t<H}\beta^t(r_t-\gamma_{k,t})\ge G^k-\varepsilon^k_H-\Gamma^\beta_k$, and $\varepsilon^k_H\le\varepsilon_H$
as $\uV{k-1}\ge\uV0\ge0$; since $s_t\ge\max\{0,r_t-\gamma_{k,t}\}$, $\EE_{k-1}\sum_{t<H}\beta^ts_t\ge(G^k-\Gamma^\beta_k-\varepsilon_H)_+$.
Conditional Jensen gives the first claim; sum over $k$, with $0=\Phi^0\le\Phi^k\le\Mb_U$. On $\{G^k>\theta+\Gamma^\beta\}$, $(G^k-\Gamma^\beta-\varepsilon_H)_+\ge(1-\varsigma)\theta$, so the $k$-th term is at
least $\kappa\hat\psi((1-\varsigma)(1-\beta)\theta)=\kappa((1-\varsigma)(1-\beta)\theta)^{d+1}/(2(d+1)\hat\rho^d)$ (first branch).
For the high-probability bound, apply \cref{lem:dbp} with $Y_k:=\Mb_U-\Phi^k$; its regret is at least
$\kappa\hat\psi((1-\varsigma)(1-\beta)\theta)\,\#\{k\in[K]:G^k>\theta+\Gamma^\beta_k\}$, and the exceptional event does not depend on
$\theta$. (ii) $\beta^H\le e^{-(1-\beta)H}\le(1-\beta)\theta/((d+2)\cb)$, so $\varepsilon_H\le\theta/(d+2)$; take
$\varsigma=1/(d+2)$, which costs only $\ln(d+2)/(1-\beta)$ more steps per pass than $\varsigma=1/2$, and use $(1+1/(d+1))^{d+1}\le e$. The first index with $G^k\le\theta+\Gamma^\beta$ is at most
$1+\#\{k\ge1:G^k>\theta+\Gamma^\beta\}$.
\end{proof}

\subsection{Instances}\label{app:ih-instances}

\Cref{tab:ih-families} summarises the constants with which the oracles of this subsection satisfy \cref{ass:Oinf}.

\begin{table}[tb]
\centering\small
\setlength{\tabcolsep}{3pt}
\begin{tabular}{>{\raggedright\arraybackslash}p{3.8cm}>{\raggedright\arraybackslash}p{2.3cm}>{\raggedright\arraybackslash}p{1.7cm}>{\raggedright\arraybackslash}p{1.3cm}>{\raggedright\arraybackslash}p{4.3cm}l}
\toprule
oracle & $\Lh$ & $\gamma$ & $U$ & floor $\Gamma^\beta$ & part\\
\midrule
valid, $\gamma$-tight; Jensen & $\Lh$ & $\gamma$; $L^\xi e(\cP)$ & $V$ & $\gamma/(1-\beta)$ & (a)\\
SAA (\cref{prop:ih-saa}) & $\Lh$ & $(\Delta^+)_+$ & $\Vh$ & $(\Delta^+)_+/(1-\beta)$; $\zeta\le(\Delta^-)_+/(1-\beta)$ & (a)\\
SDDiP, binary states & not needed & $0$ & $V$ & $0$ & (c)\\
SLDP reverse norm & $\Lh\ge\Lip V$ & $0$ & $V$ & $0$ & (a), (b)\\
Zhang--Sun & $l_\lambda+l_\varrho$ & $0$, for $\BR$ & $V$; $V^R$ & $0$; $\zeta=V(x_0)-V^R(x_0)$; $0$ & (b)\\
\bottomrule
\end{tabular}
\caption{Oracles for the shared model: constants of \cref{ass:Oinf} and part of \cref{thm:master-ih} used. A
$\gamma^F$-greedy pass adds $\gamma^F/(1-\beta)$ to the floor. For the Jensen discretisation, $L^\xi$ and $e(\cP)$ are
as in \cref{prop:jensen}; for SAA, $\Vh$ is the fixed point for the sampled laws and $\Delta^\pm$ are the biases of
\cref{prop:ih-saa}. For Zhang--Sun, $G^k$ is the regularised gap, and $\BR$ and $V^R$ are the regularised operator and
its fixed point.}\label{tab:ih-families}
\end{table}

\begin{proof}[Proof of \cref{thm:geom}]
Under~\ref{I1}, $\cB$ maps convex functions in $\mathfrak M$ to convex functions on $X$ (on $\Oe$ under~\ref{Ieta}): the
partial infimum in $y$ of the jointly convex $\ell(z,y,\xi)+\beta W(y)$ is convex in $z$, and finite; integration keeps convexity. So the models are convex, and subgradient
cuts of $\cB W$ satisfy~\ref{I3} where they exist. \ref{OI1} is~\ref{I4}, with $\uV0$ continuous, hence bounded. By~\ref{I2},
\ref{I3} and \cref{lem:validIH}, \ref{OI2} holds with $\gamma_{k,t}=0$, \ref{OI3} with $U=V$, and $r^k\ge0$; \ref{OI4}
holds by construction. Then $\zeta=0$ and $\bar U\le V$, so $\Mb_U\le\Mb\le\cb/(1-\beta)$, as $V\le\cb/(1-\beta)$ and
$\uV0\ge0$, and~\eqref{eq:setting:volume_growth} is \cref{lem:VG}(a). Apply \cref{thm:master-ih}(a) with $\kappa=1$ and
$\Gamma^\beta_k=0$: (i) is \cref{lem:gapIH} with $e_t=0$ and~\eqref{eq:inf:gap_meaning}; (ii) uses the first branch of
$\hat\psi$; for~(iii), on the event of \cref{thm:master-ih}(a), some $k\le\lfloor K_\varepsilon\rfloor+1$ has
$G^k\le\varepsilon$, so $V(x_0)-\uV{k-1}(x_0)\le\varepsilon$, and $\uV j(x_0)$ is nondecreasing in $j$.
\end{proof}

\paragraph{Valid inexact cuts.} If $\uV0\le V$ and every cut added to phase $p$ is $\Lh$-Lipschitz, valid for the current models and
$\gamma$-tight at the trial point ($f\le\cB_pW_{p+1}$ on $X_p$ and $f(x)\ge\cB_pW_{p+1}(x)-\gamma$), then
\cref{lem:validIH} gives \cref{ass:Oinf} with $\gamma_{k,t}=\gamma$ and $U=V$. If moreover $\uV0$ is convex,
the Jensen discretisation of \cref{prop:jensen}, applied phase by phase to
$\xi\mapsto\inf_y[\ell_p(x,y,\xi)+\beta W(y)]$ for the convex $\Lh$-Lipschitz $W$ with $\uV0_{p+1}\le W\le V_{p+1}$,
gives such cuts with $\gamma\le L^\xi e(\cP)$, and the models stay convex.

\begin{proposition}[SAA backward passes, true-law forward passes]\label{prop:ih-saa}
Let $\Ph_p$ be laws on $\Xi_p$, fixed or drawn independently of the forward noises and of the pass lengths, with
periodic operators $\cBh_p$ as in \cref{sec:periodic} and fixed point $\Vh$. Assume \cref{ass:Sinf} phase by
phase for both laws, $\gamma^F$-greedy forward passes from the true laws, $\uV0_p$ and all cuts $\Lh$-Lipschitz for a norm,
$\uV0_p\le\Vh_p$, and cuts exact for $\cBh$ and the current models $W$ ($f\le\cBh_pW_{p+1}$ on
$X_p$, with equality at the trial point). With $\cW_p$ the set
of $\Lh$-Lipschitz $W:X_p\to\R$ with $\uV0_p\le W\le\Vh_p$, let, as in the SAA biases~\eqref{eq:noise:saa_biases},
$\Delta^+:=\max_p\sup\{(\cB_pW-\cBh_pW)(x):x\in X_p,\ W\in\cW_{p+1}\}$ and
$\Delta^-:=\max_p\sup_{X_p}(\cBh_pV_{p+1}-\cB_pV_{p+1})$. Then, conditionally on $\Ph$ (all statements hold under $\PP(\cdot\mid\Ph)$), \cref{ass:Oinf} holds
with $U=\Vh$ and $\gamma_{k,t}=(\Delta^+)_+$, so \cref{thm:periodic} applies under~\eqref{eq:setting:volume_growth}, with
$\Gamma^\beta_k=((\Delta^+)_++\gamma^F)/(1-\beta)$ and $\Mb_U\le\sum_p\int(\Vh_p-\uV0_p)\,d\mu_p$, and
$\zeta\le(\Delta^-)_+/(1-\beta)$. If the maps $\xi\mapsto\inf_{y\in\cX_p(x,\xi)}[\ell_p(x,y,\xi)+\beta W(y)]$, for
$x\in X_p$ and $W\in\cW_{p+1}\cup\{V_{p+1}\}$, are $L^\xi$-Lipschitz on $\Xi_p$, then $\Delta^\pm\le L^\xi\max_p\Wa(\PP_p,\Ph_p)$.
\end{proposition}

For empirical laws of $N$ i.i.d.\ draws, \cref{prop:uniformSAA} applies with a union bound over the $S$ phases (the net
condition holds for convex $X_p$ by \cref{lem:nets}(a)), the deterministic class of $\Lh$-Lipschitz $W$ with
$\inf\uV0_{p+1}\le W\le\cb/(1-\beta)$, which contains $\cW_{p+1}\cup\{V_{p+1}\}$ if $\inf\uV0_{p+1}\le0$ (e.g.\ $\uV0=0$) and $V_{p+1}$ is $\Lh$-Lipschitz, and
scenario values Lipschitz in $x$: $\Delta^\pm=O(N^{-1/(d+2)}\sqrt{\ln N})$, with constants depending on $d$, $D$, $\Lh$,
the range in $\xi$ and the Lipschitz constant in $x$ of the scenario values, and only logarithmically on $\cb/(1-\beta)$.

\begin{proof}[Proof of \cref{prop:ih-saa}]
As for \cref{thm:saa}, apply \cref{lem:validIH}(a) (periodic) to the problem with laws $\Ph$: since $\uV0\le\Vh$ and every cut satisfies
$f\le\cBh_pW_{p+1}$, every model satisfies $W_p\le\Vh_p$, so~\ref{OI3} holds with the ceiling $U=\Vh$, and every model
of phase $p$ lies in $\cW_p$ by~\ref{OI1}. At $x=\bx^k_t$, the current models satisfy $W\ge\uV{k-1}$, so
$f(x)=\cBh_pW_{p+1}(x)\ge\cBh_p\uV{k-1}_{p+1}(x)\ge\cB_p\uV{k-1}_{p+1}(x)-\Delta^+$, because $\uV{k-1}_{p+1}\in\cW_{p+1}$:
this is~\ref{OI2} with the deterministic error $(\Delta^+)_+$. Also $\bar U_p\le\Vh_p$.
For $\zeta$, let $e:=\max_p\sup_{X_p}(\Vh_p-V_p)_+$. Since $\cBh_p$ is monotone and
$\cBh_p(W+c)=\cBh_pW+\beta c$ for constants $c$,
$\Vh_p-V_p=(\cBh_p\Vh_{p+1}-\cBh_pV_{p+1})+(\cBh_pV_{p+1}-\cB_pV_{p+1})\le\beta e+\Delta^-$, so
$e\le\beta e+(\Delta^-)_+$, that is, $\zeta\le e\le(\Delta^-)_+/(1-\beta)$. For the Wasserstein bound, $(\cB_pW-\cBh_pW)(x)$
is the integral of an $L^\xi$-Lipschitz function against $\PP_p-\Ph_p$; use Kantorovich--Rubinstein duality.
\end{proof}

\begin{proposition}[Shifted sampled cuts]\label{prop:ih-sampled}
Assume \cref{ass:I} phase by phase for general laws, without~\ref{I3}, with $\gamma^F$-greedy forward
passes from the true laws. At $\bx^k_t$, add the shifted sampled cut of \cref{app:sampled} for $\cB_{p(t)}$ and the
current models, built from a deterministic number $N_{k,t}$ of fresh i.i.d.\ draws from $\PP_{p(t)}$ and the scenario oracle of \cref{ass:scen}
(values and subgradients of norm at most $\Lh$ of $\inf_y[\ell_p(\cdot,y,\xi)+\beta W(y)]$ for every convex
$\Lh$-Lipschitz $W$ with $\uV0_{p+1}\le W\le V_{p+1}$, with range at most
$\mathrm{rg}$ in $\xi$, for instance $\mathrm{rg}\le\cb+\Lh D$), with
the shift $b_{k,t}:=(\mathrm{rg}+2\Lh D)\sqrt{\ln(4/\delta_{k,t})/N_{k,t}}$, $\delta_{k,t}:=\delta'/(k(k+1)(t+1)(t+2))$ and
$\delta'\in(0,1]$. Let $\Gamma^\beta_k:=2\sum_{t\ge0}\beta^tb_{k,t}+\gamma^F/(1-\beta)$ and
$M^V:=\sum_{p\in\mathbb Z_S}\sup_{X_p}(V_p-\uV0_p)$. There is an event $A$ with $\PP(A)\ge1-\delta'$ on which every cut is
valid, and $\EE\big[\one_A\sum_{k\ge1}\chi_S((G^k-\Gamma^\beta_k)_+)\big]\le(1+\delta'/2)\,M^V$. If
$N_{k,t}\ge\big(2(\mathrm{rg}+2\Lh D)/((1-\beta)g)\big)^2\ln\big(4k(k+1)(t+1)(t+2)/\delta'\big)$ for all $t$, then
$\Gamma^\beta_k\le g+\gamma^F/(1-\beta)$. For $g=\erel\cb/(1-\beta)$ and $\mathrm{rg}\le\cb+\Lh D$, this sample size suffices:
$\big(2(1+3\Lh D/\cb)/\erel\big)^2\ln\big(4k(k+1)(t+1)(t+2)/\delta'\big)$; for a $\beta$-uniform $\Lh$, it is independent of $\beta$ except through the
logarithm, and of the noise dimension.
\end{proposition}

\begin{proof}
We follow the proof of \cref{thm:sampled}, with $A_{k,t}:=\{\sup_{X_{p(t)}}|\hat f-m|\le b_{k,t}\}$ for the cut at $\bx^k_t$,
$\Omega_0$ the whole space, $\Omega_k:=\Omega_{k-1}\cap\bigcap_{t<\tau_k}A_{k,t}$ and $A:=\bigcap_k\Omega_k$. On $\Omega_{k-1}\cap\bigcap_{s>t}A_{k,s}$
all cuts so far are valid (\cref{lem:validIH}(a), with $\uV0\le\cB\uV0$ by~\ref{I2}), so the current models,
maxima of the convex $\uV0$ and of affine $\Lh$-Lipschitz cuts, lie in the class of the oracle, and on $A_{k,t}$ the shifted
cut is valid and at least $\cB_p\uV{k-1}_{p+1}(x)-2b_{k,t}$ at $x=\bx^k_t$, and \cref{lem:conc} gives
$\PP(\Omega_{k-1}\setminus \Omega_k\mid\cA_{k-1})\le\sum_{t\ge0}\delta_{k,t}=\delta'/(k(k+1))$, so $\PP(A^c)\le\delta'$. In the
stopped-potential argument, replace $\Psi_T((R^k-\Gamma_k)_+)$ by $\Lambda_k$ of~\eqref{eq:appD:periodic_pathwise} with
$\kappa=1$ and $\gamma_{k,t}:=2b_{k,t}$, which is at most $\sum_p\bar s_p/2\le M^V/2$ on $\Omega_{k-1}$ (there $\uV{k-1}\le V$, and
$\hat\psi(u)\le u/2$), and \cref{lem:gap} by the two lower bounds of the proof of \cref{thm:periodic}, which only use
the definition of $\Lambda_k$: $\EE[\one_{\Omega_{k-1}}\Lambda_k\mid\cA_{k-1}]\ge\one_{\Omega_{k-1}}\chi_S((G^k-\Gamma^\beta_k)_+)$.
Finally, $b_{k,t}\le(1-\beta)g/2$ for all $t$ gives $2\sum_t\beta^tb_{k,t}\le g$, and this is the condition on $N_{k,t}$.
\end{proof}

\begin{proposition}[Other oracles]\label{prop:ih-families}
Let \cref{ass:Sinf} hold, and let every cut be computed for the current model $W$ at $x=\bx^k_t$.
\begin{enumerate}[label=(\alph*),nosep]
\item (SDDiP.) Let $X\subset\{0,1\}^n$, the noise be finitely supported, the stages be mixed-integer linear with
complete continuous recourse, $\uV0=0$, and each cut be a Lagrangian cut of $\cB W$ with an optimal multiplier or an
integer optimality cut \citep{zou2019sddip}, possibly with further cuts $f'\le\cB W$. Then \cref{ass:Oinf} holds for
$\dist=+\infty$ off the diagonal, with $U=V$ and $\gamma_{k,t}=0$, and \cref{thm:master-ih}(c) gives
$\EE\sum_{k\ge1}(G^k-\gamma^F_k/(1-\beta))_+\le|X|\cb/(1-\beta)^2\le2^n\cb/(1-\beta)^2$.
\item (Reverse-norm cuts.) If $\dist$ is a norm, deterministic $\uV0\le V$, $\uV0$ and $V$ are $\Lh$-Lipschitz, and
$f(y):=\cB W(x)-\Lh\|y-x\|$, then \cref{ass:Oinf} holds with $U=V$ and $\gamma_{k,t}=0$.
\item (Zhang--Sun, shared-model version.) In the stationary setting of \cref{prop:ZS} (finite noise, compact $X$, nodal
costs in $[0,\cb]$ where finite, penalties $\varpi_\xi\mathcal R_\xi\ge0$, $\mathcal R_\xi(0)=0$), let the forward pass be greedy
for $\BR W(\bar x):=\int\inf_{z\in X}[\varpi_\xi\mathcal R_\xi(\bar x-z)+\Bd(W;z,\xi)]\,\PP(d\xi)$, the cut average the
generalised conjugacy cuts of $\Bd(W;\cdot,\xi)$ at $\bar x=\bx^k_t$, and $\uV0=0$. Then the regularised problem satisfies
\cref{ass:Sinf} with $\cb+\max_\xi\varpi_\xi\sup_{X-X}\mathcal R_\xi$ for $\cb$ (given~\ref{SI3}; $\ell^R$ below is lower
semicontinuous under their assumptions), $V^R\le V$, and \cref{ass:Oinf} holds for $\BR$, with $G^k$ the regularised gap
(losses $\ell^R$), $\Lh=\sum_\xi\PP(\xi)(l_{\xi,\lambda}+l_{\xi,\varrho})$, $\gamma_{k,t}=0$ and $U=V$, so
$\zeta=V(x_0)-V^R(x_0)$, or, under \cref{lem:ZSreg}, $U=V^R$ and $\zeta=0$.
\end{enumerate}
\end{proposition}

\begin{proof}
(a) $X$ and the noise are finite, so~\ref{SI3} is not needed, every function on $X$ is $0$-Lipschitz for this metric,
and $\bar U=U$. Validity on $X$ and tightness at binary trial points are properties of one nodal problem, here with
$\beta W$, a maximum of affine functions, as cost-to-go term: \citet[Thm.~3]{zou2019sddip} for Lagrangian cuts with
optimal multipliers, and their~(4.2) for integer optimality cuts. So \cref{lem:validIH} gives~\ref{OI2} with
$\gamma_{k,t}=0$ and~\ref{OI3} with $U=V$, and $C_{\rm fin}=\sum_xV(x)\le|X|\cb/(1-\beta)$. (b) $f$ is continuous and
$\Lh$-Lipschitz, and $f(x)=\cB W(x)$, so~\ref{OI2} holds by \cref{lem:validIH}(b). By induction, if $W\le V$, then
$f(y)\le\cB V(x)-\Lh\|y-x\|=V(x)-\Lh\|y-x\|\le V(y)$, so every model is $\le V$. \emph{Stage-wise absorption.} If
$\cX(\cdot,\xi)$ is constant on $X$ and $\ell(\cdot,y,\xi)$ is $c$-Lipschitz for $\dist$, comparing the infima over the
common set gives $\Bd(W;x,\xi)-\Bd(W;x',\xi)\le c\dist(x,x')$ for $W\in\mathfrak M$, so $\cB W$ and $V=\cB V$ are
$c$-Lipschitz, uniformly in $\beta$, and~(b) applies with $\Lh=c$ and $\uV0=0$. (c) With
$\ell^R(\bar x,y,\xi):=\inf\{\ell(z,y,\xi)+\varpi_\xi\mathcal R_\xi(\bar x-z):z\in X,\ y\in\cX(z,\xi)\}$ and
$\cX^R(\bar x,\xi):=\bigcup_{z\in X}\cX(z,\xi)\subset X$, $\BR$ is the operator~\eqref{eq:inf:bellman_operator} of the
regularised problem (\cref{app:ZS}), and $0\le\ell^R\le\cb+\varpi_\xi\sup_{X-X}\mathcal R_\xi$. The checks
of~\ref{O1}--\ref{O4} in the proofs of \cref{prop:ZS,lem:ZSreg} concern one nodal problem; with $\beta W$ as cost-to-go
term, they give $\Lh$-Lipschitz averaged cuts $f\le\cB W$ (and $f\le\BR W$ under the conditions of \cref{lem:ZSreg}) with
$f(\bar x)\ge\BR W(\bar x)$. So~\ref{OI1} holds, and \cref{lem:validIH}, for $\BR$ for~\ref{OI2} and for $\cB$ (respectively
$\BR$) for~\ref{OI3}, gives the rest. Finally $z=\bar x$ is feasible, so $\BR W\le\cB W$, and $V^R=\BR V^R\le\cB V^R$
gives $V^R\le V$ by \cref{lem:std}(ii).
\end{proof}

\subsection{Lipschitz constants}\label{app:ih-lip}

\paragraph{Idea of the proof, shared model.} The proof of \cref{prop:lipiter} bounds the slope of a cut of $\cB W$ by
the oscillation of $\cB W$ on $\Oe$ divided by $\eta$ (\cref{lem:slope}), and the oscillation by optional stopping of
the submartingale $\sum_{u<s}\beta^u\ell_u+\beta^sW(\bx_s)$ ($W\le\cB W$). The same $W$ appears in
$\cB W(x)\ge\beta\inf W$ and in $\cB W(z)\le m\cb+\beta\inf W+\beta\Lh\chi$, so $\inf W$ cancels. For stage-wise functions
no such cancellation is available when $m\ge2$, and a $\beta$-uniform bound for the truncation of
\cref{sec:trunc} is not established here. \Cref{rem:lipex} gives examples of $(R_{m,0})$.

\begin{lemma}[Slopes of a cut]\label{lem:slope}
Assume \ref{I1} and \ref{Ieta}. Let $W:X\to[0,\infty)$ be convex and bounded, $x\in X$, and $f$ affine
with $f\le\cB W$ on $\Oe$ and $f(x)=\cB W(x)$. Let $\Theta:=\sup_{z\in\Oe}\cB W(z)-\cB W(x)$. Then $f$ is
$(\Theta/\eta)$-Lipschitz on $\aff X$. Moreover:
\begin{enumerate}[label=(\roman*),nosep]
\item $\Theta\le\cb+\beta\,(\sup_X W-\inf_X W)$;
\item under $(R_{1,0})$ (\cref{def:reach}), $\Theta\le\cb$;
\item if $W\le\cB W$ on $X$, $W$ is $L_W$-Lipschitz on $X$ and $(R_{m,\chi})$ holds, then
$\Theta\le m\cb+\beta L_W\chi$.
\end{enumerate}
\end{lemma}

\begin{proof}
\emph{Lipschitz constant.} Write $f=f(x)+\langle g,\cdot-x\rangle$, and let $\Pi$ be the orthogonal projection
onto the direction space of $\aff X$. If $\Pi g\ne0$, the point $z:=x+\eta\,\Pi g/\|\Pi g\|$ lies in $\Oe$, so
$\cB W(x)+\eta\|\Pi g\|=f(z)\le\cB W(z)$, that is, $\|\Pi g\|\le\Theta/\eta$. For $y,y'\in\aff X$,
$|f(y)-f(y')|=|\langle\Pi g,y-y'\rangle|$. \emph{Lower bound at $x$.} Losses are nonnegative and
$\cX(x,\xi)\subset X$, so $\cB W(x)\ge\beta\inf_X W$. (i) For $z\in\Oe$, $\cX(z,\xi)\subset X$, so
$\cB W(z)\le\cb+\beta\sup_X W$. (ii) For $z\in\Oe$, $\cX(z,\xi)=X$, so $\cB W(z)\le\cb+\beta\inf_X W$.
(iii) Fix $z\in\Oe$, $\epsilon_0>0$ and $y\in X$ with $W(y)\le\inf_X W+\epsilon_0$. Take the policy and the stopping
time of $(R_{m,\chi})$, write $\ell_s:=\ell(\bx_s,\bx_{s+1},\bxi_{s+1})$, and let $Y_s:=\sum_{u<s}\beta^u\ell_u+\beta^sW(\bx_s)$ for
$s\ge1$. For $s\ge1$, $\bx_s\in X$. Since $\bxi_{s+1}$ is independent of $\bxi_{[s]}$ and $\bx_{s+1}$ is a feasible
selection, $\EE[Y_{s+1}-Y_s\mid\bxi_{[s]}]\ge\beta^s(\cB W(\bx_s)-W(\bx_s))\ge0$. So $(Y_{s\wedge\sigma})_{s\ge1}$ is a
submartingale, and optional stopping at the bounded time $\sigma$ gives
$\cB W(z)\le\EE[\ell_0+\beta W(\bx_1)]=\EE Y_1\le\EE Y_\sigma$. Now $Y_\sigma\le m\cb+\beta^\sigma W(\bx_\sigma)$ and
$W(\bx_\sigma)\le W(y)+L_W\chi$; since $W(y)+L_W\chi\ge0$ and $\sigma\ge1$,
$\cB W(z)\le m\cb+\beta(\inf_X W+\epsilon_0+L_W\chi)$. Subtract the lower bound at $x$ and let $\epsilon_0\to0$.
\end{proof}

\begin{proof}[Proof of \cref{prop:lipiter}]
By \cref{lem:validIH}(a), \cref{lem:std}(ii) and the convexity shown in the proof of \cref{thm:geom}, every current
model $W$ is convex and satisfies $0\le\uV0\le W\le V\le\cb/(1-\beta)$ and $W\le\cB W$. (i) \Cref{lem:slope}(i) gives $\Theta\le\cb+\beta\cb/(1-\beta)=\cb/(1-\beta)$. (ii) By induction on the
cuts, in order of creation. Let $\Lh:=m\cb/(\eta-\chi)$. If $\uV0$ and the cuts created before are
$\Lh$-Lipschitz, so is the current $W$, and \cref{lem:slope}(iii) bounds the slope of the new cut by
$(m\cb+\beta\Lh\chi)/\eta\le(m\cb+\Lh\chi)/\eta=\Lh$. (iii) \Cref{lem:slope}(ii).
\end{proof}

\paragraph{Periodic problems.} Assume \ref{I1}--\ref{I3} phase by phase, with~\ref{Ieta} for the data of each phase
($\cX_p(z,\xi)\subset X_{p+1}$ for $z$ in the neighbourhood of $X_p$), so that every family of models satisfies
$0\le W_{p'}\le\cB_{p'}W_{p'+1}$, and let each $\uV0_p$ be $L_0$-Lipschitz. Then \cref{prop:lipiter}(i)
and~(iii) hold phase by phase, with $(R_{1,0})$ read as $\cX_p(z,\xi)=X_{p+1}$. For~(ii), assume that, for every $p$, every $z$ in the
neighbourhood of $X_p$ and every $y\in X_{p+1}$, some policy started at $z$ at a time of phase $p$ and a stopping time
$\sigma\equiv1\pmod S$ with $1\le\sigma\le m$ give $\|\bx_\sigma-y\|\le\chi<\eta$ almost surely, and $L_0\le m\cb/(\eta-\chi)$. Then every cut is $m\cb/(\eta-\chi)$-Lipschitz: in the induction of
\cref{prop:lipiter}(ii), apply the proof of \cref{lem:slope}(iii) to $Y_s:=\sum_{u<s}\beta^u\ell_u+\beta^sW_{p+s}(\bx_s)$,
a submartingale since $W_{p'}\le\cB_{p'}W_{p'+1}$; $\sigma\equiv1\pmod S$ makes $\inf_{X_{p+1}}W_{p+1}$ cancel. Unless
one step suffices, this forces $m\ge S+1$, a constant of the order of the loss of a period divided by $\eta$; whether
this is necessary is open.

\begin{remark}[Examples]\label{rem:lipex}
\emph{Inventory with backlog} (the newsvendor of \citet[eq.~(30)]{julan2023infinite}, with $n=1$): state
$x\in X=[-b',b]$, demand $\xi\ge0$, next state $y\in[\max(-b',x-\xi),b]$, convex ordering, holding and backlog
costs. \ref{Ieta} at $z=b+\eta$ forces $\xi\ge\eta$ a.s., so the rule $\bx_{s+1}=\max(y,\bx_s-\bxi_{s+1})$ reaches any lower
target $y$, going down by at least $\eta$ per step, and orders reach any higher level in one step: $(R_{m,0})$ holds
with $m=\lceil(b+b')/\eta\rceil+1$. \emph{Hydro reservoir with free spillage}: next level
$y\in[0,\min(b,x+\xi)]$, inflow $\xi\ge0$, convex thermal cost. \ref{Ieta} at $z=-\eta$ forces $\xi\ge\eta$, and the
rule $\bx_{s+1}=\min(y,\bx_s+\bxi_{s+1})$ gives $(R_{m,0})$ with $m=\lceil b/\eta\rceil+1$. In these examples~\ref{Ieta}
alone gives a $\beta$-uniform bound of order $(D/\eta)\cb/\eta$. Some neighbourhood
condition is needed: $\cX(q)=[q,1]$ on $[0,1]$ with loss $q'$ gives $V(q)=q/(1-\beta)$, a Lipschitz constant of order
$1/(1-\beta)$; this example violates feasibility at $q>1$, hence~\ref{Ieta} and Ju--Lan's Assumption~2.2.
\end{remark}

\subsection{Non-stationary truncation and the warning instance}\label{app:trunc}

Let $\mathrm{Pb}_H$ be the $H$-stage problem of \cref{sec:setting} with stage losses $\beta^{t-1}\ell$, $X_0=\{x_0\}$,
$X_t=X$ for $t\ge1$, and $V_H=0$. Run SDDP on $\mathrm{Pb}_H$ with $\lv t0=0$ and exact cuts (variant~(N)).
Write $\mathcal V^k_t:=\beta^{-t}\lv tk$ for the normalised functions, let $\pi^{k,H}$ be the
non-stationary greedy policy of iteration $k$ and $G^k_H:=V^{\pi^{k,H}}_0(x_0)-\lv0{k-1}(x_0)$, with $V^{\pi^{k,H}}_0$ the $H$-stage cost of this policy.

\begin{theorem}[Truncation]\label{thm:trunc}
Assume \ref{I1}, \cref{ass:S} for $\mathrm{Pb}_H$, and that the normalised stage cuts are $\Lh$-Lipschitz on $X$ (under \ref{Ieta},
with stage cuts valid on $\Oe$, this holds with $\Lh=\cb/((1-\beta)\eta)$, and with $\Lh=\cb/\eta$ under
$(R_{1,0})$).
\begin{enumerate}[label=(\roman*),nosep]
\item For every $H$ and every $\varepsilon'>0$ with $(1-\beta)\varepsilon'\le\hat\rho$ ($\hat\rho=4\Lh D$),
$\EE\,\#\{k\ge1:G^k_H>\varepsilon'\}\le2(d+1)\cb\hat\rho^d/((1-\beta)^{d+2}\varepsilon'^{d+1})$ (\cref{cor:disc}).
\item Let $\bar\pi^k$ be any policy that applies $\pi^{k,H}$ for the first $H$ decisions and then anything
feasible (for instance, restarts $\pi^{k,H}$ every $H$ steps). Then $\lv0{k-1}(x_0)\le V(x_0)$ and
$V^{\bar\pi^k}(x_0)-V(x_0)\le V^{\bar\pi^k}(x_0)-\lv0{k-1}(x_0)\le G^k_H+\varepsilon_H$.
\item Hence, for $0<\varepsilon<\cb/(1-\beta)$ with $(1-\beta)\varepsilon\le\hat\rho$ and $H$ as in \cref{thm:det}(ii),
$\EE\,\#\{k\ge1:V^{\bar\pi^k}(x_0)-\lv0{k-1}(x_0)>\varepsilon\}\le2e(d+1)\cb\hat\rho^d/((1-\beta)^{d+2}\varepsilon^{d+1})$,
with $H$ cuts per iteration (trial-point cuts; the optional cuts of~\ref{O4} are not counted).
\end{enumerate}
\end{theorem}

\begin{proof}
(i) The assumptions of \cref{cor:disc} hold, with constant $\beta^{t}\Lh$ for the cuts of index $t$, $\lv{}0=0$, and exact cuts.
(ii) The cost incurred after time $H$ is at most $\beta^H\cb/(1-\beta)$. By \cref{lem:valid}(a),
$\lv0{k-1}(x_0)\le\tilde V^{(H)}_0(x_0)$, where $\tilde V^{(H)}_0=\cB^H0\le\cB^HV=V$. (iii) If
$V^{\bar\pi^k}(x_0)-\lv0{k-1}(x_0)>\varepsilon$, then $G^k_H>\varepsilon-\varepsilon_H\ge\frac{d+1}{d+2}\varepsilon$; apply~(i). The Lipschitz
constants come from \cref{lem:slope}(i)--(ii) applied to $W=\mathcal V_{t+1}\in[0,\cb/(1-\beta)]$.
\end{proof}

\begin{proposition}[Warning: the greedy policy of $\mathcal V_1$ can stay suboptimal]\label{prop:counter-ih}
Let $x=(u,w)\in[0,1]^2$, $x_0=(0,0)$, deterministic dynamics $\cX(u,w)=\{(1-u,w'):0\le w'\le(1+u)/2\}$, loss
$\ell(x,y)=|w-u|+\alpha(1-w')$ with $0<\alpha\le\beta/2$, and $\cb=1+\alpha$; this satisfies~\ref{I1}. Run~(N) with any
$H\ge2$ and $\lv t0=0$, and let $\hat\pi^k:=\pi_{\mathcal V^k_1}$ be the greedy stationary policy of
$\mathcal V^k_1$ after $k\ge0$ iterations.
\begin{enumerate}[label=(\roman*),nosep]
\item For every exact-cut oracle whose cuts have $w$-slopes in $[-1,1]$ on the slices below (for instance,
subgradients of the natural extension of $\cB W$ to $w\in\R$), \cref{cor:disc} applies ($d=1$, $D=1$, $\Lh=1$), so
almost surely $G^k_H\le\varepsilon'$ for all large $k$, for every $\varepsilon'>0$.
\item For every exact-cut oracle that adds, at index~$1$, only the cut at the trial point, and every $k$, $\hat\pi^k$ is the policy $(u,w)\mapsto(1-u,(1+u)/2)$, and
$V^{\hat\pi^k}(x_0)-V(x_0)\ge\beta(\beta-\alpha)/(1-\beta^2)\ge\beta^2/(4(1-\beta))$.
\end{enumerate}
\end{proposition}

\begin{proof}
Use the slices $X_0=\{x_0\}$, $X_t=\{1\}\times[0,1]$ ($t$ odd) and $\{0\}\times[0,1]$ ($t\ge2$ even): convex, compact,
forward invariant, with $d_t\le1$ and $D_t\le1$. (i) Cuts with $w$-slopes in $[-1,1]$ are $1$-Lipschitz on the slices
(for the natural extension, $\cB W(u,w)=|w-u|+\min_{0\le w'\le(1+u)/2}[\alpha(1-w')+\beta W(1-u,w')]$ has $w$-slopes in
$[-1,1]$); apply \cref{cor:disc} and Borel--Cantelli (a finite expected count). (ii) By induction on $k$, suppose $\mathcal V^k_1(1,\cdot)$ is nonincreasing (true
for $k=0$). The first decision minimises $\alpha(1-w_1)+\beta\mathcal V^k_1(1,w_1)$ over $w_1\in[0,1/2]$, a strictly
decreasing function, so $\bx_1=(1,1/2)$. On the slice $\{u=1\}$, $\cB\mathcal V_2(1,w)=(1-w)+\mathrm{const}$ is
differentiable at the interior point $w=1/2$, so an affine $f\le\cB\mathcal V_2$ on the slice that is tight at $(1,1/2)$
has $w$-slope $-1$, whatever its $u$-slope. Hence $\mathcal V^{k+1}_1=\max(0,\text{cuts})$ is nonincreasing in $w$ on
\emph{every} line $u=\mathrm{const}$, which closes the induction, and at any $(u,w)$ the objective
$\alpha(1-w')+\beta\mathcal V^k_1(1-u,w')$ of $\hat\pi^k$ is strictly decreasing in $w'$, so $w'=(1+u)/2$. From $x_0$, $\hat\pi^k$ alternately visits $(1,\frac12)$ and $(0,1)$, starting with
$(1,\frac12)$, with costs $\alpha/2$ at $t=0$, $1/2$ at odd $t$ and $1+\alpha/2$ at even $t\ge2$, so
$(1-\beta^2)V^{\hat\pi^k}(x_0)=\alpha/2+\beta/2+\beta^2$. The policy that alternates $(0,0)\to(1,\frac12)\to(0,0)$ costs
$\alpha/2$ at even $t$ and $1/2+\alpha$ at odd $t$, so $(1-\beta^2)V(x_0)\le\alpha/2+\beta/2+\beta\alpha$. Subtract, and use
$1+\beta\le2$ and $\alpha\le\beta/2$.
\end{proof}

\subsection{Tightness of the convex-regret form}\label{app:tight}

\begin{proposition}[Tightness of the convex-regret form]\label{prop:tight}
Let $X=[0,1]$, $x_0=0$, deterministic dynamics $\cX(z)=X$ for all $z\in\Oe$, loss
$\ell(z,y)=(r_0-az)_++\alpha|y-z|$ with $0<r_0\le a$ and $0<\alpha<\beta a$, and $\uV0=0$. Then \ref{I1}, \ref{I2}, \ref{Ieta} and
$(R_{1,0})$ hold, and every cut of every possible model is $(a+\alpha)$-Lipschitz, so $\Lh=a+\alpha$ and $D=1$ are
valid. Moreover (i) $\pi^1$ stays at $0$ forever and $G^1=r_0/(1-\beta)$; (ii) $V(x)=(1+\alpha/a)(r_0-ax)_+$ on $X$,
so in variants~(D) and~(G), for every run, $\sup_k\Phi^k\le\avg_X V=(1+\alpha/a)\,r_0^2/(2a)$. Consequently, if
for some $\vartheta>0$ the bound $\EE\sum_{k\ge1}\hat\psi(\vartheta G^k)\le\avg_X(V-\uV0)$ held for all instances, then
$\vartheta\le2\sqrt2(1+\alpha/a)(1-\beta)$. The same holds for any per-iteration bound
$\EE_{k-1}[\Phi^k-\Phi^{k-1}]\ge\hat\psi(\vartheta G^k)$.
\end{proposition}

\begin{proof}
For every $W$, $\cB W(z)=(r_0-az)_++\min_{y\in X}[\alpha|y-z|+\beta W(y)]$ is the sum of an $a$-Lipschitz and an
$\alpha$-Lipschitz function on $\Oe$, so every subgradient at a point of $X\subset\operatorname{int}\Oe$ has norm at
most $a+\alpha$. (i) With $\uV0=0$, the greedy objective is $\alpha|y-x|$, uniquely minimised at $y=x$, so $\bx_t=0$ and
$V^{\pi^1}(0)=r_0/(1-\beta)$. (ii) Let $U(x):=(1+\alpha/a)(r_0-ax)_+$. For $x\ge r_0/a$, $y=x$ gives $\cB U(x)=0=U(x)$. For
$x<r_0/a$: on $y\in[x,r_0/a]$ the objective $\alpha(y-x)+\beta(a+\alpha)(r_0/a-y)$ is decreasing, since
$\alpha<\beta(a+\alpha)$; on $y\ge r_0/a$ it is increasing, and on $y\le x$ it is decreasing in $y$. So the minimum is
$\alpha(r_0/a-x)$, and $\cB U(x)=r_0-ax+\alpha(r_0/a-x)=U(x)$. Hence $U=V$, and $\Phi^k\le\avg V$ by \cref{lem:validIH}(a).
The average is $(1+\alpha/a)\int_0^{r_0/a}(r_0-as)\,ds$. \emph{Consequence.} Here $\hat\rho=4(a+\alpha)$, and the $k=1$
term alone gives $\hat\psi(s)\le(1+\alpha/a)r_0^2/(2a)$ with $s:=\vartheta r_0/(1-\beta)$. If $s>\hat\rho$, then
$\hat\psi(s)>\hat\psi(\hat\rho)=a+\alpha>(1+\alpha/a)a/2\ge(1+\alpha/a)r_0^2/(2a)$, a contradiction. So $s\le\hat\rho$, and
$s^2/(4\hat\rho)\le(1+\alpha/a)r_0^2/(2a)$ gives $\vartheta^2\le8(1+\alpha/a)^2(1-\beta)^2$.
\end{proof}

In this instance the true potential increase is at most $\avg V$, of the order of $\hat\varphi(r_0)$, while
$G^1=r_0/(1-\beta)$, so \emph{every} lower bound on the potential increase loses the factor $1-\beta$ there, cut-by-cut
accounting included.

\subsection{Comparison with Ju and Lan}\label{app:julan}

\citet{julan2023infinite} work with an SAA of $N$ scenarios, their $\lambda$ is our $\beta$, and one iteration of their
methods is one forward step and one cut; their Lipschitz constants (their Lemma~2.4) are of the order of our crude
constant. With a common $\beta$-uniform constant $\Lh$ (\cref{prop:lipiter} provides one under
$(R_{m,\chi})$), $n=d$, $\Lambda:=4\Lh D/(\cb\,\erel)$ and $(1-\beta)\varepsilon=\cb\,\erel$, reading their
Theorems~2.8, 2.11 and~3.5 with the truncation length needed for accuracy $\varepsilon$ gives the following numbers of
cuts until $\varepsilon$-accuracy, up to constants and norms (Inf-EDDP and Inf-SDDP are their explorative and sampling
methods, and CE stands for continual exploration); all rows give an $\varepsilon$-tight lower bound at the root.
\begin{center}
\footnotesize
\setlength{\tabcolsep}{4pt}
\begin{tabular}{lll}
\toprule
method & cuts until $\varepsilon$-accuracy & guarantee\\
\midrule
Inf-EDDP, det. \citep[Thm.~2.8]{julan2023infinite} & $\big(\ln(2/\erel)/(1-\beta)\big)^2(\Lambda+1)^d$ & first decision\\
CE-Inf-EDDP, det. \citep[Thm.~2.11]{julan2023infinite} & $\big(4\ln(2/\erel)/(1-\beta)\big)(\Lambda+1)^d$ & first decision\\
CE-Inf-SDDP, random \citep[Thm.~3.5]{julan2023infinite} & CE-Inf-EDDP $\times\,N^{2\ln(2/\erel)/(1-\beta)}$ & first decision\\
(G), \cref{thm:geom} & $2(d+1)\Lambda^d/\big((1-\beta)^2\erel\big)$ & stationary policy\\
(D), \cref{thm:det} & (G) $\times\,e\ln((d+2)/\erel)$ & stationary policy\\
(N), \cref{thm:trunc} & as (D) & non-stationary policy\\
\bottomrule
\end{tabular}
\end{center}

\section{Notation}\label{app:notation}

\begin{center}\footnotesize
\setlength{\tabcolsep}{3pt}
\renewcommand{\arraystretch}{1.05}
\begin{tabular}{>{$}l<{$}>{\raggedright\arraybackslash}p{5.9cm}>{$}l<{$}>{\raggedright\arraybackslash}p{5.9cm}}
\toprule
X_t,d_t,D_t & state set at time $t$ (state $x_t$), its dimension and diameter & \bxi_t,\PP_t,\Xi_t & noise of stage $t$, its law, its space\\
\cX_t,\ell_t,\cb & next-state map, stage loss, initial error per stage \eqref{eq:main:initial_error} & \Bd_t,\cB_t & pointwise (given $\xi$) and expected Bellman operators\\
V_t,\lv tk & value function at time $t$, approximation after $k$ iterations & f^k_t & cut added at iteration $k$\\
\pi^k,V^{\pi^k}_0(x_0) & greedy policy of iteration $k$, its expected cost & G^k & gap $V^{\pi^k}_0(x_0)-\lv0{k-1}(x_0)$\\
r^k_t,R^k & residual, its sum along the path & \bgamma^k_t,\bgamma^{F,k}_t & cut error, forward error\\
\cF^k_{t-1} & forward operator of iteration $k$ (greedy up to $\bgamma^{F,k}_t$) & \bGamma_k,\bGamma^F_k & total cut error, total forward error\\
\Gamma,\bar{\bGamma}_k & total errors (deterministic, random) & U_t,\Ub t & ceiling, its Lipschitz envelope\\
M_t,M^V & potential bounds & \zeta_0 & ceiling excess at $x_0$\\
\dist_t,L_t & metric, Lipschitz constant & \mu_t,\kappa_t & reference measure, growth constant\\
\Phi^k,\cA_k & potential after $k$ iterations, $\sigma$-field of $k$ iterations & \Ph_t,\Delta^\pm_t & SAA law, SAA biases\\
\varphi_t,\hat\psi_t,\psi_t & lift function and its convex minorants & \rho_t,\Psi_T & $4L_tD_t$, infimal convolution\\
\beta,\cB,V & discount, stationary operator, fixed point & \Lh,\hat\rho & normalised Lipschitz constant, $4\Lh D$\\
\tau_k,H & pass length (random, fixed) & \eta,\chi & recourse margin, reachability tolerance\\
S,p(t) & period, phase $t\bmod S$ of time $t$ (\cref{sec:periodic}) & \chi_S,q_S & periodic lift function \eqref{eq:appD:periodic_chi}, period factor (\cref{thm:periodic})\\
\hat M_t,\Vh_t & SAA potential bound and value function (\cref{sec:noise}) & \Mb_U,M_U,\Mb & infinite-horizon potential bounds $\int(\bar U-\uV0)\,d\mu$, $\sup_X(\bar U-\uV0)$, $\avg_X(V-\uV0)$\\
\Gamma^\beta_k,\zeta & infinite-horizon error floor (\cref{sec:ih-master}), ceiling excess at $x_0$ & M_C,e_t & net potential bound (\cref{thm:master-ih}), forward error at time $t$\\
C_U,C_{\rm fin} & potential bounds $\sum_{t<T}\int(\Ub t-\lv t0)\,d\mu_t$ and, for finite $X_t$, $\sum_{t<T}\sum_x(U_t-\lv t0)(x)$ & \varepsilon_H,\erel & truncation tail $\beta^H\cb/(1-\beta)$, relative accuracy ($(1-\beta)\varepsilon=\cb\,\erel$)\\
\bottomrule
\end{tabular}
\captionof{table}{Main notation. Time and stage indices follow TFDP: stage $t\in[T]$ has loss $\ell_t$ and noise $\bxi_t$ (noise history $\bxi_{[t]}:=(\bxi_s)_{s\in[t]}$), the state $x_t$ is chosen at stage $t$, and $V_t$, $\cB_t$, $\lv tk$, $f^k_t$, $r^k_t$ are indexed by the time $t<T$ of the state they act on. Quantities with superscript $k$ are produced in iteration $k$, which uses the approximations of index $k-1$.}\label{tab:notation}
\end{center}

\end{document}